%% file: main.tex
\documentclass[12pt]{article}
\usepackage[utf8]{inputenc}
\usepackage[english]{babel}
\usepackage{t1enc}
\usepackage{amsmath}
\usepackage{amsthm}
\usepackage{amsfonts}  
\usepackage{amssymb}
\usepackage{color,graphicx,transparent}
\usepackage{hyperref}
\usepackage{todonotes}

\allowdisplaybreaks[1]

\newcommand{\bit}{\begin{itemize}}
\newcommand{\eit}{\end{itemize}}
\newcommand{\real}{\mathbb{R}}
\newcommand{\N}{\mathbb{N}}
\newcommand{\Z}{\mathbb{Z}}

\newtheorem{theorem}{Theorem}[section]
\newtheorem{definition}[theorem]{Definition}

\newtheorem{proposition}[theorem]{Proposition}
\newtheorem{remark}[theorem]{Remark}

\title{Sobolev meets Besov: \\
Regularity for the Poisson equation with  Dirichlet, Neumann and mixed boundary values }
\date{\today}

\author
{Cornelia Schneider\footnote{Friedrich-Alexander University Erlangen-Nuremberg, Applied Mathematics III, Cauerstr. 11, 91058 Erlangen, Germany. Email: \href{mailto:schneider@math.fau.de}{schneider@math.fau.de}}\ \thanks{The work of this author has been supported by Deutsche Forschungsgemeinschaft (DFG), Grant No. SCHN 1509/1-2.}
\qquad 
Flóra O. Szemenyei\footnote{\emph{Corresponding author}. Friedrich-Alexander University Erlangen-Nuremberg, Applied Mathematics III, Cauerstr. 11, 91058 Erlangen, Germany. Email: \href{mailto:szemenyei@math.fau.de}{szemenyeif@math.fau.de}\vspace{0.2cm}}
}

\begin{document}
\maketitle

\footnotetext{{\em Math Subject Classifications. Primary:}   35B65, 46E35.   {\em Secondary:} 35J05, 35J25, 42C40, 65M12. }
%\footnotetext{\textit{Math Subject Classifications.} 35J06, 46E06, 47B06.}


\begin{abstract}

We study the regularity of solutions of the Poisson equation with  Dirichlet, Neumann and mixed boundary values in polyhedral cones $K\subset \real^3$ in the  specific scale $\ B^{\alpha}_{\tau,\tau}, \ \frac{1}{\tau}=\frac{\alpha}{d}+\frac{1}{p}\ $ of Besov spaces. The regularity of the solution in these  spaces determines the order of approximation that can be achieved by adaptive and nonlinear numerical schemes.  We aim for a thorough discussion of homogeneous and inhomogeneous boundary data in all settings studied and show that the solutions are much smoother in this specific Besov scale compared to the  %ususal $L_2$-Sobolev scale
fractional Sobolev scale $H^s$ in all cases,  which justifies the use of adaptive schemes.  
%Our findings considerably  extend  \cite[Thms.~2.1,~3.1]{HS1}. 
%The proofs are performed 
%by combining  Sobolev estimates with mixed weights -- which measure the distance to the vertex and the edges of the cone, respectively --  with characterizations of Besov spaces by wavelet expansions. 

%We introduce weighted Sobolev spaces on polyhedral cones $K$ involving a mixture of weights, which measure the distance to the vertex and the edges of the cone, respectively. We provide an extension between these spaces and Besov spaces defined via wavelets. Using that result we investigate Besov regularity for the Poisson equation with Dirichlet, Neumann and mixed boundary value conditions on $K$.
\end{abstract}

%\tableofcontents

\section{Introduction}

In this article we investigate  the regularity of  solutions of 
 the Poisson equation in (bounded) polyhedral cones $K \subset \real^3$ with different boundary conditions, i.e.,  we consider the model problem 
\begin{equation} \label{poisson-gen}
-\Delta u\ =\ f     \  \text{ in }\  K , \qquad  u=g_j \ \text{ on }\  \Gamma_j, \ j\in J_0, \qquad \frac{\partial u}{\partial \nu}=g_j \   \text{ on }\  \Gamma_j, \ j\in J_1, 
\end{equation}
where $\Gamma_j$, $j=1,\ldots,  n$ denote the faces of the cone with  inhomogeneous   Dirichlet ($j\in J_0$) and/or Neumann boundary conditions ($J\in J_1$). 
We are interested in regularity properties  of the unknown solution $u$ in specific nonstandard smoothness spaces, i.e.,  the so-called  {adaptivity scale of Besov spaces} 
\begin{equation}\label{adaptivityscale}
B^{\alpha}_{\tau,\tau}(\Omega), \qquad \frac{1}{\tau}=\frac{\alpha}{d}+\frac 1p, \qquad \alpha>0,
\end{equation} 
where $\alpha>0$ stands for the smoothness of the solution and $\tau$ displays its integrability, cf. \cite{DDV97}.  
%The smoothness in these scales of spaces is related with  the achievable convergence order of adaptive algorithms. We briefly explain the concept. 
%In an adaptive strategy, the choice of the underlying degrees of freedom is not a priori fixed but depends on the shape of the unknown solution. In particular, additional degrees of freedom are only spent in regions where the numerical approximation is still 'far away' from the exact solution. \\
% 
%  
% Special attention is paid to the spatial regularity of the
%solutions to  (\ref{parab-1a-i})  in specific
%non-standard  smoothness spaces, i.e., in the so-called {\em adaptivity
%scale {of Besov spaces}}
%\begin{equation} \label{adaptivityscale}
%B^r_{\tau,\tau}(D), \quad \frac{1}{\tau}=\frac{r}{3}+\frac{1}{p}, \quad r>0.
%\end{equation}
Our investigations are motivated by  fundamental questions arising in the context of the numerical treatment of PDEs. In particular, we aim at justifying the use of adaptive numerical methods for solving such boundary value problems. 
Let us explain these relationships in more detail:   In an adaptive strategy, the choice of the underlying degrees of freedom is not a priori fixed but depends on the shape of the unknown solution. In particular, additional degrees of freedom are only spent in regions where the numerical approximation is still 'far away' from the exact solution.
Although the basic idea is convincing,  adaptive algorithms are hard to implement, so that beforehand a rigorous mathematical analysis to justify their use  is highly desirable. 

%However, in the realm of adaptivity one is always faced with the following question: does adaptivity really pay for the problem under consideration, i.e., does our favorite adaptive scheme really provide a substantial gain of effiency compared to more conventional nonadaptive schemes which are usually much easier to implement?

For adaptive wavelet methods the best one can expect is an optimal performance in the sense that it realizes the convergence rate of best $N$-term approximation schemes, which serves as a benchmark in this context. If we  denote by  $\sigma_N\bigl(u;X\bigr)$  the error of best $N$-term approximation in a Banach space $X$  (i.e., the error when  we consider the best approximation of a given function by linear combinations of some  basis functions   of $X$ consisting of at most $N$ terms), then 
%Given an adaptive algorithm based on a dictionary for the solution spaces of the PDE, the best one can expect is an optimal performance in the sense that it realizes the convergence rate of best $N$-term approximation schemes, which serves as a benchmark in this context.  Given a dictionary
% $\Psi =\{\psi_{\lambda}\}_{\lambda \in \Lambda}$  of functions in a Banach space $X$,
%the error of best $N$-term approximation is defined as
%\begin{equation}\label{error-n-term}
%\sigma_N\bigl(u;X\bigr)=\inf_{\Gamma\subset\Lambda:\#\Gamma\leq
%N}\inf_{c_\lambda}
%                \biggl\|u-\sum_{\lambda \in\Gamma}c_{\lambda}\psi_{\lambda}\big|X\biggr\|\, , 
%\end{equation}
%i.e., as the name suggests we consider the best approximation by linear combinations of the basis functions consisting of at most $N$ terms. 
% In particular, 
\cite[Thm. 11, p.~586]{DNS2} implies for $\tau<p$, 
\begin{equation}\label{error}
    \sigma_N\bigl(u;L_p(\Omega)\bigr)\leq C\,N^{-\alpha/d}\|u|B^{\alpha}_{\tau,\tau}(\Omega)\|, \qquad  \frac{1}{\tau}<\frac{\alpha}{d}+\frac 1p,  
\end{equation}
where $d$ denotes the dimension of the underlying domain $\Omega$. 
Quite recently, it has turned out that the same interrelations also hold for the very important and  widespread adaptive finite
 element schemes, cf.   \cite[Thm. 2.2]{GM09}. % gives direct estimates,
%\begin{equation*}
%    \sigma_N^{FE}\bigl(u;L_p(D)\bigr)\leq C\,N^{-s/d}\|u|B^s_{\tau,\tau}(D)\|\,,  
%\end{equation*}
% where $\sigma_N^{FE}$ denotes the counterpart to the quantity $\sigma_N(u;X)$, which corresponds to wavelet approximations. 
Hence, it can be seen from \eqref{error} that 
the achievable order of adaptive algorithms depends on the regularity of the target function in the specific scale of Besov spaces \eqref{adaptivityscale}. 

On the other hand, it is the  regularity of the solution in the scale of (fractional) Sobolev spaces $H^s(\Omega)$, where now $s$ indicates the smoothness of $u$, which encodes information on the convergence order for  nonadaptive (uniform) methods.  
Therefore, the use of adaptivity is justified if the  smoothness within the scale of Besov spaces of the exact solution of a given PDE is high enough compared to the classical Sobolev smoothness. \\

At this point, the shape of the domain comes into play. As the classical model problem of elliptic  equations, let us discuss the Poisson equation. It is well--known
that  if the domain under consideration, the right--hand side and the coefficients are sufficiently smooth, then the problem is completely regular, cf. \cite{ADN59}, meaning that  if the right-hand side is in $H^s(\Omega)$, $s\geq 1$, then the solution is contained in $H^{s+2}(\Omega)$ and there
is no reason why the  Besov smoothness should be higher. 
However, on nonsmooth (e.g. polyhedral) domains, the situation changes completely. On  these domains, singularities at the boundary may occur that diminish the Sobolev regularity of the solution significantly \cite{CW20,Cost19, Gris, Gris11,JK95} (but can be compensated by suitable weight functions).
Moreover, recent studies show that these boundary singularities do not influence the Besov regularity too much \cite{DDD, DDV97}, so that for certain nonsmooth domains the use of 
adaptive algorithms is completely justified. Moreover, for the specific case of polygonal domains in $\real^2$, even more can be said, since then the exact solution can be decomposed into a regular and a singular part corresponding to the reentrant corners \cite{Gris}. In this case the Besov smoothness depends on the smoothness of the right-hand side only, i.e., for arbitrary smooth right-hand sides, in principle one gets arbitrary high order of convergence, cf. \cite{Dah99b}. In the polyhedral case, a decomposition of the solution into a singular and a regular part is also possible, however, the situation is much more complicated since edge and vertex singularities occur, cf.  \cite{HS1, Gris}. We contribute to these studies by extending and improving the above mentioned forerunners. For this we  use  Sobolev spaces with mixed weights from Maz'ya, Rossmann  \cite{MR10}, sometimes denoted as $V$-spaces in the sequel. In particular, spaces with mixed weights are needed when studying stochastic PDEs as is demonstrated in  \cite{CLK19, Cio20}. 
To be precise, the 
 weighted Sobolev spaces  $V_{\beta,\delta}^{l,p}(K)$ with parameters  $l\in \N_0$, $1\leq p<\infty$, $\beta\in \real$, and $\delta=(\delta_1,\ldots, \delta_n)\in \real^n$ %, were  introduced and studied in detail by Maz'ya, Rossmann in \cite{MR10} and 
 we use are characterized by the norm %contain all measurable functions $u$ such that the norm 
\begin{equation}\label{V-spaces}
\|u| V_{\beta,\delta}^{l,p}(K)\|:=
\displaystyle \bigg( \int_{K}\sum_{|\alpha|\leq l} 
\rho_{0}(x)^{p(\beta-l+|\alpha|)}\prod_{k=1}^{n} \bigg( \dfrac{r_k(x)}{\rho_{0}(x)}  \bigg)^{p(\delta_k-l+|\alpha|)} |\partial^{\alpha}u(x)|^p\: dx
\bigg)^{1/p},  
\end{equation}
%is finite. 
where   $\rho_0$ denotes the  distance of a point $x$ to the vertex of the cone $K$ and $r_k(x)$ the  distances to the respective edges $M_k$ of the cone. It is precisely this mixture of the different weights involved and their interplay, which allows for an even finer description when investigating the singularities of the solutions of PDEs on polyhedral cones or even general  domains of polyhedral type.

%The main tool to prove the results in \cite{DS19, DS18}
% were regularity estimates in weighted Sobolev spaces, so-called {\em Kondratiev spaces} 
%$\calk^m_{p,a}(\Omega)$. 

Let us point out here that the above $V$-spaces are adequate when studying  pure Dirichlet problems  (i.e., $J_1=\emptyset$ in \eqref{poisson-gen}) but for the Neumann and mixed boundary value problems we need suitable adaptations for our investigations from \cite{MR10}, cf. Definition \ref{def-W-space} below.

The $V$-spaces can be seen as  refinements of the so called Kondratiev spaces that contain only one weight \cite{SMCW} %and are often used when studying related regularity  questions, cf. \cite{DDHSW, Han15}. 
%In particular,  these Kondratiev spaces are very much 
and which in turn are closely related with Besov spaces in
 the adaptivity scale 
 \eqref{adaptivityscale} in the sense that powerful embedding results exist \cite{Han15}. 
We generalize and improve this embedding  to the setting of  $V$-spaces. % making use of an extension operator we constructed in  \cite[Thm.~3.1]{SS20}% we constructed an extension operator for  these spaces spaces  on polyhedral cones based on the original ideas of Stein \cite{Stein} in several steps, i.e., we constructed a bounded and linear operator
%$\mathfrak{E}:V_{\beta,\delta}^{l,p}(K)\rightarrow V_{\beta,\delta}^{l,p}(\mathbb{R}^3)$.  
%Since this  result is  of interest on its own and the proof is rather lengthy and technical due to the mixture of weights involved, we decided to publish it  in a separate paper.  
Based on our extension operator from  \cite[Thm.~3.1]{SS20} when  $p=2$ we establish an embedding of the form 
\begin{equation}\label{emb-improved}
V_{\beta,\delta}^{l,2}(K)\cap H^s(K)\hookrightarrow B^{\alpha}_{\tau, \tau}(K), \qquad \alpha<\min \Big\{l, 3(l-|\delta|), 3s)\Big\}, \quad \frac{1}{\tau}=\frac{\alpha}{3}+\frac 12, 
\end{equation}
subject to further restrictions on the appearing parameters, cf. Theorems \ref{theorem4} and  \ref{theorem5} below. The results from  \cite[Thms.~2.1, 3.1]{HS1} may be seen as a first forerunner of our result. However, due to the lack of an extension operator the authors obtained the  restriction $\alpha<\frac 32 s$, which we now improve to $\alpha<3s$.  Furthermore, it turned out that the results for negative parameters $\delta<0$ in \cite{HS1} were not correct and required a more subtle analysis, which we perform in detail in Theorem \ref{theorem5}.

Afterwards,  applying the embedding result \eqref{emb-improved} and regularity assertions from  \cite{MR02, MR10}, we show that the regularity of the solution in weighted Sobolev spaces is sufficent to establish Besov smoothness (in the original unweighted sense) for the solutions of \eqref{poisson-gen}.  Moreover, 
 in comparision  with the $H^s$-regularity results from  \cite{Gris} we see that the Besov regularity  is much higher: since $\alpha<3s$ in \eqref{emb-improved} by factor $3$ when $l$ is large. Fortunately, the embedding \eqref{emb-improved} transfers mutatis mutandis to the weighted spaces needed for studying Neumann and mixed boundary values. Our main results are  stated in Theorems \ref{theorem7}, \ref{pdeforWspace}, and \ref{pdeforWspacemixed} and demonstrate that in all cases investigated, the use of adaptive wavelet schemes is  justified. 

%In particular,  for mixed boundary conditions our findings  are completely new and not published elsewhere so far. In total, compared  with the Sobolev regularity results from  \cite{gris} it turns out that in all cases studied the Besov regularity is much higher than the regularity of the solutions in fractional Sobolev spaces which justifies the use of adaptive algorithms!  

In a forthcomming paper we will further exploit \eqref{emb-improved} for studying similar questions for parabolic problems in the spirit of \cite{DS19}. \\

%By applying the embedding result and regularity assertions from \cite{MR02} related to weighted Sobolev spaces we can state regularity result for the Poisson equation with different boundary conditions. Our aim is to provide Besov regularity for the solution of the boundary value problem with smoothness parameter as large as it possible. With the extension operator in hand we can improve the result of Dahlke and Sickel of \cite{HS1}.
%Furthermore, for the Dirichlet problem it is sufficient to consider the spaces $V_{\beta,\delta}^{l,p}(K) $ but for the Neumann and mixed boundary value problems we need the generalization of these spaces from \cite{MR10}. Good news that without larger difficulties we can transfer the results of the spaces $V_{\beta,\delta}^{l,p}(K) $ mutatis mutandis to the generalization of these spaces.\\

   The paper is organized as follows. In Section 2 we introduce Sobolev spaces with mixed weights on polyhedral cones and state some of  their relevant  properties needed for our later investigations. Furthermore, we provide a definition of  Besov spaces via wavelet decompositions. Afterwards, in Section 3, we study   embeddings between these scales of spaces. Finally,  in Section 4 we apply the obtained embeddings in order to establish (and compare) regularity results in fractional Sobolev and  Besov spaces for the Poisson equation with different boundary conditions.

\section{Preliminaries}

%In this paper we want to give an embedding theorem for weighted Sobolev spaces $V_{\beta,\delta}^{l,p}$ into Besov spaces. Our results generalize and improve \cite{HS1}. Corresponding counterparts for Kondratiev spaces can be found in \cite{Han15}.  \\

%In this section we define the weighted %Sobolev spaces $V_{\beta,\delta}^{l,p}$ %and give a wavelet characterization of %Besov spaces $B_{p,q}^s$.\\

\paragraph{Notation}

We start by collecting some general notation used throughout the paper. As usual,   $\N$ stands for the set of all natural numbers, $\N_0=\mathbb N\cup\{0\}$, $\Z$ denotes the integers, and 
$\real^d$, $d\in\N$, is the $d$-dimensional real Euclidean space with $|x|$, for $x\in\real^d$, denoting the Euclidean norm of $x$. \\
Let $\N_0^d$ be the set of all multi-indices, $\alpha = (\alpha_1, \ldots,\alpha_d)$ with 
$\alpha_j\in\N_0$ and $|\alpha| := \sum_{j=1}^d \alpha_j$. For partial derivatives $\partial^\alpha f=\frac{\partial^{|\alpha|}f}{\partial x^\alpha}$ we will occasionally also write $f_{x^\alpha}$. Furthermore, $B(x,r)$  is the open ball 
of radius $r >0$ centered at $x$, and for a measurable set $M\subset\real^d$ we denote by $|M|$ its Lebesgue measure.

% If $x=(x_1,\ldots,x_d)\in\Rd$ and $\alpha = (\alpha_1, \ldots,\alpha_d)\in\N_0^d$, then we put $x^\alpha := x_1^{\alpha_1} \cdots x_d^{\alpha_d}$. 
%For $a\in\R$, let   $[a]:=\max\{k\in\Z: k\leq a\}$ and $a_+:=\max (a,0)$.
%If $a,b\in\rr$,  then $a\vee b:=\max \{a,b\}$.
% $a_+:=\max(a,0)$ and let $[a]$ denote its integer part. 
%For $p\in (0,\infty]$, the number $p'$ is defined by
%$1/p':=(1-1/p)_+$ with the convention that $1/\infty=0$. 
We denote by  $c$ a generic positive constant which is independent of the main parameters, but its value may change from line to line. 
The expression $A\lesssim B$ means that $ A \leq c\,B$. If $A \lesssim
B$ and $B\lesssim A$, then we write $A \sim B$.  

%Given two quasi-Banach spaces $X$ and $Y$, we write $X\hookrightarrow Y$ if $X\subset Y$ and the natural embedding is bounded.

%Let $\mathcal{S}(\real^d)$ denote the Schwartz space and $\mathcal{S}'(\real^d)$ its topological dual. Moreover, $\mathcal{F}$ 
%stands for the Fourier-transform on $\mathcal{S}'(\real^d)$ with inverse $\mathcal{F}^{-1}$. \\
Throughout the paper 'domain' always stands for an open and connected set. The test functions on a domain $\Omega$   are denoted by  $C^{\infty}_0(\Omega)$.    
%and    $\mathcal{D}'(\Omega)$ stands for the set of distributions on $\Omega$.  
Let $L_p(\Omega)$, $1\leq p\leq \infty$, be the Lebesgue spaces on $\Omega$ as usual.  
We denote by ${C}(\Omega)$  the space of all bounded  continuous functions $f:\Omega\rightarrow \mathbb{R}$  
and  ${C}^k(\Omega)$, $k\in \N_0$, is the space of all functions $f\in {C}(\Omega)$ such that $\partial^{\alpha}f\in {C}(\Omega)$ for all $\alpha\in\N_0$ with $|\alpha|\leq k$, 
endowed with the norm $\sum_{|\alpha|\leq k}\sup_{x\in \Omega}|\partial^{\alpha}f(x)|$.\\
Moreover, %$\mathcal{S}(\real^d)$ denotes the Schwartz space of rapidly decreasing functions. T
the set of distributions on $\Omega$ will be denoted by $\mathcal{D}'(\Omega)$, whereas $\mathcal{S}'(\real^d)$ denotes the set of tempered distributions on $\real^d$.

\paragraph{Sobolev and fractional Sobolev spaces}
Let $m \in \N_0$ and  $1\le p \leq \infty$. Then
 $W^{m}_p(\Omega)$ denotes the standard \textcolor{black}{$L_p$-Sobolev spaces of order $m$}  on the domain $\Omega$, equipped with the norm
$$
\|\, u \, |W^m_{p}(\Omega)\| := \Big(\sum_{|\alpha|\leq m} \int_\Omega |\partial^\alpha u(x)|^p\,dx\Big)^{1/p}.$$ 
If $p=2$ we shall also write $H^m(\Omega)$ instead of $W^m_2(\Omega)$.
%By $H^m_0(\Omega)$ we denote the closure of $\mathcal{D}(\Omega)$ in $H^m(\Omega)$. 
%{The dual space $\left(H^m_0(\Omega)\right)'$ of $H^m_0(\Omega)$ is denoted by $H^{-m}(\Omega)$. }
Moreover, for $s\in \real$  we define the fractional Sobolev spaces ${H}^s(\mathbb{R}^d)$
as the collection of all $u\in \mathcal{S}'(\mathbb{R}^d)$ such that 
\begin{align*}
  \|u|{H}^s(\mathbb{R}^d)\|:= \| 
  \mathcal{F}^{-1}((1+|\xi|^2)^{s/2} \mathcal{F}u )|L_2(\mathbb{R}^d)
  \|<\infty,  
\end{align*}
where   $\mathcal{F}$ denotes the Fourier transform with inverse $\mathcal{F}^{-1}$. These spaces partially coincide with the classical Sobolev spaces for $s=m$ with $m\in \mathbb{N}_0$. 
Corresponding spaces on domains can be defined via restrictions of functions from ${H}^s(\mathbb{R}^d)$ equipped with the norm 
\begin{align*}
  \|u|{H}^s(\Omega)\|:=  \inf \left\{ \|g|{H}^s(\mathbb{R}^d)   \| :\: g\in {H}^s(\mathbb{R}^d),\: g|_{\Omega}=u     \right\}.
\end{align*}
If, additionally, $\Omega\subset\mathbb{R}^d$ is  a bounded Lipschitz domain, then for $s=m+\sigma$ with $m\in \mathbb{N}_0$ and $\sigma\in (0,1)$, an alternative norm comes from the Sobolev-Slobodeckij  spaces and is given by
\begin{align*}
    \|u|H^s(\Omega)\|= \|u|H^m(\Omega)\|+
    \displaystyle\sum_{|\alpha|=m} \left( \int_{\Omega} \int _{\Omega} \frac{|\partial^{\alpha}u(x)-\partial^{\alpha}u(y) |^2}{|x-y |^{d+2\sigma}}  \:dx\:dy \right)^{1/2}
    <\infty. 
\end{align*}
This follows from the extension operator in \cite[Thm.~1.2.10]{Gris} and the fact that the corresponding spaces defined  on $\real^d$ coincide, cf. \cite[Sect.~2.3.5]{T-F1}.

%\cosc
%Moreover, the Triebel-Lizorkin spaces  $F_{2,2}^s(\mathbb{R}^d)$ coincide with ${H}^s(\mathbb{R}^d)$, cf. \cite[Remark~1.3.3, p.~33]{Cornelia}.
%\scco 

\paragraph{Polyhedral cone}\label{polyhedralcone}
We mainly consider function spaces defined on polyhedral cones in the sequel. Let 
$$K:=\{ x \in \mathbb{R}^3: \: 0<|x|<\infty,\: x/|x|\in \Omega  \}$$
be an infinite cone in $ \mathbb{R}^3 $ with vertex at the origin. Suppose that the boundary $\partial K$ consists of the vertex $x=0$, the edges (half lines) $M_1,...,M_n$, and smooth faces $\Gamma_1,...,\Gamma_n$. Hence, $\Omega=K\cap S^2$ is a domain of polygonal type on the unit sphere $S^2$ with sides $\Gamma_k \cap S^2$. The angle at the edge $M_j$ is denoted by $\theta_j$. 
Moreover, we consider the bounded polyhedral cone $\tilde{K}$ obtained via  truncation
\[
    \label{truncatedcone}
    \tilde{K}:=K\cap B(0,r).
\]
\begin{figure}[h] 
\begin{minipage}{0.5\textwidth}             
    \def\svgwidth{160pt}    
    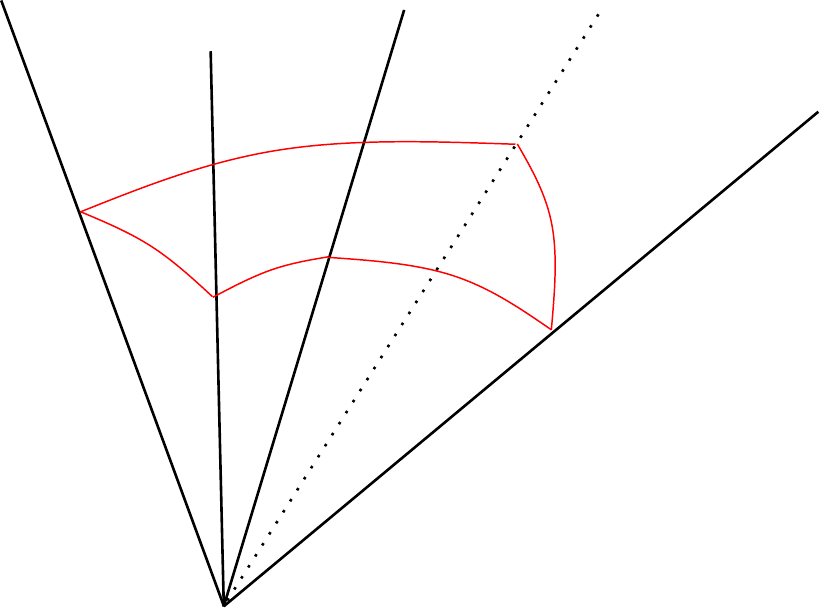  
    \caption{Infinite polyhedral cone $K$} 
    \label{fig:inf}
    \end{minipage}\hfill \begin{minipage}{0.5\textwidth}  
%\end{figure}
%\begin{figure}[ht]  
           \centering              
    \def\svgwidth{130pt}    
    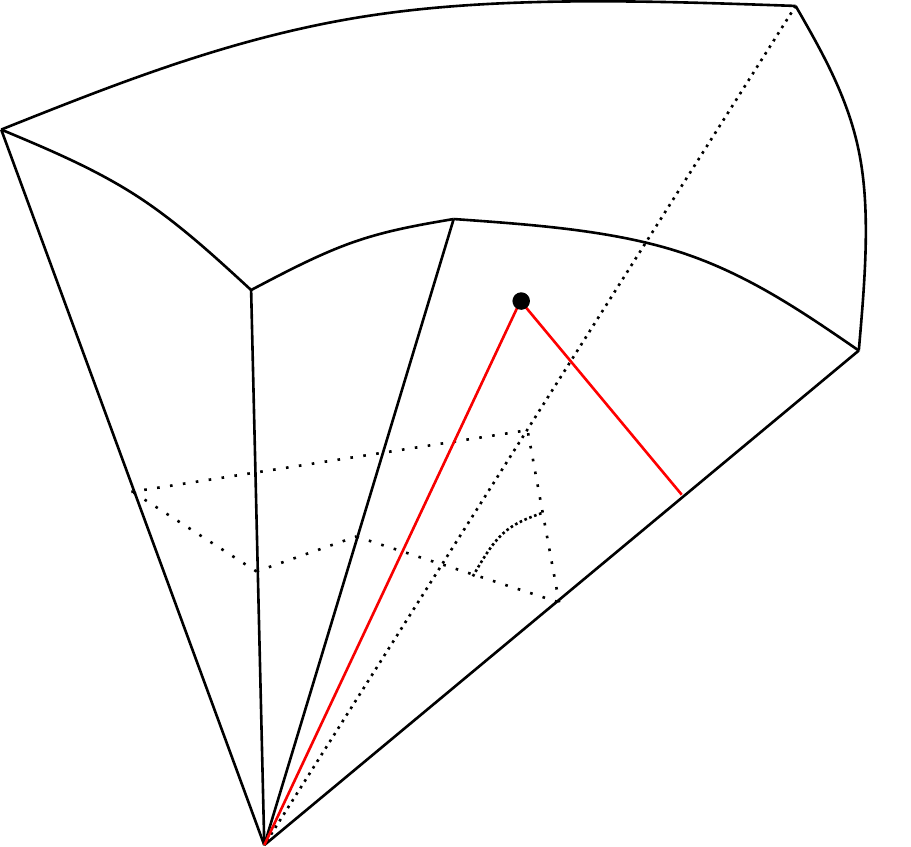
    \caption{Bounded polyhedral cone $\tilde{K}$}
    \label{fig:bound}  
    \end{minipage}\\
\end{figure}

%\todo[inline]{inner angle at $M_j$ denoted by $theta_j$; include angles in picture }

The singular points of the polyhedral cone $K$ are those $x\in \partial K$ for which for any $\varepsilon>0$ the set $\partial K \cap B(x, \varepsilon)$ is not smooth, i.e., the vertex $0$ and the edges $M_1,...,M_n$. 
When we consider the bounded cone $\tilde{K}$  we omit the non-smooth points induced by the truncation and also consider $S=\{0\}\cup M_1 \cup...\cup M_n$, which in this case is not the entire singularity set.\\
%Throughout the paper we assume that the cone $K$ has vertex at the origin.

\subsection{Weighted Sobolev spaces and their properties}
\label{Vspace}

\begin{definition}
Let $K$ be a (bounded or unbounded) polyhedral cone in $\mathbb{R}^3$ and  $S=\{0\}\cup M_1 \cup...\cup M_n$. Then the space $ V_{\beta,\delta}^{l,p}(K,S)$ is defined as the closure of the set $$C_*^{\infty}(K,S):=\{ u|_K: \: u\in C_0^{\infty}(\mathbb{R}^3\setminus S) \}$$ 
with respect to the norm 
%collection of all measurable functions which admit $l$ weak derivatives 
%such that the norm 
\begin{align}
\label{norm}
&\|u| V_{\beta,\delta}^{l,p}(K,S)\|:=
\displaystyle \bigg( \int_{K}\sum_{|\alpha|\leq l} 
\rho_{0}(x)^{p(\beta-l+|\alpha|)}\prod_{k=1}^{n} \bigg( \dfrac{r_k(x)}{\rho_{0}(x)}  \bigg)^{p(\delta_k-l+|\alpha|)} |\partial^{\alpha}u(x)|^p\: dx
\bigg)^{1/p}, 
\end{align}
where $  \beta \in \mathbb{R}$, $l\in\mathbb{N}_0$, $\delta=(\delta_1,...,\delta_n)\in \mathbb{R}^n$, $1\leq p< \infty$, $\rho_0(x)=\mathrm{dist}(0,x)$ denotes the distance to the vertex, and $r_j(x):= \mathrm{dist} (x,M_j) $ the distance to the edge $M_j$.
\end{definition}

%\todo[inline]{Check welche Eigenschaften unten tats\"achlich im Paper benoetigt werden!}

\begin{remark}
We collect some remarks and properties concerning the weighted Sobolev spaces $ V_{\beta,\delta}^{l,p}(K,S)$.

\begin{itemize}
\item In the sequel if there is no confusion we  omit $S$ from the notation and write shortly $V_{\beta,\delta}^{l,p}(K)$.
\item The space $ V_{\beta,\delta}^{l,p}(K)$ is a Banach space for $1\leq p <\infty$. The proof can be derived from a generalized result, see \cite[p.~18, Thm.~3.6]{Kuf}.

    \item $V_{\beta,\delta}^{l,p}(K) \subset L_p(K)$ for $l\geq\beta $ and $l\geq\delta_k,\: k=1,...,n$.
      
    \item We have the following embeddings which one obtains easily from the definition of the spaces $ V_{\beta,\delta}^{l,p}(K)$: 
    $$V_{\beta,\delta}^{l,p}(K)\subset V_{\beta-1,\delta-1}^{l-1,p}(K)\subset ...\subset V_{\beta-l,\delta-l}^{0,p}(K). $$

    %  \item  \textcolor{blue}{in diesem Paper brauchen wir das nicht, das haben wir in dem Beweis der Ext.Op. benutzt} 
    %  There exist regularized versions of the distance functions $\rho_0$ and $r_1$,...,$r_n$. We denote them by $\tilde{\rho}_0, \tilde{r}_1,...,\tilde{r}_n$. According to \cite[Thm.~VI.2.2, p.~171]{Stein}, these functions defined on $\overline{K}$ have nonnegative function values and are infinitely often differentiable. Moreover, there exist positive constants $A_0, B_0,C_{\alpha,0}$ such that 
    %  \begin{align*}
    %  A_0 \rho_0(x)\leq \tilde{\rho}_0(x)\leq B\rho_0(x),\quad x\in K,
    %   \end{align*}
    %and for all $\alpha\in\mathbb{N}_0^n$, 
    %\begin{align*}
    %|\partial^{\alpha}\tilde{\rho}_0(x)|\leq C_{\alpha,0}\rho_0^{1-|\alpha|}(x),\quad x\in K.
    %%\end{align*}
    %One has similar results for $\tilde{r_i},\: i=1,...,n$.
    
  %  \item \textcolor{blue}{wieder: in diesem Paper brauchen wir das nicht, das haben wir in dem Beweis der Ext.Op. benutzt} 
  %  Replacing $\rho_0,r_1,...,r_n$ by $\tilde{\rho}_0,\tilde{r}_1,...,\tilde{r}_n$ in the norm of $V_{\beta,\delta}^{l,p}(K)$, one can prove that the operator
  %  \begin{align}
  %  \label{aug12}
  %  T_{\beta',\delta'}& :V_{\beta,\delta}^{l,p}(K)  \longrightarrow V_{\beta-\beta',\delta-\delta'}^{l,p}(K),\qquad 
  %  u \mapsto \rho_0^{\beta'}(x) \prod_{k=1}^{n}\left(\frac{r_k(x)}{\rho_0(x)}   \right)^{\delta_k'}u,
  %  \end{align}
  %   is an isomorphism, where $\delta-\delta'=(\delta_1-\delta_1',...,\delta_n-\delta_n')$. The inverse %operator is given by 
  % $  ( T_{\beta',\delta'})^{-1}= T_{-\beta',-\delta'}.$ 
   \item  %\textcolor{blue}{Da du die Regularitätsaussagen veraendert hast, mit $\tilde{K}$, brauchen wir dies in diesem Paper nicht mehr}
   A function $\varphi\in C^l(K)$ is a pointwise multiplier in $V_{\beta,\delta}^{l,p}(K)$, i.e., for  
     $u\in V_{\beta,\delta}^{l,p}(K)$ we have
   \begin{align}
   \label{okt15d}
   \|\varphi u|V_{\beta,\delta}^{l,p}(K)\|\leq c \|u|V_{\beta,\delta}^{l,p}(K)\|.
   \end{align}
   \item The weighted Sobolev spaces $ V_{\beta,\delta}^{l,p}(K)$ are refinements of the Kondratiev spaces  $\mathcal{K}_{a,p}^m(K)$, which for $m\in \mathbb{N}_0$, $1\leq p<\infty $,  and $a\in \mathbb{R}$  are  defined as the collection of all measurable functions such that 
	 \begin{align*}
	 \|u|\mathcal{K}_{a,p}^m(K)\|:=\left(
	 \displaystyle\sum_{|\alpha|\leq m}\int_{K} |\rho(x)^{|\alpha|-a}\partial^{\alpha}u(x)|^p\: dx
	 \right)^{1/p}<\infty,
	 \end{align*}
	 with  weight function $
	\rho(x):=\min(1,\mathrm{dist}(x,S))$ for $x\in K$. 
   In particular,  the  scales coincide if
   $$m=l,\: \delta=(\delta_1,...,\delta_n)=(l-a,...,l-a)\quad \text{and}\quad \beta=l-a.$$ 
   %To show this, one has  to prove that the power of the weight function of $\mathcal{K}_{a,p}^m(K)$ is equivalent to the power of the weight function of $ V_{\beta,\delta}^{l,p}(K)$ with these parameter assumptions, i.e.,
   %$$\rho_{0}(x)^{\beta-l+|\alpha|}\prod_{k=1}^{n} \bigg( \dfrac{r_k(x)}{\rho_{0}(x)}  \bigg)^{\delta_k-l+|\alpha|} \sim \rho(x)^{|\alpha|-a}. $$
   This follows from \cite[p.~90, Subsection~3.1.1, (3.1.2)]{MR10}. 
  % \item \textcolor{blue}{dies brauchen wir jetzt auch nicht}
  % As an alternative approach to Definition \ref{Vspace} one could  define spaces $\tilde{V}_{\beta,\delta}^{l,p}(K)$, which contain all measurable functions such that the norm \eqref{norm} is finite, and investigate under which conditions the set $C^{\infty}_{\ast}(K,S)$ is dense in $\tilde{V}_{\beta,\delta}^{l,p}(K)$, i.e., when these spaces coincide with the spaces ${V}_{\beta,\delta}^{l,p}(K)$. This interesting question is postponed to a forthcoming paper.  
   \end{itemize}
 
\end{remark}

 %  \textcolor{blue}{die aeqivalente Norm brauchen wir jetzt nicht}
 %  Next we provide an equivalent norm in $ V_{\beta,\delta}^{l,p}(K) $. For this let 
 % $\varphi_{j}$ be infinitely differentiable functions depending only on $\rho_0(x)=|x|$ such that 
 %  \begin{align}
 %  \label{zeta}
 %  \text{supp }\varphi_{j}\subset \{2^{-j-1}<|x|<2^{-j+1}\},\quad |\partial^{\alpha}\varphi_{j}(\cdot)|\leq c_{\alpha} 2^{|\alpha| j},\quad \displaystyle\sum_{j=j_0}^{\infty}\varphi_{j}=1,
 %  \end{align}
 %  where $c_{\alpha}$ are constants independent of $j$. 
   
 %  \begin{remark}
 %  	If $K$ is a bounded polyhedral cone, then w.l.o.g. we assume $j_0=0$. If it is unbounded, then $j_0=-\infty$ is required.
 %  \end{remark}

%The following lemma is a direct consequence of the localization result in \cite[p.~91, Lem.~3.1.2]{MR10}.

%\begin{lemma}
%	\label{lemma2}
%	Let $K$ be a polyhedral cone,  $l\in \mathbb{N}_0$, $ \beta\in \mathbb{R}$, $\delta\in\mathbb{R}^n $, and $1\leq p<\infty$. 
	%We define a layer of the cone as
	%$$K_j=\{ x\in K: \: 2^{-j-1}<|x|\leq 2^{-j+1}   \}. $$
%	Then for all $u\in V_{\beta,\delta}^{l,p}(K)$, 
%	\begin{align*}
%	\|u|V_{\beta,\delta}^{l,p}(K)\|^p \sim \displaystyle\sum_{j=j_0}^{\infty}\|\varphi_j u|V_{\beta,\delta}^{l,p}(K)\|^p,
%	\end{align*}
%	where $(\varphi_j)_{j\geq j_0}$ is  resolution of unity satisfying \eqref{zeta} as above.
%\end{lemma}
%\begin{remark}
%Note that Lemma \ref{lemma2} actually holds for every family $(\varphi_j)_{j\geq j_0}$ as above, which satisfies the weaker assumption $\displaystyle\sum_{j=j_0}^{\infty}\varphi_j\sim 1$. 
%\end{remark}

In a preceding paper \cite[Thm.~3.1]{SS20},  we constructed an extension operator for the spaces $V_{\beta,\delta}^{l,p}(K)$. 
%This will be used later on  in Theorem \ref{theorem7} in order to improve the embedding between weighted Sobolev   and Besov spaces.  in the sense that we get a better upper bound for the smoothness in the Besov spaces.  
Our result reads as follows: 
%to obtain an improved upper bound for the Besov regularity compared to  the results from \cite[Thm. 2.1]{HS1}. 

 \begin{theorem}[\bf Extension operator]
\label{extop}

Let $K\subset \mathbb{R}^3 $ be a  polyhedral cone and let $l\in \mathbb{N}_0$, $ \beta\in \mathbb{R}$, $\delta\in\mathbb{R}^n $ and $1\leq p<\infty$. Then there exists a bounded linear extension operator
\begin{align}
    \label{dec8}
    \mathfrak{E}:V_{\beta,\delta}^{l,p}(K,S)\rightarrow V_{\beta,\delta}^{l,p}(\mathbb{R}^3,S),
\end{align}
where $S$ is the singularity set of $K$.
\end{theorem}

\begin{remark}
The norm of the space $V_{\beta,\delta}^{l,p}(\mathbb{R}^3,S) $ is defined similarly as the norm of $ V_{\beta,\delta}^{l,p}(K,S)$ replacing the integral domain $K$ by $ \mathbb{R}^3$.
\end{remark}

In order to  study inhomogeneous boundary conditions $g_j\neq 0$ later on,  we introduce the   trace spaces $V_{\beta,\delta}^{l-1/p,p}(\Gamma_j) $  from    \cite[Section~3.1.4]{MR10}.

\begin{definition}
\label{tracespaceV}
Let $l\in\mathbb{N}$, $1\leq p<\infty$, $\beta \in \mathbb{R}$, and $\delta=(\delta_1,...,\delta_n)\in \mathbb{R}^n$. We denote by $V_{\beta,\delta}^{l-1/p,p}(\Gamma_j) $ the trace space for $V_{\beta,\delta}^{l,p}(K)$ on the face $\Gamma_j$ of the cone $K$. The corresponding norm is defined as 
\begin{align*}
    \|u|V_{\beta,\delta}^{l-1/p,p}(\Gamma_j) \| =\inf \{\|v| V_{\beta,\delta}^{l,p}(K)\|: \: v\in V_{\beta,\delta}^{l,p}(K),\  v=u  \text{ on }  \Gamma_j    \}.
\end{align*}
\end{definition}

\subsection{Besov spaces via wavelets}
\label{besovchar}

We introduce Besov spaces via wavelet decompositions as in   \cite[Section~2, p.~563]{Han15}. % using the notation from Hansen \cite{Han15}. This wavelet characterization will be extremely useful when establishing embeddings between weighted Sobolev spaces into Besov spaces.\\
Consider the wavelet construction of Daubechies with the mother function $\eta$ and the scaling function $\phi$,  which satisfy the following conditions:
\begin{align*}
(i)& \quad \text{ compact support:} \quad\text{ supp}\:\eta,\phi \subset [-N,N]^d,\\
(ii)& \quad\text{ moment conditions:}\quad \displaystyle\int_{\mathbb{R}^d} x^{\alpha}\eta(x)dx=0 \: \ \text{ for }\ \: \alpha\in \mathbb{N}_0^d,\ |\alpha|\leq r,\\
(iii)& \quad\text{ sufficiently high smoothness:}\quad \eta,\phi \in C^r(\mathbb{R}^d).
\end{align*}

We set $\psi^0=\phi$, $\psi^1=\eta$, define $E$ as the set of the nontrivial vertices of $[0,1]^d$, and put
$$\psi^e(x_1,...,x_d):=\displaystyle\prod_{j=1}^d \psi^{e_j}(x_j)\: \text{ for }\: e\in E . $$
Moreover, let 
$$ \Psi':=\{\psi^e:  \: e\in E  \},$$
and consider the set of dyadic cubes,
$$ D=\{I_{j,k}:=2^{-j}k+2^{-j}[0,1]^d: \: k\in\mathbb{Z}^d, \: j\in \mathbb{Z}\}.$$
Now we put $\Lambda'=D^+\times \Psi'$, where $D^+$ is the set of the dyadic cubes with measure at most 1. Furthermore, we use the notation $D_j:=\{I\in D: \: |I|=2^{-jd}\}$. 
The functions from $\Psi'$ are rescaled in the following way:
$$ \psi_{I_{j,k}}:=\psi_{j,k}:=2^{jd/2}\psi(2^j\cdot -k)  \quad \text{with}\quad I_{j,k}\in D,\: k\in\mathbb{Z}^d, \: j\in \mathbb{Z},\: \psi \in \Psi'. $$
Then the set
$$\{ \psi_I:\: I\in D  , \psi\in \Psi' \}$$
forms an orthonormal basis in $L_2(\mathbb{R}^d)$. 
Let $Q(I)$ denote some dyadic cube of minimal size such that $ \text{supp}\:\psi_I\in Q(I)$ for every $\psi \in \Psi'$. 
 Then we can write every function $f\in L_2(\mathbb{R}^d)$ as
$$f=P_0f+\displaystyle\sum_{(I,\psi)\in \Lambda'} \langle f,\psi_I\rangle \psi_I,$$
where $P_0$ is the orthogonal projector  onto the closure of the span of the function $\Phi(x)=\phi(x_1)...\phi(x_d)$ and its integer shifts $\Phi(\cdot-k),\: k\in \mathbb{Z}^d$, in $L_2(\mathbb{R}^d)$. Moreover, $f$ can be written as 
$$f=\displaystyle\sum_{(I,\psi)\in \Lambda} \langle f,\psi_I\rangle \psi_I, \quad \text{where} \quad  \Lambda=D^+\times \Psi,\: \Psi=\Psi'\cup  \{ \Phi\}.$$

With  these considerations we can now 
define Besov spaces $B^s_{p,q}(\real^d)$ by decay properties of the wavelet coefficients, if the parameters fulfill certain conditions, as follows.

 %With these considerations we can now define Besov spaces via wavelet decompositions .
 
 \begin{definition} 
 Let  $0<p,q<\infty$ and $s>\max\left(0,d(\frac{1}{p}-1)\right)$. Choose $r\in \mathbb{N}$ such that $r > s$ and construct a wavelet Riesz basis as described above.
 Then the function $f\in L_p(\mathbb{R}^d)$ belongs to the Besov space $B_{p,q}^s(\mathbb{R}^d)$ if it  admits a decomposition of the form
 \begin{align*}
 %f=\displaystyle\sum_{(I,\psi)\in \Lambda} \langle f,\psi_I\rangle \psi_I,\quad \text{where} \quad  \Lambda=D^+\times \Psi,\: \Psi=\Psi'\cup  \{ \Phi\},
 f=P_0f+\displaystyle\sum_{(I,\psi)\in \Lambda'} \langle f,\psi_I\rangle \psi_I
 \end{align*}
(convergence in $ \mathcal{S}'(\mathbb{R}^n) $) with 
 \begin{align*}
 &\|f|B_{p,q}^s(\mathbb{R}^d)\|:=\|P_0f|L_p(\mathbb{R}^d)\|\\
 &+\displaystyle\left( 
 \sum_{j=0}^{\infty}2^{j\left(s+d\left(\frac{1}{2}-\frac{1}{p}\right)\right)q}\Bigg(
 \sum_{(I,\psi)\in D_j\times \Psi'}|\langle f,\psi_I\rangle|^p
 \Bigg)^{q/p}
   \right)^{1/q}<\infty.  
 \end{align*}
 % with the usual modification for $q=\infty$.
  \end{definition}
  
 Corresponding spaces on domains $\Omega\subset \real^d$ can be defined by  restriction via
 \begin{align*}
 B_{p,q}^s(\Omega):=\{ f\in \mathcal{D}'(\Omega):\: \exists\: g \in B_{p,q}^s(\mathbb{R}^d), \: g|_{\Omega}=f\}, % \\
 %\|f|B_{p,q}^s(\Omega)\|&:=\displaystyle\inf_{g|_{\Omega}=f}\|g|B_{p,q}^s(\mathbb{R}^d)\|.
 \end{align*}
 normed by $$\|f|B_{p,q}^s(\Omega)\|:=\displaystyle\inf_{g|_{\Omega}=f}\|g|B_{p,q}^s(\mathbb{R}^d)\|.$$

\begin{remark}
\begin{itemize}
\item We mention that according to \cite{Ry} for a bounded Lipschitz domain $\Omega$ there exists a universal extension operator $E:B_{p,q}^s(\Omega)\longrightarrow B_{p,q}^s(\mathbb{R}^d)$ satisfying 
 \begin{align*}
 \|\tilde{u}|B_{p,q}^s(\mathbb{R}^d)\|
  \lesssim \|u|B_{p,q}^s(\Omega)\|,
 \end{align*}
 where $ \tilde{u}= Eu$.
 \item Moreover, when $p=q=2$ the Besov spaces coincide with the fractional Soblev spaces, i.e., in this case we have the coincidence $B^s_{2,2}(\Omega)=H^s(\Omega)$, $s>0$. 
 \end{itemize}
 \end{remark}
 
% \cosc Wo genau brauchen wir Remark 2.9? \scco bei \eqref{dec9}

  \section{Embeddings from weighted Sobolev spaces into Besov spaces}\label{embeddings3}

In this section we establish embedding results, which show the close relation between the Sobolev spaces $V_{\beta,\delta}^{l,p}(K)$ with mixed weights and  Besov  spaces. These will be our main tools when investigating the Besov regularity of solutions to the Poisson equation with different boundary conditions in Section \ref{regularityproblems}. \\
In particular,  the proofs rely on the ideas from \cite[Thm.~1]{Han15} and \cite[Thms.~2.1,~3.1]{HS1} and are performed 
by combining  Sobolev estimates involving mixed weights   with the characterization of Besov spaces by wavelet expansions.

%we use the fact that our Besov spaces are defined by wavelet expansions, and therefore, in order to establish Besov smoothness of the solutions to the Poisson equation, their wavelet coefficients have to be estimated. This can be done by exploiting regularity estimates in weighted Sobolev spaces. 

However, due to the fact that our Sobolev spaces $V_{\beta,\delta}^{l,p}(K)$ have mixed weights now -- measuring the distance to the vertex and the edges, respectively -- the analysis becomes more delicate compared to \cite[Thm.~1]{Han15}. There the embeddings were studied in the setting of Kondratiev spaces, which only involve one weight function.  Moreover, we now invoke the extension operator from \cite{SS20}, cf. Theorem \ref{extop}, which improves the results for the upper bound of the smoothness in the Besov spaces considerably compared to \cite[Thms.~2.1,~3.1]{HS1}. \\
Finally, by a  modification of our arguments, we are also able to deal with the case of the parameter $\delta$ having negative components, which has not been considered sufficiently so far.

%  In this section we provide an embedding between Besov spaces and weighted Sobolev spaces $V_{\beta,\delta}^{l,p}(K)$ using the characterization of the Besov norm via wavelets introduced in Subsection \ref{besovchar}. In order to obtain better result for the restrictions of the regularity parameter than in the recent studies we apply the extension operator \eqref{dec8} for the spaces $V_{\beta,\delta}^{l,p}(K)$.
 
 \subsection{Embeddings with positive components of $\delta$}

 For simplicity we first deal with the case when all of the components of $\delta$ are positive. 
 
 %Afterwards we shall refine the proof for $\delta$ with negative components.

\begin{theorem}\label{theorem4}
Let $K$ be a bounded polyhedral cone in $\mathbb{R}^3$. Then we have a continuous embedding
$$V_{\beta,\delta}^{l,p}(K)\cap B_{p,p}^s(K)\hookrightarrow B_{\tau,\tau}^r(K),\qquad  r<\min\Big\{l,3(l-|\delta|), 3s\Big\}, \quad \frac{1}{\tau}=\frac{r}{3}+\frac{1}{p}$$
where $1<p<\infty$, 
%$ r<\min\{\textcolor{blue}{l},\textcolor{magenta}{3(l-|\delta|)}, \textcolor{green}{3s}\} $,  
$\delta=(\delta_1,...,\delta_n)\in \mathbb{R}^n$ with $\delta_i>0$ for all $i=1,...,n$, $l\in\mathbb{N}_0$, $  \beta \in \mathbb{R}$ and $ l>\beta$.
\end{theorem} 

\begin{proof}
 For $r=0$ the result is trivial, thus we assume $r>0$ and hence $0<\tau<p$.\\
Since $K$ is a Lipschitz domain, we can extend $u\in B_{p,p}^s(K)$ to some function $\tilde{u}=E u\in B_{p,p}^s(\mathbb{R}^3)$. In the corresponding wavelet characterization of $\tilde{u}$, the first term $P_0 \tilde{u}$ can be represented according to Subsection \ref{besovchar} as
$$P_0 \tilde{u}= \displaystyle\sum_{k\in\mathbb{Z}^3} \langle \tilde{u} ,\Phi (\cdot-k) \rangle \Phi (\cdot-k).$$ 
 Since $\Phi$ shares the same smoothness and support properties as the wavelets $\Psi_I$ for $|I|=1$, the coefficients $\langle \tilde{u} ,\Phi (\cdot-k) \rangle$ can be treated like the coefficients $\langle \tilde{u} ,\psi_I \rangle$ (moment conditions of $\psi_I$ are only relevant for $|I|<1$ in the calculations below). Thus we only have to show that
 \begin{equation}
 \label{normabecsles}
 \bigg( \displaystyle \sum_{(I,\psi)\in \Lambda} |I|^{(\frac{1}{p}-\frac{1}{2})\tau}  
 |\langle \tilde{u} ,\psi_I \rangle|^{\tau}
 \bigg)^{1/\tau} \lesssim \: \max \{\|u| V_{\beta,\delta}^{l,p}(K)\|,\: \|u| B_{p,p}^{s}(K)\|   \}.
 \end{equation}
 \textit{Step 1}
 We give a splitting of the indices of $\Lambda$ which we treat incrementally in \textit{Steps 2-4} below. Let\\
 \begin{align}
\label{lambda}
 \Lambda_j&=\{(I,\psi)\in \Lambda,\: |I|=2^{-j3}\},\notag\\
 \Lambda_{j,k}&=\{(I,\psi)\in \Lambda_j,\: k2^{-j}\leq \rho_I<(k+1)2^{-j} \},\: k\in \mathbb{N}_0,\notag\\
  \Lambda_{j,k,m}&=\{(I,\psi)\in \Lambda_{j,k},\: m2^{-j}\leq r_I<(m+1)2^{-j} \},\: m\in \mathbb{N}_0,
 \end{align}
where $ \rho_I $ and $r_I$ are defined as
\begin{align*}
\rho_I&=\mathrm{dist}(Q(I),0)\\
&=\displaystyle\inf_{x\in Q(I)}\rho_0(x)\sim \rho_0(x)\sim k2^{-j} \quad \text{for}\: \: x\in I \in \Lambda_{j,k}\quad \text{with}\:\: k\in \mathbb{N},
\end{align*}
and
\begin{align*}
r_I=\displaystyle\min_{k=1,...,n} \min_{x\in Q(I)} r_k(x).
 \end{align*}
 Here $Q(I)$ is a dyadic cube of minimal size such that $\text{supp}\: \psi_I\subset Q(I)$ for all $\psi\in \Psi'$, using the notations of Subsection \ref{besovchar}.\\
 In the proof we use the fact that for the cardinality of $\Lambda_{j,k}$ and $\Lambda_{j,k,m}$ it holds that 
 \begin{align}
 \label{kbecsles}
 |\Lambda_{j,k}|\lesssim k^2,
 \end{align}
 and
 \begin{align}
 \label{mbecsles}
 |\Lambda_{j,k,m}|\lesssim m.
 \end{align}
The estimate \eqref{kbecsles} is obtained in the following way: If we have a vertex singularity, then 
$$\text{Vol}(\Lambda_{j,k})\sim(k2^{-j})^2 2^{-j}=k^2 2^{-j3}=k^2|I|,$$
thus, $|\Lambda_{j,k}|\lesssim k^2$.\\
The estimate \eqref{mbecsles} we get from the following observation:
\begin{align*}
\text{Vol}(\Lambda_{j,k,m})\sim(m2^{-j}) 2^{-j} 2^{-j}= m 2^{-j3}= m|I|,
\end{align*}
hence $|\Lambda_{j,k,m}|\lesssim m.$\\ 
 
We proceed with the following strategy. In what follows we will give individual estimates for the following sets of wavelets. \\
In \textit{Step 2} we consider the set $\Lambda_j^0:=\displaystyle\bigcup_{k\geq 1}\bigcup_{m\geq 1}\Lambda_{j,k,m}$, i.e., all of the wavelets except the ones whose supports are close to the vertex $0$ and the edges $M_j$, but we consider also the wavelets whose supports intersect the smooth part of the surface of the cone.\\
Afterwards, in \textit{Step 3} we consider the sets $\Lambda_{j,0}$, i.e., the wavelets with support near the vertex $0$.\\
Finally, in \textit{Step 4} we deal with the set $ \Lambda\setminus \big\{\bigcup_j(\Lambda_j^0\cup \Lambda_{j,0})\big\}$, i.e., the wavelets with support near the edges $M_j$. We denote this set by $\Lambda_{\text{edges} \smallsetminus \text{vertex}}$.\\
 Since in \textit{Step 2} we consider also the wavelets with support intersecting the smooth part of the boundary of the cone (the faces), we make use of Stein's extension operator for the spaces $V_{\beta,\delta}^{l,p}(K)$ from Theorem \ref{extop}.\\ 
As a preparation for our calculations in \textit{Steps 2-4} we now estimate the coefficients $ \langle \tilde{u} ,\psi_I \rangle $. Let $I\in \Lambda_j^0$.
We recall the approximation result of Whitney which says that for every $I$ there exists a polynomial $P_I$ of degree less than $l$, such that
$$\|\tilde{u}-P_I|L_p(Q(I))\| \lesssim |Q(I)|^{l/3}|\tilde{u}|_{W_p^l(Q(I))}\sim |I|^{l/3} |\tilde{u}|_{W_p^l(Q(I))},$$
where $|\cdot|_{W_p^l}$ is the Sobolev seminorm,
$$|u|_{W_p^l(Q(I))}:=\left(\displaystyle \int_{Q(I)} |\nabla^l u(x)|^p\: dx   \right)^{1/p}.$$
Since $\psi_I$ satisfies moment conditions of order up to $l$, it is orthogonal to any polynomial of degree up to $l-1$. Thus, putting $|\delta|:=\delta_1+...+\delta_n$, we can estimate the coefficients with the help of Hölder's inequality for $1<p<\infty$,
\begin{align}
\label{kl}
|\langle \tilde{u} ,\psi_I \rangle|&=|\langle \tilde{u}-P_I ,\psi_I \rangle|\leq \|\tilde{u}-P_I|L_p(Q(I))\| \cdot \|\psi_I|L_{p'}(Q(I))\|\notag\\
&\leq c\: |I|^{l/3} |\tilde{u}|_{W_p^l(Q(I))} |I|^{\frac{1}{2}-\frac{1}{p}}\notag\\
&= c |I|^{\frac{l}{3}+\frac{1}{2}-\frac{1}{p}} \bigg(
\displaystyle\int_{Q(I)}\sum_{|\alpha|=l} |\partial^{\alpha}\tilde{u}(x)|^p\: dx
\bigg)^{1/p}\notag\\
&\leq c |I|^{\frac{l}{3}+\frac{1}{2}-\frac{1}{p}} \rho_I^{|\delta|-\beta} r_I^{-|\delta|} 
\underbrace{
\bigg(
\displaystyle\sum_{|\alpha|=l}\int_{Q(I)}
|x|^{p(\beta-|\delta|)} \prod_{k=1}^{n} r_k(x)^{p\delta_k}
 |\partial^{\alpha}\tilde{u}(x)|^p\: dx
\bigg)^{1/p}
}_{=:\mu_I}\notag\\
&= c |I|^{\frac{l}{3}+\frac{1}{2}-\frac{1}{p}} \rho_I^{|\delta|-\beta} r_I^{-|\delta|} \mu_I,
\end{align}
where we used that the cases $\beta>|\delta|$ and $\beta\leq |\delta|$ can be treated similarly by the following argument: One has for $I\in\Lambda_j^0$ that $I\in \Lambda_{j,k}$ for some $k\geq 1$. But then if $\beta>|\delta|$ we have $\rho_I^{-(\beta-|\delta|)}\lesssim (k2^{-j})^{-(\beta-|\delta|)}$, otherwise $\rho_I^{-(\beta-|\delta|)}\lesssim ((k+1)2^{-j})^{-(\beta-|\delta|)}$, because $k2^{-j}<\rho_I<(k+1)2^{-j}$. In particular, $\rho_I^{-(\beta-|\delta|)}\sim (k2^{-j})^{-(\beta-|\delta|)}$ for $I\in \Lambda_{j,k}$, $k\geq 1$.\\
\textit{Step 2: The set $ \Lambda_j^0$ }
Using \eqref{kl} and applying Hölder's inequality with $\frac{p}{\tau}>1$ one has the following estimate,
\begin{align}
\label{1a}
 &\displaystyle \sum_{(I,\psi)\in \Lambda_{j}^0} |I|^{(\frac{1}{p}-\frac{1}{2})\tau}  
 |\langle \tilde{u} ,\psi_I \rangle|^{\tau}\notag\\
 &\lesssim  \sum_{(I,\psi)\in \Lambda_{j}^0} |I|^{(\frac{1}{p}-\frac{1}{2})\tau} |I|^{(l/3+1/2-1/p)\tau} \rho_I^{(|\delta|-\beta)\tau} r_I^{-|\delta|\tau} \mu_I^{\tau}\notag\\
 &\lesssim
 \underbrace{  \bigg( 
  \sum_{(I,\psi)\in \Lambda_{j}^0} \big(
  |I|^{l \tau /3} \rho_I^{(|\delta|-\beta)\tau} r_I^{-|\delta|\tau}
  \big)^\frac{p}{p-\tau}
   \bigg)^{\frac{p-\tau}{p}} }_{=:I}
\underbrace{ \bigg(     
 \sum_{(I,\psi)\in \Lambda_{j}^0}\mu_I^p
  \bigg)^{\frac{\tau}{p}}}_{=:II}.
\end{align}
The second term II can be estimated from above using the fact that there is a controlled overlap between the cubes $Q(I)$, thus every $x \in K$ is contained in a finite number of cubes $Q(I)$. This gives
\begin{align*}
\bigg(     
 \sum_{(I,\psi)\in \Lambda_{j}^0}\mu_I^p
  \bigg)^{\frac{1}{p}}&=\bigg(\sum_{(I,\psi)\in \Lambda_{j}^0}
  \displaystyle\sum_{|\alpha|=l}\int_{Q(I)}
  |x|^{p(\beta-|\delta|)} \prod_{k=1}^{n} r_k(x)^{p\delta_k}
   |\partial^{\alpha}\tilde{u}(x)|^p\: dx
    \bigg)^{\frac{1}{p}}\\
  &\lesssim 
  \bigg(
    \displaystyle\sum_{|\alpha|=l}\int_{K}
    |x|^{p(\beta-|\delta|)} \prod_{k=1}^{n} r_k(x)^{p\delta_k}
     |\partial^{\alpha}\tilde{u}(x)|^p\: dx
      \bigg)^{\frac{1}{p}}\\
  &\leq \|u|V_{\beta,\delta}^{l,p}(K)\|.
\end{align*}
For the first term I using $|\Lambda_{j,k}|\lesssim k^2  $ and $ |\Lambda_{j,k,m}|\lesssim m  $ one obtains
\begin{align}
\label{csillag}
 I&= \bigg( 
  \sum_{(I,\psi)\in \Lambda_{j}^0} \big(
  |I|^{l \tau /3} \rho_I^{(|\delta|-\beta)\tau} r_I^{-|\delta|\tau}
  \big)^\frac{p}{p-\tau}
   \bigg)^{\frac{p-\tau}{p}}\notag \\
   &\lesssim
    \bigg( 
     \sum_{k=1}^{2^j} \sum_{m=1}^{ck}
     \sum_{(I,\psi)\in \Lambda_{j,k,m}}\big(
     |I|^{l \tau /3} \rho_I^{(|\delta|-\beta)\tau} r_I^{-|\delta|\tau}
     \big)^\frac{p}{p-\tau}
      \bigg)^{\frac{p-\tau}{p}}\notag\\
      &\lesssim
       \bigg( 
           \sum_{k=1}^{2^j} \sum_{m=1}^{ck}
           m \big(
           |I|^{l \tau /3} 
           (k2^{-j})^{(|\delta|-\beta)\tau} (m2^{-j})^{-|\delta|\tau}
           \big)^\frac{p}{p-\tau}
            \bigg)^{\frac{p-\tau}{p}}\notag\\
           &=
                   \bigg( 
                       \sum_{k=1}^{2^j} \sum_{m=1}^{ck}
                       m \big(
                       2^{-jl\tau-j(|\delta|-\beta)\tau+j|\delta|\tau}k^{(|\delta|-\beta)\tau} m^{-|\delta|\tau}
                       \big)^\frac{p}{p-\tau}
                        \bigg)^{\frac{p-\tau}{p}}\notag\\
&=        \bigg( 
                       \sum_{k=1}^{2^j} k^{(|\delta|-\beta)\tau \frac{p}{p-\tau}}2^{-j\tau(l-\beta)\frac{p}{p-\tau}}
                        \bigg(
                       \sum_{m=1}^{ck} m^{1-|\delta|\tau \frac{p}{p-\tau}}
                       \bigg)
                        \bigg)^{\frac{p-\tau}{p}}=:\circledast.     
\end{align}
Subject to the exponents of $k$ and $m$ in the sums above (i.e., distinguishing between exponents $\alpha<-1$, $\alpha=-1$, $\alpha>-1$), we have three times three cases. We first consider the cases according to $m$ using
\begin{align}
\label{becsles}
\displaystyle\sum_{m=1}^{ck}m^{\alpha} \sim \begin{cases}
     k^{\alpha+1}  &\mbox{if } \alpha>-1,\\
    1   & \mbox{if } \alpha<-1, \\
       \log (k+1) &\mbox{if } \alpha=-1,
   \end{cases}
\end{align}
 and then the subcases according to $k$.\\
\textit{Case 1} Exponent of $m>-1$:
\begin{align*}
1-|\delta|\tau \frac{p}{p-\tau}>-1\quad\iff\quad 2>|\delta|\tau \frac{p}{p-\tau}\quad\iff\quad \frac{2}{3}r>|\delta|,
\end{align*}
where we used that $\frac{1}{\tau}=\frac{r}{3}+\frac{1}{p}$, i.e., $r=3\big(\frac{p-\tau}{p\tau} \big)$.
In \eqref{csillag} we obtain
\begin{align}
\label{kell}
\circledast&\lesssim 2^{-j\tau(l-\beta)}
\bigg( 
                       \sum_{k=1}^{2^j} k^{(|\delta|-\beta)\tau \frac{p}{p-\tau}}
                        k^{1-|\delta|\tau \frac{p}{p-\tau}+1}
                        \bigg)^{\frac{p-\tau}{p}}\notag\\
&=2^{-j\tau(l-\beta)}
\bigg( 
                       \sum_{k=1}^{2^j} k^{2-\beta\tau\frac{p}{p-\tau}}
                        \bigg)^{\frac{p-\tau}{p}}\notag\\
       &=2^{-j\tau(l-\beta)}
\bigg( 
                                               \sum_{k=1}^{2^j} k^{2
                -\frac{3\beta}{r}}
                                                \bigg)^{\frac{p-\tau}{p}}.
\end{align}
\textit{Case 2} Exponent of $m<-1$:
\begin{align*}
1-|\delta|\tau \frac{p}{p-\tau}<-1 \quad\iff\quad \frac{2}{3}r<|\delta|
\end{align*}
Then,
\begin{align*}
\displaystyle \sum_{m=1}^{k} m^{1-|\delta|\tau \frac{p}{p-\tau}}\lesssim 1,
\end{align*}
which inserted in \eqref{csillag} yields
\begin{align}
\label{4}
\circledast&\lesssim 2^{-j\tau(l-\beta)}
\bigg( 
                       \sum_{k=1}^{2^j} k^{(|\delta|-\beta)\tau\frac{p}{p-\tau}}
                        \bigg)^{\frac{p-\tau}{p}}\notag\\
                        &=2^{-j\tau(l-\beta)}
                        \bigg( 
                                               \sum_{k=1}^{2^j} k^{(|\delta|-\beta)\frac{3}{r}}
                                                \bigg)^{\frac{p-\tau}{p}}.
\end{align}
\textit{Case 3} Exponent of $m=-1$:
\begin{align*}
1-|\delta|\tau \frac{p}{p-\tau}=-1 \Longleftrightarrow \frac{2}{3}r=|\delta|
\end{align*}
Then,
\begin{align*}
\circledast\lesssim 2^{-j\tau(l-\beta)}
\bigg( 
                       \sum_{k=1}^{2^j} k^{(|\delta|-\beta)\frac{3}{r}} \log(k+1)
                        \bigg)^{\frac{p-\tau}{p}}.
\end{align*}
 
We consider the subcases of \textit{Case 1}. The exponent of $k$ in \eqref{kell} can be greater, smaller and equal to $-1$, which gives us three subcases, $r>\beta,\: r<\beta,\: r=\beta$ with consideration $2-\frac{3\beta}{r}<-1\Longleftrightarrow \frac{\beta}{r}>1\Longleftrightarrow \beta>r$. 
Using the calculations
\begin{align}
\label{becslesk}
\displaystyle\sum_{k=1}^{2^j}k^{\beta} \sim \begin{cases}
     (2^j)^{\beta+1}  &\mbox{if } \beta>-1,\\
    1   & \mbox{if } \beta<-1, \\
       j+1 &\mbox{if } \beta=-1,
   \end{cases}
\end{align}
one obtains from \eqref{csillag} and \eqref{kell},
\begin{align*}
 &\bigg( 
  \sum_{(I,\psi)\in \Lambda_{j}^0} \big(
  |I|^{l \tau /3} \rho_I^{(|\delta|-\beta)\tau} r_I^{-|\delta|\tau}
  \big)^\frac{p}{p-\tau}
   \bigg)^{\frac{p-\tau}{p}} \\
   &\lesssim 2^{-j\tau(l-\beta)} \begin{cases}
    2^{j(r-\beta)\tau} &\mbox{if } r>\beta \quad\text{Case 1.1}\\
    1   & \mbox{if } r<\beta \quad\text{Case 1.2}\\
      (j+1)^{\frac{p-\tau}{p}} & \mbox{if } r=\beta \quad\text{Case 1.3}.
   \end{cases}
\end{align*}
Using a similar argument for \textit{Case 2} ($ \frac{2}{3}r<|\delta| $) we see that
\begin{align*}
 &\bigg( 
  \sum_{(I,\psi)\in \Lambda_{j}^0} \big(
  |I|^{l \tau /3} \rho_I^{(|\delta|-\beta)\tau} r_I^{-|\delta|\tau}
  \big)^\frac{p}{p-\tau}
   \bigg)^{\frac{p-\tau}{p}} \\
   &\lesssim 2^{-j\tau(l-\beta)} \begin{cases}
    2^{j(\frac{p-\tau}{p}+(|\delta|-\beta)\tau)} &\mbox{if } r>3(\beta-|\delta|) \quad\text{Case 2.1} \\
    1   & \mbox{if }  r<3(\beta-|\delta|)\quad\text{Case 2.2}\\
   (j+1)^{\frac{p-\tau}{p}} & \mbox{if } r=3(\beta-|\delta|) \quad\text{Case 2.3},
   \end{cases}
\end{align*}
where we considered again the exponent of $k$, $(|\delta|-\beta)\frac{3}{r}<-1\Longleftrightarrow r<3(\beta-|\delta|)$.\\
And for the \textit{Case 3} ($ \frac{2}{3}r=|\delta| $) we have 
\begin{align*}
 &\bigg( 
  \sum_{(I,\psi)\in \Lambda_{j}^0} \big(
  |I|^{l \tau /3} \rho_I^{(|\delta|-\beta)\tau} r_I^{-|\delta|\tau}
  \big)^\frac{p}{p-\tau}
   \bigg)^{\frac{p-\tau}{p}} \\
   &\lesssim 2^{-j\tau(l-\beta)}(j+1)^{\frac{p-\tau}{p}} \begin{cases}
    2^{j\left((|\delta|-\beta)\tau+\frac{p-\tau}{p}\right)} &\mbox{if } r>\beta \quad\text{Case 3.1} \\
    1   & \mbox{if }r<\beta \quad\text{Case 3.2}\\
      (j+1)^{\frac{p-\tau}{p}} & \mbox{if }  r=\beta \quad\text{Case 3.3}.
   \end{cases}
\end{align*}
Now we put $\Lambda^0=\displaystyle\bigcup_{j\geq 0}\Lambda_j^0$ and summarize the results from above according to the 9 cases. Summing over $j$ from $0,...,\infty$ finally yields\\
\textit{(1.1)}
\begin{align*}
 \displaystyle \sum_{(I,\psi)\in \Lambda^0} |I|^{(\frac{1}{p}-\frac{1}{2})\tau}  
 |\langle \tilde{u} ,\psi_I \rangle|^{\tau}
&\lesssim \displaystyle\sum_{j\geq 0} 2^{-j\tau(l-\beta)+j3\frac{p-\tau}{p}-j\beta\tau}\|u|V_{\beta,\delta}^{l,p}(K)\|^{\tau} \\
 &\lesssim\|u|V_{\beta,\delta}^{l,p}(K)\|^{\tau},
 \end{align*}
 if $\tau l>3\frac{p-\tau}{p}\Longleftrightarrow l>r.$\\
\textit{(1.2)}
\begin{align*}
 \displaystyle \sum_{(I,\psi)\in \Lambda^0} |I|^{(\frac{1}{p}-\frac{1}{2})\tau}  
 |\langle \tilde{u} ,\psi_I \rangle|^{\tau}&\lesssim \displaystyle\sum_{j\geq 0} 2^{-j\tau(l-\beta)}\|u|V_{\beta,\delta}^{l,p}(K)\|^{\tau} \\
 &\lesssim\|u|V_{\beta,\delta}^{l,p}(K)\|^{\tau},
 \end{align*}
 if $l-\beta>0$.
\\
\textit{(1.3)}
Summing over $j$ incorporates Case 1.3 in Case 1.2 with the following calculation:
\begin{align*}
\displaystyle\sum_{j\geq 0}2^{-j(l-\beta)\tau}(\log 2^j+1)^{\frac{p-\tau}{p
}}\lesssim \sum_{j\geq 0} 2^{-j(l-\beta-\varepsilon)\tau},
\end{align*}
because 
$$ (\log 2^j+1)^{\frac{p-\tau}{p
}} \leq (j+1)^{\frac{p-\tau}{p
}} \leq (2^j)^{\varepsilon \tau} $$
for $j\geq j_0$ and $\varepsilon>0$ small by the following argument: One has for $\varepsilon>0$, that
\begin{align*}
\displaystyle\lim_{x\rightarrow \infty} \frac{\ln x}{x^{\varepsilon}}=0.
\end{align*}
Thus, $\ln x \leq x^{\varepsilon}$ for $x\geq x_0$ large enough. Hence, 
$$\ln j \leq j^{\varepsilon} \:\: \text{for} \:\: j\geq j_0,$$
and then 
$$j\leq 2^{j\varepsilon} \:\: \text{for} \:\: j\geq j_0 \:\: \text{and} \:\: \varepsilon \:\: \text{sufficiently small.} \:\: $$
Since $j\leq j+1\leq 2j$ for all $j\geq 1$, we have
\begin{align*}
j+1&<2j  \leq 2\cdot 2^{j\varepsilon}\\
j+1 &\lesssim 2^{j\varepsilon}\\
(j+1)^{\frac{p-\tau}{p
}} & \lesssim 2^{j\varepsilon\frac{p-\tau}{p
} }=: 2^{j\varepsilon'} \:\: \text{with}\:\: \varepsilon':=\varepsilon\frac{p-\tau}{p
}.
\end{align*}
Thus, if we replace $\varepsilon$ by $\varepsilon\tau$, we have
$$(j+1)^{\frac{p-\tau}{p
}}\leq 2^{j\varepsilon \tau} .$$
And hence,
$$\displaystyle\sum_{j\geq 0} 2^{-j(l-\beta-\varepsilon)\tau}<\infty $$
if 
$$l-\beta-\varepsilon>0 \quad\iff\quad l>\beta-\varepsilon,$$
thus, $l>\beta$ is sufficient.\\
\textit{(2.1)}
\begin{align*}
 \displaystyle \sum_{(I,\psi)\in \Lambda^0} |I|^{(\frac{1}{p}-\frac{1}{2})\tau}  
 |\langle \tilde{u} ,\psi_I \rangle|^{\tau}&\lesssim \displaystyle\sum_{j\geq 0} 2^{-j\tau(l-\beta)+j\frac{p-\tau}{p}+j(|\delta|-\beta)\tau}\|u|V_{\beta,\delta}^{l,p}(K)\|^{\tau} \\
 &\lesssim\|u|V_{\beta,\delta}^{l,p}(K)\|^{\tau},
 \end{align*}
 if $\tau l-|\delta|\tau -\frac{p-\tau}{p}>0 \iff l-|\delta|>\frac{r}{3}.$
\\
\textit{(2.2)} This case gives the same condition as \textit{(1.2)}. \\
\textit{(2.3)} Summing over $j$ incorporates Case 2.3 in Case 2.2 as before.\\
\textit{(3.1)-(3.3)} With the following consideration summing over $j$ incorporates Case 3 in Case 2:
\begin{align*}
& 2^{-j\tau(l-\beta)}
\bigg( 
                       \sum_{k=1}^{2^j} k^{(|\delta|-\beta)\frac{3}{r}} \log(k+1)
                        \bigg)^{\frac{p-\tau}{p}}\\
   & \lesssim 2^{-j\tau(l-\beta)} (j+1)^{\frac{p-\tau}{p}} \bigg( 
                          \sum_{k=1}^{2^j} k^{(|\delta|-\beta)\frac{3}{r}} 
                           \bigg)^{\frac{p-\tau}{p}}\\
                           & \lesssim 2^{-j\tau(l-\beta-\varepsilon')}\bigg( 
                                                     \sum_{k=1}^{2^j} k^{(|\delta|-\beta)\frac{3}{r}} 
                                                      \bigg)^{\frac{p-\tau}{p}},
\end{align*}
where we used again that 
$$(j+1)^{\frac{p-\tau}{p}}\lesssim 2^{j\varepsilon} \:\:\text{for}\: \: j\geq j_0,$$ 
thus, we arrive at Case 2.\\
\textit{Step 3: The set $\Lambda_{j,0}$} 
We use for the estimate that the dimension of the vertex $\overline{\delta}=0$, hence $|\Lambda_{j,0} |\lesssim 2^{j\overline{\delta}}\sim 1$. We leave $\overline{\delta}$ in the calculations, because then one can understand better \textit{Step 4}, where we make the same calculations with $\overline{\delta}=1$. Now we recall that $\tilde{u}\in B_{p,p}^s(\mathbb{R}^3)$. An application of Hölder's inequality yields
\begin{align*}
\displaystyle \sum_{(I,\psi)\in \Lambda_{j,0}} |I|^{(\frac{1}{p}-\frac{1}{2})\tau}  
 |\langle \tilde{u} ,\psi_I \rangle|^{\tau}&\lesssim 2^{j\overline{\delta}\frac{p-\tau}{p}}   \bigg(  \sum_{(I,\psi)\in \Lambda_{j,0}} 2^{-j3(\frac{1}{p}-\frac{1}{2})p } |\langle \tilde{u} ,\psi_I \rangle|^p  \bigg)^{\tau/p}\\
 &=2^{j\overline{\delta}\frac{p-\tau}{p}} 2^{-js\tau}
 \bigg(  \sum_{(I,\psi)\in \Lambda_{j,0}} 2^{j(s+\frac{3}{2}-\frac{3}{p})p}  |\langle \tilde{u} ,\psi_I \rangle|^p  \bigg)^{\tau/p}.
\end{align*}
Summing over $j$ and using Hölder's inequality again one has
\begin{align}
\label{dec9}
\displaystyle\sum_{j\geq 0} \sum_{(I,\psi)\in \Lambda_{j,0}} |I|^{(\frac{1}{p}-\frac{1}{2})\tau}  
 |\langle \tilde{u} ,\psi_I \rangle|^{\tau} &\lesssim \sum_{j\geq 0} 2^{j\overline{\delta}\frac{p-\tau}{p}} 2^{-js\tau} \bigg(  \sum_{(I,\psi)\in \Lambda_{j,0}} 2^{j(s+\frac{3}{2}-\frac{3}{p})p} |\langle \tilde{u} ,\psi_I \rangle|^p  \bigg)^{\tau/p} \notag\\
 &\lesssim \bigg( \sum_{j\geq 0} 2^{-js\frac{p\tau}{p-\tau}}2^{j\overline{\delta}} \bigg)^{\frac{p-\tau}{p}} \bigg(  
 \sum_{j\geq 0} \sum_{(I,\psi)\in \Lambda_{j,0}}
 2^{j(s+\frac{3}{2}-\frac{3}{p})p} |\langle \tilde{u} ,\psi_I \rangle|^p
   \bigg)^{\tau/p} \notag\\
 &\lesssim \|\tilde{u}|B_{p,p}^s(\mathbb{R}^3)\|^{\tau} \notag\\
 &\lesssim \|u|B_{p,p}^s(K)\|^{\tau},
\end{align}
under the condition
$$0=\overline{\delta}<\frac{sp\tau}{p-\tau}\Longleftrightarrow 0<s\frac{3}{r},$$
which is always satisfied since $s>0$ and $r>0$.\\
\textit{Step 4: The set $\Lambda_{\text{edges} \smallsetminus \text{vertex}}$}  
This case is very similar to \textit{Step 3}, but now close to the edges the dimension of the singular set is one, i.e. $\overline{\delta}=1$. Therefore, $|\Lambda_{\text{edges} \smallsetminus \text{vertex}}|\lesssim 2^{j\overline{\delta}}=2^j.$ The calculation is the same as in \textit{Step 3}, the only differences are that $\overline{\delta}=1$ and the sum $ \sum_{(I,\psi)\in \Lambda_{j,0}}$ is replaced by $ \sum_{(I,\psi)\in \Lambda_{\text{edges} \smallsetminus \text{vertex}} }$. Hence, the new condition is
$$1=\tilde{\delta}	<\frac{sp \tau}{p-\tau}=\frac{3s}{r}\Longleftrightarrow r<3s. $$

In conclusion, Steps 2-4 show that under the given restrictions on the parameters we have \eqref{normabecsles} with constant independent of $u$ which completes the proof.
\end{proof}

\subsection{Embeddings with negative components of $\delta$}

By  some refinements in the proof of 
Theorem \ref{theorem4}, we can extend the previous results to parameters   $\delta$ with   negative components. 

\begin{theorem} \label{theorem5}
 Let $K$ be a bounded polyhedral cone in $\mathbb{R}^3$. Then we have a continuous embedding
$$V_{\beta,\delta}^{l,p}(K)\cap B_{p,p}^s(K)\hookrightarrow B_{\tau,\tau}^r(K),\quad \frac{1}{\tau}=\frac{r}{3}+\frac{1}{p},\quad 1<p<\infty,$$
where $\delta=(\delta_1,...,\delta_n)\in \mathbb{R}^n$ with at least one negative component, $l\in\mathbb{N}_0$, $  \beta \in \mathbb{R}$, $l>\beta$, $r<3s$ and $$  r<\min \left\{  |\delta|,\frac{3}{2}(\beta-|\delta|^+) \right\}, $$ or $$ \frac{3}{2}(\beta-|\delta|^+)<r<\min \left\{   |\delta|,\frac{3}{2}(l-|\delta|^+)\right\} , $$
or
\begin{align*}
 \max \bigg\{\frac{3}{2}(l-|\delta|^+),\frac{3}{2}(\beta-|\delta|^+), |\delta|, \frac{3}{4}\beta  \bigg\}<r<\min \bigg\{ \frac{3}{2}|\delta|^+, \frac{3}{4}l\bigg\},
\end{align*}
or
\begin{align*}
\max \bigg\{  |\delta|, \frac{3}{4}\beta  \bigg\}<r<\min \bigg\{\frac{3}{2}|\delta|^+, \frac{3}{4}l,\frac{3}{2}(\beta-|\delta|^+)  \bigg\},
\end{align*}
or
\begin{align*}
 \max \bigg\{ |\delta|, \frac{3}{2}(\beta-|\delta|^+)  \bigg\}<r<\min \bigg\{\frac{3}{2}|\delta|^+, \frac{3}{4}\beta,\frac{3}{2}(l-|\delta|^+)  \bigg\},
\end{align*}
or
\begin{align*}
|\delta| <r< \min \bigg\{\frac{3}{2}|\delta|^+, \frac{3}{4}\beta,\frac{3}{2}(\beta-|\delta|^+)  \bigg\},
\end{align*}
or
$$\quad \max \left\{ \frac{3}{2}|\delta|^+, \frac{3}{4}\beta  \right\}<r<\frac{3}{4}l, $$
or
\begin{align}
\label{sept16}
\quad \frac{3}{2}|\delta|^+<r< \frac{3}{4}\beta.
\end{align}
\end{theorem}

\begin{proof}
The proof is a refinement of the estimates from Theorem \ref{theorem4}. We only have to modify \textit{Steps 1} and \textit{2}, but \textit{Steps 3}, \textit{4} remain the same.\\
We consider a layer $L$ of the cone, where $\rho_I\sim k2^{-j}$ and denote the vertices derived from the intersection of the layer and the boundary of the cone by $E_1,...,E_n$. Moreover, we set $\mathrm{dist}(E_i,E_l):=C_{i,l}k2^{-j}$, where the constant $C_{i,l}$ is independent of $j$ and $k$. We illustrate the cone and the above defined layer $L$ in Figures \ref{fig:d} and \ref{fig:l}.

\begin{figure}[ht]
\centering
  \begin{minipage}[b]{0.4\textwidth}
  \def\svgwidth{150pt}
 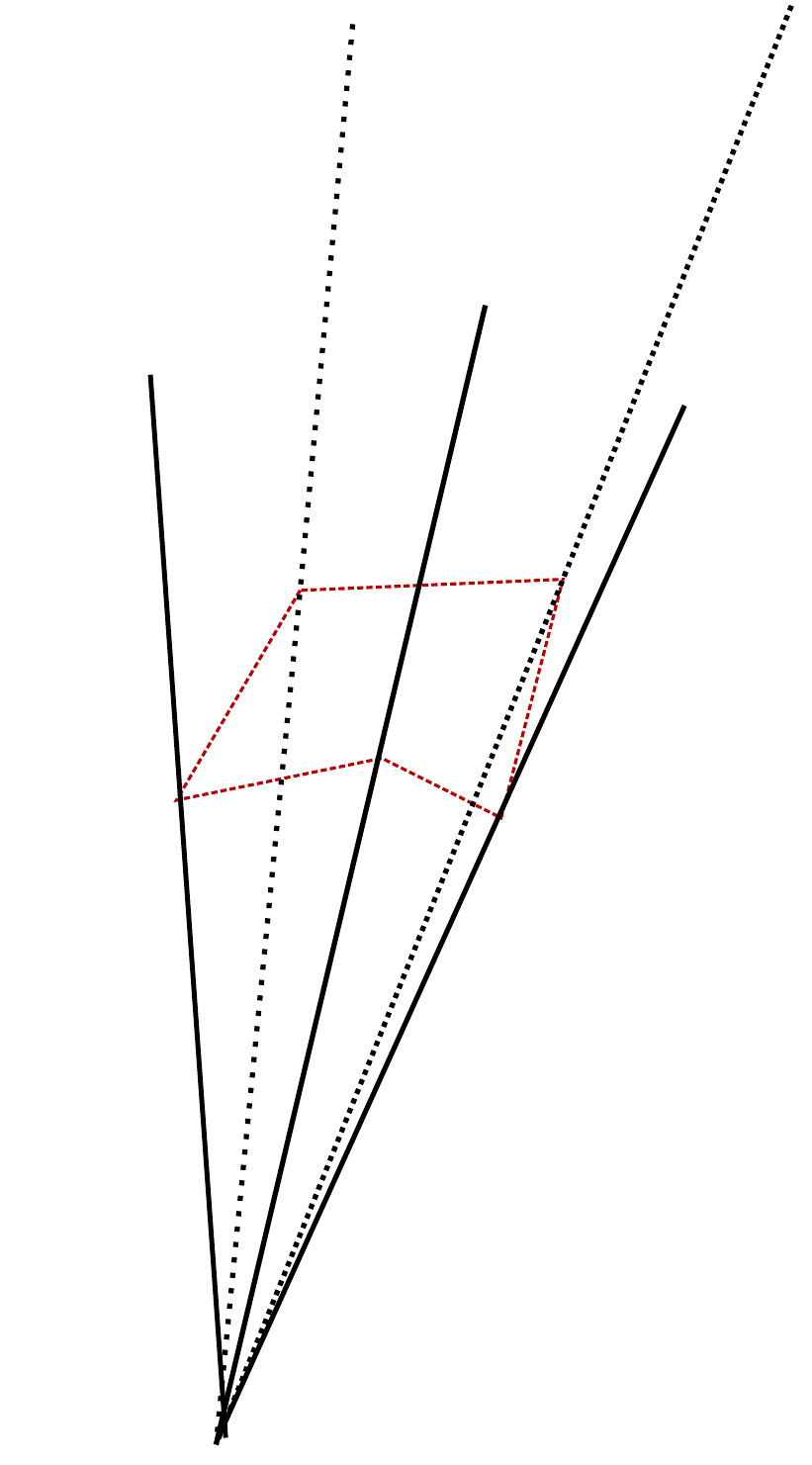
    \caption{Cone K}
    \label{fig:d}
  \end{minipage}
    \begin{minipage}[b]{0.4\textwidth}
    \def\svgwidth{150pt}
  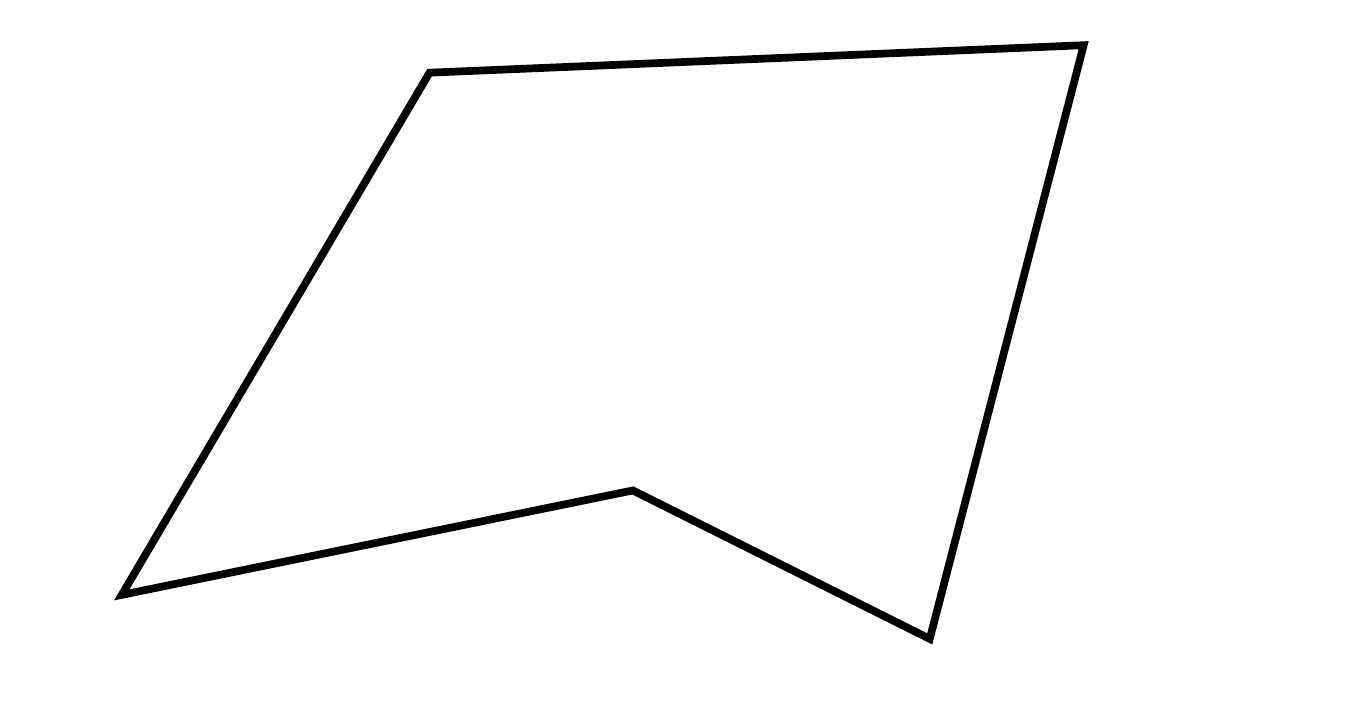
      \caption{Layer L}
      \label{fig:l}
    \end{minipage}
\end{figure}

%\begin{figure}[!ht]           
 %   \centering              
  %  \def\svgwidth{200pt}    
  %  \input{layer.pdf_tex}  
%\end{figure}
Furthermore, we consider the slice $S$ defined as  $$S:=\{x\in L: \: \mathrm{dist}(x,E_1)\in [m2^{-j},(m+1)2^{-j}]   \},$$ 
which satisfies $\text{Vol}(S)\sim m2^{-2j}$.\\
 Without loss of generality we can assume that $\delta$ has the form 
$$\delta=(\delta_1,...,\delta_n) \:\:\text{with}\:\: \delta_1,...,\delta_{L-1}\geq 0 \:\:\text{and}\:\: \delta_L,...,\delta_n<0.$$
Moreover, we put $|\delta|=\delta_1+...+\delta_n$,  $|\delta|^+:=\delta_1+...+\delta_{L-1}$, $|\delta|^-:=\delta_L+...+\delta_n$, and split the function $r_I$ into two parts depending on the sign of the components of $\delta$:
\begin{align*}
r_I^+ & :=\displaystyle\min_{j=1,...,L-1}\min_{x\in Q(I)}r_j(x),\\
r_I^- & :=\displaystyle\max_{j=L,...,n}\max_{x\in Q(I)}r_j(x),
\end{align*}
where $\text{supp}\:Q(I)\cap L \neq \emptyset$.\\
According to the notations in \eqref{lambda} we set
\begin{align*}
\Lambda_{j,k,(m_1,m_2)}:=\{I\in \Lambda_{j,k}:\:\: & 2^{-j}m_1\leq r_I^+ \leq 2^{-j}(m_1+1),\\  
& 2^{-j}m_2\leq r_I^- \leq 2^{-j}(m_2+1)
  \},
  \end{align*}
  where $m_1,m_2 \in \mathbb{N}$.\\
Concerning the cardinality of $ \Lambda_{j,k,(m_1,m_2)} $, let $A:=\displaystyle\min_{i\neq l}C_{i,l}$ and $ B:=\displaystyle\max C_{i,l} $, where $i,l\in \{ 1,...,n\}$. By the definitions one can see that $m_1\leq Bk$ and $m_2\leq Bk$. If $m_1>Bk$ or $m_2>Bk$ no suitable $I$ exists. Thus, $\max(m_1,m_2)\leq Bk.$\\
By a similar reasoning as in Theorem \ref{theorem4} we get
\begin{align}
\label{lambdakjm}
|\Lambda_{j,k,(m_1,m_2)}|\lesssim \min(m_1,m_2),
\end{align}
since the cubes $I$ under consideration now satisfy both conditions with $m_1$ and $m_2$, respectively. Moreover, the calculations from Theorem \ref{theorem4} stay correct subject to some modifications because of $r_I^+$ and  $r_I^-$. In particular, a suitable modification of formula \eqref{kl} now yields
\begin{align*}
|\langle \tilde{u} ,\psi_I \rangle|
\leq  c |I|^{\frac{l}{3}+\frac{1}{2}-\frac{1}{p}} \rho_I^{|\delta|-\beta} (r_I^+)^{-|\delta|^+} (r_I^-)^{-|\delta|^-}\mu_I.
\end{align*}
 This together with the following calculation
 \begin{align*}
 &\left(\displaystyle\sum_{I\in \Lambda_{j,k,(m_1,m_2)}} (r_I^+)^{-\tau|\delta|^+\frac{p}{p-\tau}} (r_I^-)^{-\tau|\delta|^-\frac{p}{p-\tau}}  \right)^{\frac{p-\tau}{p}}\\
& \leq 2^{-j(-\tau)|\delta|}\bigg(
\displaystyle\sum_{I\in \Lambda_{j,k,(m_1,m_2)}} m_1^{-\tau\frac{p}{p-\tau}|\delta|^+}
m_2^{-\tau\frac{p}{p-\tau}|\delta|^-}
   \bigg)^{\frac{p-\tau}{p}}\\
&\lesssim 2^{j\tau|\delta|}m_1^{-\tau |\delta|^+}m_2^{-\tau |\delta|^-}\cdot \min(m_1,m_2)^{\frac{p-\tau}{p}},
 \end{align*}
where in the last step we used \eqref{lambdakjm}, gives the following substitute of \eqref{1a} in \textit{Step 2} of Theorem \ref{theorem4}:
\begin{align*}
 &\displaystyle \sum_{(I,\psi)\in \Lambda_{j}^0} |I|^{(\frac{1}{p}-\frac{1}{2})\tau}  
 |\langle \tilde{u} ,\psi_I \rangle|^{\tau}\\&\lesssim  \sum_{(I,\psi)\in \Lambda_{j}^0} |I|^{(\frac{1}{p}-\frac{1}{2})\tau} |I|^{(l/3+1/2-1/p)\tau} \rho_I^{(|\delta|-\beta)\tau} (r_I^+)^{-\tau |\delta|^+} (r_I^-)^{-\tau |\delta|^-} \mu_I^{\tau}\\
 &\lesssim
 \underbrace{  \bigg( 
  \sum_{(I,\psi)\in \Lambda_{j}^0} \big(
  |I|^{l \tau /3} \rho_I^{(|\delta|-\beta)\tau} (r_I^+)^{-\tau |\delta|^+} (r_I^-)^{-\tau |\delta|^-}
  \big)^\frac{p}{p-\tau}
   \bigg)^{\frac{p-\tau}{p}} }_{=:I}
\underbrace{ \bigg(     
 \sum_{(I,\psi)\in \Lambda_{j}^0}\mu_I^p
  \bigg)^{\frac{\tau}{p}}}_{=:II}
\end{align*}
The estimate for term $II$ stays as in \textit{Step 2}, but for term $I$ we now have to modify the calculations as follows: 
\begin{align}
\label{6}
 &\bigg( 
  \sum_{(I,\psi)\in \Lambda_{j}^0} \big(
  |I|^{l \tau /3} \rho_I^{(|\delta|-\beta)\tau} 
  (r_I^+)^{-\tau |\delta|^+} (r_I^-)^{-\tau |\delta|^-}
  \big)^\frac{p}{p-\tau}
   \bigg)^{\frac{p-\tau}{p}} \notag\\
   &=
    \bigg( 
     \sum_{k=1}^{2^j} \sum_{m_1=1}^{Bk} \sum_{m_2=1}^{Bk}
     \sum_{(I,\psi)\in \Lambda_{j,k,(m_1,m_2)}}\big(
     |I|^{l \tau /3} \rho_I^{(|\delta|-\beta)\tau} 
     (r_I^+)^{-\tau |\delta|^+} (r_I^-)^{-\tau |\delta|^-}
     \big)^\frac{p}{p-\tau}
      \bigg)^{\frac{p-\tau}{p}}\notag\\
      &\lesssim
       \bigg( 
           \sum_{k=1}^{2^j} \sum_{m_1=1}^{Bk} \sum_{m_2=1}^{Bk}
           \min(m_1,m_2) \big(
           |I|^{l \tau /3} 
           (k2^{-j})^{(|\delta|-\beta)\tau} 
           (m_1^{-\tau |\delta|^+} m_2^{-\tau |\delta|^-}  ) 2^{j\tau |\delta|}
           \big)^{\frac{p}{p-\tau}}
            \bigg)^{\frac{p-\tau}{p}}\notag\\
           &=
                   \bigg( 
                       \sum_{k=1}^{2^j} 
                       \sum_{m_1=1}^{Bk} \sum_{m_2=1}^{Bk}
                                  \min(m_1,m_2)
                        \big(
                       2^{-jl\tau-j(|\delta|-\beta)\tau+j|\delta|\tau}k^{(|\delta|-\beta)\tau} 
                       m_1^{-\tau |\delta|^+} m_2^{-\tau |\delta|^-}
                       \big)^\frac{p}{p-\tau}
                        \bigg)^{\frac{p-\tau}{p}}\notag\\
&=   2^{-j\tau(l-\beta)}     \bigg( 
                       \sum_{k=1}^{2^j} k^{(|\delta|-\beta)\tau \frac{p}{p-\tau}}
                        \bigg(
                       \sum_{m_1=1}^{Bk} \sum_{m_2=1}^{m_1}m_1^{-\tau |\delta|^+ \frac{p}{p-\tau}}
                       m_2^{-\tau |\delta|^- \frac{p}{p-\tau}}m_2 \notag \\
                       & \qquad +\sum_{m_1=1}^{Bk} \sum_{m_2=m_1+1}^{Bk}m_1^{-\tau |\delta|^+ \frac{p}{p-\tau}}
                                              m_2^{-\tau |\delta|^- \frac{p}{p-\tau}}m_1  \bigg)
                        \bigg)^{\frac{p-\tau}{p}}\notag\\
                        & \lesssim 2^{-j\tau(l-\beta)\frac{p}{p-\tau}} 
      \bigg( 
                            \sum_{k=1}^{2^j} k^{(|\delta|-\beta)\tau \frac{p}{p-\tau}}
          \bigg[ 
        \sum_{m_1=1}^{Bk}   m_1^{-\tau |\delta|^+\frac{p}{p-\tau}} m_1^{-\tau |\delta|^-\frac{p}{p-\tau}+2} \notag\\
         & \qquad + \sum_{m_2=1}^{Bk} m_2^{-\tau |\delta|^-\frac{p}{p-\tau}}           \sum_{m_1=1}^{m_2}  m_1^{-\tau |\delta|^+\frac{p}{p-\tau}+1}
          \bigg]            
     \bigg)^{\frac{p-\tau}{p}},
\end{align}
where in the last step the estimate for the first term is valid since the exponent of $m_2$ is $>0$ and in the second term we interchanged the order of the summation.\\
We need to distinguish several cases according to the exponents of $m_1$ and $m_2$. Note that we do not deal with the case when the exponent is equal to $-1$ in order to keep our presentation as simple as possible: as in Theorem \ref{theorem4} when finally summing over $j$ this particular case can be incorporated in the other cases and gives no additional information.
\begin{align*}
&(a) \:\: \sum_{m_1=1}^{Bk}   m_1^{-\frac{p\tau}{p-\tau}|\delta|+2}\lesssim 
 \begin{cases}
  1 &\mbox{if }  -\frac{p\tau}{p-\tau}|\delta|+2<-1\Longleftrightarrow 3 \frac{p-\tau}{p\tau} <|\delta| \\
   k^{3-\frac{p\tau}{p-\tau}|\delta|}  & \mbox{if } -\frac{p\tau}{p-\tau}|\delta|+2>-1\Longleftrightarrow 3 \frac{p-\tau}{p\tau}> |\delta|
   \end{cases}\\
   &(b)\:\: \sum_{m_1=1}^{m_2}  m_1^{-\tau |\delta|^+\frac{p}{p-\tau}+1}   \lesssim 
    \begin{cases}
       1 & \mbox{if } 
       -\tau |\delta|^+\frac{p}{p-\tau}+1<-1
       \Longleftrightarrow 
       2\frac{p-\tau}{p\tau}<|\delta|^+
        \\
     m_2^{-\tau |\delta|^+\frac{p}{p-\tau}+2}  & \mbox{if } -\tau |\delta|^+\frac{p}{p-\tau}+1>-1
             \Longleftrightarrow 
             2\frac{p-\tau}{p\tau}>|\delta|^+
      \end{cases}\\
      &\text{Moreover, for the first case of (b) we further have}\\
      &(c)\:\: \sum_{m_2=1}^{Bk} m_2^{-\tau |\delta|^-\frac{p}{p-\tau}} 
      \lesssim 
          \begin{cases}
             1 &\mbox{if } 
             -\tau |\delta|^-\frac{p}{p-\tau}<-1
             \Longleftrightarrow 
             \frac{p-\tau}{p\tau}<|\delta|^-
              \\
          k^{-\frac{p\tau}{p-\tau}|\delta|^-+1}  & \mbox{if }  
          -\tau |\delta|^-\frac{p}{p-\tau}>-1                        \Longleftrightarrow \frac{p-\tau}{p\tau}>|\delta|^-.
            \end{cases}
\end{align*}
For the time being we focus on the most interesting cases. From (a)-(c) we obtain in \eqref{6} that
\begin{align*}
& 2^{-j\tau(l-\beta)}\bigg( 
                            \sum_{k=1}^{2^j} k^{(|\delta|-\beta)\tau \frac{p}{p-\tau}}
          \bigg[ 
        \sum_{m_1=1}^{Bk}   m_1^{-\tau |\delta|\frac{p}{p-\tau}+2} \\
         &\qquad + \sum_{m_2=1}^{Bk} m_2^{-\tau |\delta|^-\frac{p}{p-\tau}}  \sum_{m_1=1}^{m_2}  m_1^{-\tau |\delta|^+\frac{p}{p-\tau}+1}
          \bigg]            
     \bigg)^{\frac{p-\tau}{p}}\\
     &\lesssim  2^{-j\tau(l-\beta)}\bigg( 
                                 \sum_{k=1}^{2^j} k^{(|\delta|-\beta)\tau \frac{p}{p-\tau}}
               \big[ 1+k^{-\frac{p\tau}{p-\tau}|\delta|^-+1}
             ]            
          \bigg)^{\frac{p-\tau}{p}}
\end{align*}
for the case  $3\frac{p-\tau}{p\tau}<|\delta|$, $2\frac{p-\tau}{p\tau}<|\delta|^+$, and $\frac{p-\tau}{p\tau}>|\delta|^-$.\\
We can further estimate
\begin{align*}
& 2^{-j\tau(l-\beta)}\bigg(  \sum_{k=1}^{2^j} k^{(|\delta|-\beta)\tau \frac{p}{p-\tau}}
         \underbrace{      \big[ 1+k^{-\frac{p\tau}{p-\tau}|\delta|^-+1}
             ]  }_{\lesssim k^{-\frac{p\tau}{p-\tau}|\delta|^-+1}}                \bigg)^{\frac{p-\tau}{p}}\\
 &= 2^{-j\tau(l-\beta)} \bigg(  \sum_{k=1}^{2^j}
 k^{(|\delta|-\beta-|\delta|^-)\frac{p\tau}{p-\tau}+1}
  \bigg)^{\frac{p-\tau}{p}}\\
          &= 2^{-j\tau(l-\beta)} \bigg(  \sum_{k=1}^{2^j}
           k^{(|\delta|^+-\beta)\frac{p\tau}{p-\tau}+1}
            \bigg)^{\frac{p-\tau}{p}}\\
 &\lesssim    2^{-j\tau(l-\beta)}  
  \begin{cases}
             1
            & \mbox{if }
             (|\delta|^+ -\beta)\frac{p\tau}{p-\tau}+1<-1
              \Longleftrightarrow 
             r<\frac{3}{2}(\beta-|\delta|^+) \:  (A)
               \\
   \left( 2^{j  [ (|\delta|^+-\beta)\frac{p\tau}{p-\tau}+2]}
     \right)^{\frac{p-\tau}{p}}       & \mbox{if } 
           (|\delta|^+ -\beta)\frac{p\tau}{p-\tau}+1>-1
                       \Longleftrightarrow 
                      r>\frac{3}{2}(\beta-|\delta|^+) \:  (B).
             \end{cases}     
                                \end{align*}
Now we consider the case $(A)$
and sum the estimate over $j$ from $0,...,\infty$. This yields
\begin{align*}
\displaystyle\sum_{j\geq 0} 2^{-j\tau(l-\beta)}<\infty \quad\iff\quad    \beta<l.
\end{align*}
For case $(B)$ we obtain
\begin{align*}
& \displaystyle\sum_{j\geq 0} 2^{-j\tau(l-\beta)} 2^{j(|\delta|^+ -\beta)\tau + 2j\frac{p-\tau}{p}}\\
&= \displaystyle\sum_{j\geq 0} 2^{-j [\tau (l-|\delta|^+)-2\frac{p-\tau}{p}]},
\end{align*}
which is finite precisely, if
\begin{align*}
   l-|\delta|^+-2 \frac{p-\tau}{p\tau}>0 \quad\iff\quad \frac{3}{2} (l-|\delta|^+)>r.  
\end{align*}
We summarize the conditions which we required in the proof using the relation $\frac{1}{\tau}=\frac{r}{3}+\frac{1}{p}\Longleftrightarrow \frac{p-\tau}{p\tau}=\frac{r}{3}$:
\begin{align*}
&(i)\:  3\frac{p-\tau}{p\tau}<|\delta|\Longleftrightarrow r<|\delta|\\
&(ii)\:  2\frac{p-\tau}{p\tau}<|\delta|^+\Longleftrightarrow r<\frac{3}{2}|\delta|^+\\
 &\left(  (iii)\: \frac{p-\tau}{p\tau}>|\delta|^- \Longleftrightarrow r>3|\delta|^- \right)\\
&(iv_A)\:  r<\frac{3}{2}(\beta-|\delta|^+)\\
&(v_A)\:   \beta<l\\
&(iv_B)\:  r>\frac{3}{2}(\beta-|\delta|^+)\\
&(v_B)\:  \frac{3}{2}(l-|\delta|^+)>r.
\end{align*}
The assumption \textit{(iii)} is always satisfied because $|\delta|^-$ is negative. Thus, according to the case $ (A) $, $u\in B^r_{\tau,\tau}(\mathbb{R}^3)$ for
\begin{align*}
   r<\min \{  |\delta|,\frac{3}{2}(\beta-|\delta|^+) \} \quad \text{and}\quad  \beta<l, 
\end{align*}
since $\min (\frac{3}{2}|\delta|^+, |\delta|)=|\delta|.$\\
Moreover, according to case $ (B) $, $u\in B^r_{\tau,\tau}(\mathbb{R}^3)$ for 
\begin{align*}
    \frac{3}{2}(\beta-|\delta|^+)<r<\min \{   |\delta|,\frac{3}{2}(l-|\delta|^+)\} . 
\end{align*}
Now we consider another interesting case, where we choose the second line from $(a)$ and the first one from $(b)$, i.e., when
$$3\frac{p-\tau}{p\tau}>|\delta| \quad \iff\quad r>|\delta|$$
and
$$2\frac{p-\tau}{p\tau}<|\delta|^+\quad \iff\quad r<\frac{3}{2}|\delta|^+ .$$
 Then we obtain in \eqref{6} that
\begin{align}
\label{uj}
& 2^{-j\tau(l-\beta)}\bigg( 
                            \sum_{k=1}^{2^j} k^{(|\delta|-\beta)\tau \frac{p}{p-\tau}}
          \bigg[ 
        \sum_{m_1=1}^{Bk}   m_1^{-\tau |\delta|\frac{p}{p-\tau}+2}\notag \\
         & \qquad+ \sum_{m_2=1}^{Bk} m_2^{-\tau |\delta|^-\frac{p}{p-\tau}}  \sum_{m_1=1}^{m_2}  m_1^{-\tau |\delta|^+\frac{p}{p-\tau}+1}
          \bigg]            
     \bigg)^{\frac{p-\tau}{p}}\notag\\
     &\lesssim   2^{-j\tau(l-\beta)}\bigg( 
                                 \sum_{k=1}^{2^j} k^{(|\delta|-\beta)\tau \frac{p}{p-\tau}}
               \bigg[ 
               k^{3-\frac{p\tau}{p-\tau}|\delta|}+
               \sum_{m_2=1}^{Bk}m_2^{-\tau |\delta|^- \frac{p}{p-\tau}}
            \bigg]            
          \bigg)^{\frac{p-\tau}{p}}\notag\\
             &  \lesssim   2^{-j\tau(l-\beta)}\bigg( 
                                           \sum_{k=1}^{2^j} k^{(|\delta|-\beta)\tau \frac{p}{p-\tau}}
                         \bigg[ 
                         k^{3-\frac{p\tau}{p-\tau}|\delta|}+
                        k^{-\frac{p\tau}{p-\tau}|\delta|^-+1}
                       \bigg]            
                    \bigg)^{\frac{p-\tau}{p}}\notag\\
  & = 2^{-j\tau(l-\beta)}\bigg( 
                                              \sum_{k=1}^{2^j}
k^{-\beta\frac{p\tau}{p-\tau}+3}     +      \sum_{k=1}^{2^j} k^{(|\delta|^+ -\beta)\frac{p\tau}{p-\tau}+1}                                        \bigg)^{\frac{p-\tau}{p}}\notag\\
&\lesssim  2^{-j\tau(l-\beta)} \left[ 
\left(
\sum_{k=1}^{2^j}
k^{-\beta\frac{p\tau}{p-\tau}+3}
 \right)^{\frac{p-\tau}{p}}+
\left( 
\sum_{k=1}^{2^j} k^{(|\delta|^+ -\beta)\frac{p\tau}{p-\tau}+1}
\right)^{\frac{p-\tau}{p}}
\right]=:\star.
\end{align}
Now we have four cases according to the exponents of $k$ in \eqref{uj}. We consider the first one, when
$$-\beta\frac{p\tau}{p-\tau}+3 >-1 \quad\iff\quad \frac{3}{4}\beta<r$$
and 
$$(|\delta|^+ -\beta)\frac{p\tau}{p-\tau}+1 >-1\quad\iff\quad \frac{3}{2}(\beta-|\delta|^+)<r.$$
Then we obtain that
\begin{align*}
(i) \quad \star \lesssim 2^{-j\tau(l-\beta)} \left[
\left(
(2^j)^{-\beta\frac{p\tau}{p-\tau}+4}
\right)^{\frac{p-\tau}{p}}
+
\left(
(2^j)^{(|\delta|^+ -\beta)\frac{p\tau}{p-\tau}+2}
\right)^{\frac{p-\tau}{p}}
\right].
\end{align*}
Now we consider the second case, when 
\begin{align*}
  -\beta\frac{p\tau}{p-\tau}+3 >-1\quad\iff\quad \frac{3}{4}\beta<r
\end{align*}
and 
\begin{align*}
   (|\delta|^+ -\beta)\frac{p\tau}{p-\tau}+1 <-1\quad\iff\quad \frac{3}{2}(\beta-|\delta|^+)>r,
\end{align*}
which imply
\begin{align*}
(ii) \quad \star \lesssim 2^{-j\tau(l-\beta)} \left[
\left(
(2^j)^{-\beta\frac{p\tau}{p-\tau}+4}
\right)^{\frac{p-\tau}{p}}
+
1
\right].
\end{align*}
In the third case we have
\begin{align*}
   -\beta\frac{p\tau}{p-\tau}+3 <-1\quad\iff\quad \frac{3}{4}\beta>r 
\end{align*}
and 
\begin{align*}
    (|\delta|^+ -\beta)\frac{p\tau}{p-\tau}+1 >-1\quad\iff\quad \frac{3}{2}(\beta-|\delta|^+)<r,
\end{align*}
from which we obtain
\begin{align*}
(iii) \quad \star \lesssim 2^{-j\tau(l-\beta)} \left[
1
+
\left(
(2^j)^{(|\delta|^+ -\beta)\frac{p\tau}{p-\tau}+2}
\right)^{\frac{p-\tau}{p}}
\right].
\end{align*}
And for the fourth case we have the conditions
\begin{align*}
    -\beta\frac{p\tau}{p-\tau}+3 <-1\quad\iff\quad \frac{3}{4}\beta>r
\end{align*}
and 
\begin{align*}
    (|\delta|^+ -\beta)\frac{p\tau}{p-\tau}+1 <-1\quad\iff\quad \frac{3}{2}(\beta-|\delta|^+)>r.
\end{align*}
Then we get that
\begin{align*}
(iv) \quad \star \lesssim 2^{-j\tau(l-\beta)} .
\end{align*}
Now we sum the estimate over $j$ from $0,...,\infty$. We detail only the first case $ (i) $, the other ones are very similar.
\begin{align*}
&\displaystyle\sum_{j\geq 0}
2^{-j\tau(l-\beta)} \left[
\left(
(2^j)^{-\beta\frac{p\tau}{p-\tau}+4}
\right)^{\frac{p-\tau}{p}}
+
\left(
(2^j)^{(|\delta|^+ -\beta)\frac{p\tau}{p-\tau}+2}
\right)^{\frac{p-\tau}{p}}
\right]\\
&=\displaystyle\sum_{j\geq 0} \left(
2^{-j(\tau l-4\frac{p-\tau}{p})}+2^{-j(\tau l-|\delta|^+\tau -2\frac{p-\tau}{p})}
\right),
\end{align*}
which is finite, if
\begin{align*}
  \tau l-4\frac{p-\tau}{p}>0 \quad \iff\quad \frac{3}{4}l>r 
\end{align*}
and 
\begin{align*}
  \tau l-|\delta|^+\tau -2\frac{p-\tau}{p}>0\quad \iff\quad  \frac{3}{2}(l-|\delta|^+)>r .
\end{align*}
Now we consider the case $ (ii) $:
For the first term we have the same condition as in $ (i) $, thus
\begin{align*}
    \frac{3}{4}l>r, 
\end{align*}
and in the second term we have
\begin{align*}
   \displaystyle\sum_{j\geq 0}
2^{-j\tau(l-\beta)}\cdot 1 <\infty \quad \iff\quad  l> \beta. 
\end{align*}
In the case $(iii)$ we obtain the two conditions 
\begin{align*}
   l>\beta \quad \text{and}\quad \frac{3}{2}(l-|\delta|^+)>r.
\end{align*}
Finally, in the last case $ (iv) $ for both terms we obtain the same condition, which is
\begin{align*}
   l>\beta.
\end{align*}
We summarize the conditions which we required in these cases. The function $u$ belongs to the Besov space $ B^r_{\tau,\tau}(\mathbb{R}^3)$ for
\begin{align*}
(i)\quad  \max \bigg\{\frac{3}{2}(l-|\delta|^+),\frac{3}{2}(\beta-|\delta|^+), |\delta|, \frac{3}{4}\beta  \bigg\}<r<\min \bigg\{ \frac{3}{2}|\delta|^+, \frac{3}{4}l\bigg\},
\end{align*}
and for
\begin{align*}
(ii)\quad \max \bigg\{ |\delta|, \frac{3}{4}\beta  \bigg\}<r<\min \bigg\{\frac{3}{2}|\delta|^+, \frac{3}{4}l,\frac{3}{2}(\beta-|\delta|^+)  \bigg\},\quad l>\beta,
\end{align*}
and for
\begin{align*}
(iii)\quad \max \bigg\{ |\delta|, \frac{3}{2}(\beta-|\delta|^+)  \bigg\}<r<\min \bigg\{\frac{3}{2}|\delta|^+, \frac{3}{4}\beta,\frac{3}{2}(l-|\delta|^+)  \bigg\},\quad l>\beta,
\end{align*}
and for
\begin{align*}
(iv)\quad |\delta| <r< \min \bigg\{\frac{3}{2}|\delta|^+, \frac{3}{4}\beta,\frac{3}{2}(\beta-|\delta|^+)  \bigg\},\quad l>\beta.
\end{align*}
Finally we consider one more case, when we choose the second line from $(a)$, as well as from $ (b) $, i.e., when
\begin{align*}
   3\frac{p-\tau}{p\tau}>|\delta| \quad \iff\quad r>|\delta|  
\end{align*}
and
\begin{align*}
   \frac{p-\tau}{p\tau}>|\delta|^+\quad \iff\quad r>\frac{3}{2}|\delta|^+ , 
\end{align*}
hence,
\begin{align*}
    r>\max\left(|\delta|,\frac{3}{2}|\delta|^+ \right)=\frac{3}{2}|\delta|^+.
\end{align*}
 Then we obtain in \eqref{6} that
\begin{align}
\label{ujj}
& 2^{-j\tau(l-\beta)}\bigg( 
                            \sum_{k=1}^{2^j} k^{(|\delta|-\beta)\tau \frac{p}{p-\tau}}
          \bigg[ 
        \sum_{m_1=1}^{Bk}    m_1^{-\tau |\delta|\frac{p}{p-\tau}+2} \notag\\
         &\qquad + \sum_{m_2=1}^{Bk} m_2^{-\tau |\delta|^-\frac{p}{p-\tau}}  \sum_{m_1=1}^{m_2}  m_1^{-\tau |\delta|^+\frac{p}{p-\tau}+1}
          \bigg]            
     \bigg)^{\frac{p-\tau}{p}}\notag\\
    & \lesssim 
     2^{-j\tau(l-\beta)}\bigg( 
      \sum_{k=1}^{2^j} k^{(|\delta|-\beta)\tau \frac{p}{p-\tau}}
               \bigg[ 
     k^{3-\frac{p\tau}{p-\tau}|\delta|}+
     \sum_{m_2=1}^{Bk} \underbrace{ m_2^{-\tau|\delta|^-\frac{p}{p-\tau}}m_2^{-\tau |\delta|^+\frac{p}{p-\tau}+2}}_{= m_2^{-\tau|\delta|\frac{p}{p-\tau}+2}}
     \bigg] 
     \bigg)^{\frac{p-\tau}{p}}=:\diamond.
\end{align}
Since we have the condition $r>|\delta|$, we have in \eqref{ujj} that the exponent of $m_2$ is $>-1$. Therefore,
\begin{align*}
&\diamond \lesssim 2^{-j\tau(l-\beta)}\bigg( 
      \sum_{k=1}^{2^j} k^{(|\delta|-\beta)\tau \frac{p}{p-\tau}}
               \bigg[ 
     k^{3-\frac{p\tau}{p-\tau}|\delta|}+
     k^{-|\delta|\frac{p\tau}{p-\tau}+3}
     \bigg] 
     \bigg)^{\frac{p-\tau}{p}}\\
     &= 2^{-j\tau(l-\beta)}\bigg( 
           \sum_{k=1}^{2^j} k^{(|\delta|-\beta)\tau \frac{p}{p-\tau}}\cdot 2\cdot k^{3-\frac{p\tau}{p-\tau}|\delta|}\bigg)^{\frac{p-\tau}{p}}\\
           &\lesssim 2^{-j\tau(l-\beta)}
           \bigg( 
                      \sum_{k=1}^{2^j}
                      k^{3-\beta \frac{p\tau}{p-\tau}}
                      \bigg)^{\frac{p-\tau}{p}}\\
                      &\lesssim 2^{-j\tau(l-\beta)}
       \begin{cases}
     \left[ (2^j)^{3-\beta \frac{p\tau}{p-\tau}+1}   \right]^{\frac{p-\tau}{p}} &\mbox{if }\: 3-\beta \frac{p\tau}{p-\tau}>-1\quad \iff\quad r>\frac{3}{4}\beta \quad (C)\\
      1 &\mbox{if }  \: 3-\beta \frac{p\tau}{p-\tau}<-1 \quad \iff\quad r<\frac{3}{4}\beta\quad (D).
              \end{cases}
\end{align*}
Summing over $j$ from $0,...,\infty$, we obtain for case $(C)$,
\begin{align*}
\displaystyle\sum_{j\geq 0} 2^{-j\tau(l-\beta)} 2^{j4\frac{p-\tau}{p}-j\beta \tau}=\sum_{j\geq 0} 2^{-j(\tau l-4\frac{p-\tau}{p})},
\end{align*}
which is finite, if
\begin{align*}
   \tau l-4\frac{p-\tau}{p}>0 \quad \iff\quad  \frac{3}{4}l>r. 
\end{align*}
For the case $(D)$ we have 
\begin{align*}
    \displaystyle\sum_{j\geq 0} 2^{-j\tau(l-\beta)}<\infty \quad \iff\quad  l>\beta. 
\end{align*}
Summarizing the conditions in these cases the function $u$ belongs to the Besov space $ B^r_{\tau,\tau}(\mathbb{R}^3)$ for
\begin{align*}
    (C)\quad \max \left\{ \frac{3}{2}|\delta|^+, \frac{3}{4}\beta  \right\}<r<\frac{3}{4}l 
\end{align*}
and for 
\begin{align*}
  (D)\quad \frac{3}{2}|\delta|^+<r< \frac{3}{4}\beta,\quad l>\beta. 
\end{align*}
This completes the proof.
\end{proof}

\section{Regularity results in Sobolev and Besov spaces for the Poisson equation}
\label{regularityproblems}

In this chapter we compare regularity results in Sobolev and Besov spaces for the Poisson equation with Dirichlet, Neumann and mixed boundary conditions.

\paragraph{Fundamental problem} The general boundary value problem we want to study is the following: 
\begin{equation}\label{generalproblem}
 \begin{cases} \ 
  -\Delta u \ =\ f &\text{in}\quad  K, \\
  \quad \quad u \ =\ g_{j} &\text{on}\quad  \Gamma_j \quad\text{for}\quad j\in{J_0}, \\
   \quad  \frac{\partial}{\partial \nu}u \ =\ g_{j} \quad &\text{on}\quad  \Gamma_j \quad\text{for}\quad j\in J_1,
 \end{cases}
\end{equation}
%\begin{align}
%\label{generalproblem}
%    -\Delta u & =f \quad &\text{in}\quad  K,\notag\\
%   u &=g_{j} \quad &\text{on}\quad  \Gamma_j \quad\text{for}\quad j\in{J_0}, \notag\\
%    \frac{\partial}{\partial \nu}u &=g_{j} \quad &\text{on}\quad & \Gamma_j \quad\text{for}\quad j\in J_1,
%\end{align}
where   $J_0\cap J_1=\emptyset$ and $J_0\cup J_1=J:=\{1,...,n \}$. Moreover, $\nu$ denotes the outer normal on the faces $\Gamma_j$, $j=1,\ldots, n$ of the polyhedral cone $K$ and we denote the  whole boundary by $\Gamma:=\bigcup_{j=1}^n \Gamma_j$.  We will also need to consider the truncated cone  $\tilde{K}$  with boundary $\tilde{\Gamma}$.  
%Moreover, 
%where $\tilde{J} $ is a subset of $J=\{1,...,d \} $ and 
%The functions $f$ and $g_{j}$ have certain regularity on the polyhedral cone $K$ with edges $M_1,...,M_d$ and faces $\Gamma_1,...,\Gamma_d$. 
%If   $J_0=J$ or $J_0=\emptyset$ we recover the Dirichlet or Neumann problem, respectively. The case of  $\emptyset \neq J_0 \subsetneq J$ corresponds to  mixed boundary value problems.   These will be the three cases discussed separately in the following subsections. 
We separately discuss (and distinguish between) the following three  cases  in the sequel:  
\bit 
\item $J_0=J$:  Dirichlet problem, 
\item $J_0=\emptyset$:  Neumann problem, 
\item $\emptyset \neq J_0 \subsetneq J$: Mixed boundary values. 
\eit 

%We shall consider some special cases of \eqref{generalproblem} in the following subsections. 

Concerning the regularity in the  scale of fractional Sobolev spaces we will use results from \cite{Gris,Zan00,Ag11} in order to obtain a suitable upper bound. For the regularity results in Besov spaces we apply the embeddings from Section \ref{embeddings3}. In particular, when $p=2$ we have  $B_{2,2}^s=H^s$ for $s>0$, and the embedding may be formulated as 
$$V_{\beta,\delta}^{l,2}({K})\cap H^s(\tilde{K})\hookrightarrow B_{\tau,\tau}^r(\tilde{K}),\qquad \frac{1}{\tau}=\frac{r}{3}+\frac{1}{2},  $$
%where $K$ is some (possibly unbounded) polyhedral cone and $\tilde{K}$ its truncation,  
subject to further restrictions on the parameters, cf. Theorems  \ref{theorem4}, \ref{theorem5}. 
We see that for the Besov regularity we need to  rely on regularity results in the weighted Sobolev spaces $V^{l,p}_{\beta,\delta}(K)$ as well. However, these results are stated with the help of so-called  operator pencils, which we briefly introduce below. Further information on this subject may be found in  \cite[Sect.~7.1]{MR10} or \cite[Sect.~4.1]{MR02}.  \\

\paragraph{Operator pencils generated by the boundary value problem} 
%Our next step is to define the operator pencil of a differential operator. We use the summary of Schneider \cite[Subsection 2.5]{Cornelia}. 
There are two types of operator pencils for  problem \eqref{generalproblem}. The first one is related to the edges of the cone and the second one to its vertex.\\

{\em 1. Operator pencils $A_{j}(\lambda)$ for edges:} \  Let $M_j$ be an edge of the cone ${K}$, and let $\Gamma_{j_{\pm}}$ be the faces adjacent to $M_j$. Then by $\mathcal{D}_j$ we denote the dihedron which is bounded by the half-planes $\mathring{\Gamma}_{j_{\pm}}$ tangent to $\Gamma_{j_{\pm}}$ at $M_j$. Let $r,\varphi$ be polar coordinates in the plane perpendicular to $M_j$ such that 
\[
\mathring{\Gamma}_{j_{\pm}}=\left\{x\in \real^3: \ r>0, \ \varphi=\pm \frac{\theta_j}{2}\right\}.
\]

For $\lambda \in \mathbb{C}$ we define the operators $\widetilde{\Delta}_j(\lambda)$ and $\mathcal{B}_{j_{\pm}}(\lambda)$ on the Sobolev space $H^2(I_j)$ with  $I_j:=\left(\frac{-\theta_j}{2}, \frac{\theta_j}{2}\right)$ by 
\[
\widetilde{\Delta}_j(\lambda)u(\varphi)=r^{2-\lambda} \big(-\Delta\big)(r^{\lambda}u(\varphi)), \qquad 
\mathcal{B}_{j\pm}(\lambda)=
\begin{cases}
u(\varphi)& \text{if } j_{\pm} \in J_0, \\
r^{1-\lambda}\frac{\partial}{\partial \varphi}\Big(r^{\lambda}u(\varphi)\Big)& \text{if } j_{\pm} \in J_1. 
\end{cases}
\]
By $A_j(\lambda)$ we denote the operator 
\[
H^2(I_j)\ni u\mapsto \Big(\widetilde{\Delta}_j(\lambda){u}, \ \mathcal{B}_{j+}(\lambda)u(\varphi)\big|_{\varphi=-\frac{\theta_j}{2}}, \ 
\mathcal{B}_{j-}(\lambda)u(\varphi)\big|_{\varphi=\frac{\theta_j}{2}}\Big)\in L_2(I_j)\times \mathbb{C}\times \mathbb{C}. 
\]
%for $\lambda\in \mathbb{C}$ as follows: 
%\begin{equation}\label{op-1}
% A_j(\lambda)U=r^{2-\lambda}(-\Delta u), 
%\end{equation}
%where $u(x)=r^{\lambda}U(\varphi)$ and    $U$ is a function on $I_j:=\left(\frac{-\theta_j}{2}, \frac{\theta_j}{2}\right)$.
%This way we obtain in \eqref{op-1} a boundary value problem for the function $U$ on the $1$-dimensional subdomain $I_j$ with the complex parameter $\lambda$. Moreover, $A_j(\lambda)$  is a polynomial of degree $2m$ in $\lambda$. 
%The operator $A_j(\lambda)$ realizes a continuous mapping 
%\[
%H^{2}(I_j)\rightarrow L_2(I_j),
%\]
%for every $\lambda\in \mathbb{C}$. A complex number $\lambda_0$ is called an eigenvalue of the pencil $A_j(\lambda)$ if there exists  a nonzero function $U\in H^{2}(I_j)$ such that $A_j(\lambda_0)U=0$. 
%We denote by {$\delta_{\pm}^{(j)}$} the greatest positive real numbers such that the strip 
%\[
%-\delta_{-}^{(j)}<\mathrm{Re}\lambda<\delta_{+}^{(j)}
%\]
%is free of eigenvalues of the pencil $A_j(\lambda)$. Since we consider the Poisson equation an easy computation yields   $\delta_{\pm}^{(j)}=\frac{\pi}{\theta_j}$, cf. \cite[Ex.~2.5.2]{Cornelia}.  \\ %Furthermore, we put 
%\[
%{\delta_{\pm}^{(j)}}=\inf_{t\in [0,T]}{\delta_{\pm}^{(k)}(t)}, \qquad k=1,\ldots, n. 
%\] 
{\em 2. Operator pencil $\mathfrak{A}(\lambda)$ for the vertex:} We   introduce spherical coordinates $\rho=|x|$, $\omega=\frac{x}{|x|}$ on  {$K$} and consider the sperical cap $\Omega=K\cap S^2$.  On the space $\mathcal{H}_{\Omega}=\{u\in H^1(\Omega): \ u=0 \text{ on } \gamma_j=\Omega\cap \Gamma_j\text{ for } j\in J_0
\}$,  we consider the bilinear form 
\[
a(u,v;\lambda)=\frac{1}{\log 2} \int_{K\cap \{1<|x|<2\}} \nabla U\cdot \nabla \overline{V}~dx, \qquad u,v\in \mathcal{H}_{\Omega}, 
\]
%on $\mathcal{H}_{\Omega}\times \mathcal{H}_{\Omega}$ 
where $U(x)=\rho^{\lambda}u(\omega)$, $V(x)=\rho^{-1-\overline{\lambda}}v(\omega)$, and $\lambda \in \mathbb{C}$. Then the operator  $\mathfrak{A}(\lambda):\mathcal{H}_{\Omega}\rightarrow \mathcal{H}_{\Omega}^{\ast}$ is defined by 
\[
\big(\mathfrak{A}(\lambda)u,v\big)_{\Omega}=a(u,v;\lambda), \qquad u,v\in \mathcal{H}_{\Omega}, 
\]
where $(\cdot, \cdot)_{\Omega}$ denotes the extension of the $L_2$ scalar product to $\mathcal{H}_{\Omega}^{\ast}\times \mathcal{H}_{\Omega}$. \\

It is known that the  spectra of the  pencils $A_j$ and  $\mathfrak{A}$ consist of isolated points, the eigenvalues of these pencils, which have no accumulation points, cf. \cite[p. 292]{MR10}.  
Moreover, we denote by  {$\delta_{\pm}^{(j)}$} the 
greatest positive real numbers such that the strip 
\[
-\delta_{-}^{(j)}<\mathrm{Re}\lambda<\delta_{+}^{(j)}
\]
is free of eigenvalues of the pencil $A_j(\lambda)$ with the exception of the eigenvalue $\lambda=0$ in case of the Neumann problem.\\

For  the Poisson equation with homogeneous boundary conditions,   %computation yields   $\delta_{\pm}^{(j)}=\frac{\pi}{\theta_j}$, cf. \cite[Ex.~2.5.2]{Cornelia}.
according to \cite[p.~72]{Gris} and \cite[Ch.~2]{KMR01} the eigenvalues $\big(\lambda_{j,m}\big)_{m\in \mathbb{Z}}$ of $A_j$ are given by 
\begin{align}
    \label{lambdajkm}
    \lambda_{j,m} &\ =\ m \frac{\pi}{\theta_j} &\quad &\text{for}\   j,\ k\in J_0, \quad m\neq 0,  &(\text{Dirichlet problem})\notag\\
    \lambda_{j,m}&\ =\ (m-1) \frac{\pi}{\theta_j} &\quad &\text{for}\  j,\ k\in  J_1, & (\text{Neumann problem})\notag\\
    \lambda_{j,m}&= \left(m-\frac{1}{2}\right) \frac{\pi}{\theta_j} &\quad &\text{for}\  j\in J_0 \wedge  k \in J_1\quad \text{or}\quad   j\in J_1\wedge k\in  J_0 & (\text{Mixed problem}).
\end{align}
Therefore, we deduce that $\delta_{\pm}^{(j)}=\frac{\pi}{\theta_j}$  for the pure Dirichlet and Neumann problem,  whereas  $\delta_{\pm}^{(j)}=\frac{\pi}{2\theta_j}$ for the mixed boundary value problem. 
Moreover, in this case  the eigenvalues of $\mathfrak{A}$ are given by 
\[
\Lambda_{\pm l}=-\frac 12\pm \sqrt{\widetilde{\lambda}_l+\frac 14}, 
\]
where $\widetilde{\lambda}_l$ are the eigenvalues of the Laplace-Beltrami operator $-\Delta_{\omega}$ on $\Omega$ 
%We denote by $\Delta_{\omega}$ the Laplace-Beltrami operator on $\Omega$ with eigenvalues $\lambda_l $  and  
with corresponding orthonormalized eigenfunctions $\psi_l$ (and boundary conditions), i.e., 
\[
-\Delta_{\omega}\psi_l=\widetilde{\lambda}_l\psi_l \quad \text{on}\quad \Omega. 
\]
%It follows that the interval $[-1,0]$ is free of eigenvalues of the pencil $\mathfrak{A}$, cf. \cite[p. 154]{MR10}. 

In our regularity theorems below we will usually assume that the line $l-\beta-\frac 32$ is free of the eigenvalues of $\mathfrak{A}(\lambda)$. Since the set of eigenvalues of $\mathfrak{A}$ consists of isolated points,  this is not a severe restriction (shift $\beta\in \real$ slightly).

\subsection{Regularity for the Dirichlet problem}

First we want to study the regularity of the solution of the Poisson equation with Dirichlet boundary conditions only, i.e., $J_0=J$ in \eqref{generalproblem}, which reads as:
\begin{equation}\label{poisson}
    -\Delta u  =f \quad \text{in}\quad  K, \qquad 
    u =g_{j} \quad \text{on}\quad  \Gamma_j \quad\text{for}\quad j=1,...,n. 
\end{equation}

%i.e., $J_0=J$ in \eqref{generalproblem}.  
%Before we state our main result we need some preparations. 

%In order to study inhomogeneous boundary conditions $g_j\neq 0$ in \eqref{poisson}, we need to introduce the   trace spaces $V_{\beta,\delta}^{l-1/p,p}(\Gamma_j) $  from    \cite[Section 3.1.4]{MR10}. 

%\begin{definition}
%\label{tracespaceV}
%Let $l\in\mathbb{N}$, $1\leq p<\infty$, $\beta \in \mathbb{R}$, and $\delta=(\delta_1,...,\delta_d)\in \mathbb{R}^d$. We denote by $V_{\beta,\delta}^{l-1/p,p}(\Gamma_j) $ the trace space for $V_{\beta,\delta}^{l,p}(K)$ on the face $\Gamma_j$ of the cone $K$. The corresponding norm is defined as 
%\begin{align*}
%    \|u|V_{\beta,\delta}^{l-1/p,p}(\Gamma_j) \| =\inf \{\|v| V_{\beta,\delta}^{l,p}(K)\|: \: v\in V_{\beta,\delta}^{l,p}(K),\  v=u  \text{ on }  \Gamma_j    \}.
%\end{align*}
%\end{definition}

For the regularity in fractional Sobolev spaces we use the following result of Grisvard, cf.  \cite[Corollary~2.6.7]{Gris} or \cite[Proposition~4.1]{HS1}.

\begin{proposition}[\bf Fractional Sobolev regularity for  Dirchlet and Neumann problem with homogeneous boundary values]
\label{prop4.1}
Let $D\subset \real^3$ be any bounded Lipschitz domain. Then there exists a number $\alpha_0>3/2$ such that for every $f\in L_2(D)$ the  solution $u$ of the Poisson equation (either with homogeneous Dirichlet or with homogeneous Neumann boundary conditions) belongs to $H^s(D)$ for every $s<\alpha_0$.
\end{proposition}

In order to also allow inhomogeneous Dirichlet boundary conditions we use  a result of Jerison and Kenig \cite[Thm.~5.1]{JK95}. 
\begin{proposition}\label{prop-dirichlet-inhom-0}
\label{JerisonKenig}
Let $D\subset \real^3$ be any bounded Lipschitz domain and let $g\in H^{s-1/2}(\partial D)$ for some $1/2<s<3/2$. Then, there exists a unique harmonic function $u$, i.e. $\Delta u=0$ in $D$, such that
\begin{align*}
    u\in H^s(D) \qquad \text{and}\qquad u=g \quad\text{on}\quad \partial D.
\end{align*}
\end{proposition}
Proposition \ref{prop-dirichlet-inhom-0} allows us to reduce  \eqref{poisson} to a Dirichlet problem with homogeneous boundary values: 
%\begin{align}
%    \label{decomposition1}
%    -\Delta u_1 &=f \quad &\text{in}\quad &K \notag\\
%    u_1 &=0 \quad &\text{on}\quad &\Gamma=\bigcup_{j=1,...,n} \Gamma_j
%\end{align}
%and
%\begin{align}
%    \label{decomposition2}
%    -\Delta u_2 &=0 \quad &\text{in}\quad &K \notag\\
%    u_2 &=g \quad &\text{on}\quad &\Gamma=\bigcup_{j=1,...,n} \Gamma_j,
%\end{align}
%where $g=g_j$ on $\Gamma_j,\: j=1,...,n$. Then, 
%If $u_1$ is the solution of \eqref{decomposition1}, and $u_2$ of \eqref{decomposition2},  $u=u_1+u_2$ is the solution of \eqref{poisson}. 
%Additionally, Propositions \ref{prop4.1}, \ref{JerisonKenig} are applicable for $u_1,\: u_2$, respectively.
If $u_1$ is the solution of the homogeneous problem 
\[
\Delta u_1=0\quad \text{in}\quad  D,  \qquad u_1=g \quad \text{on}\quad \partial D,
%\quad \text{for}\quad j=1,...,n,  
\]
and $u_2$ solves the  problem 
\[
    -\Delta u_2  =f \quad \text{in}\quad  D, \qquad 
    u_2 =0 \quad \text{on}\quad  \partial D, %\quad\text{for}\quad j=1,...,n,  
\]
then $u=u_1+u_2$ is the solution of the inhomogeneous problem 
\begin{equation}\label{inhom-dirichlet}
    -\Delta u  =f \quad \text{in}\quad  D, \qquad 
    u =g \quad \text{on}\quad  \partial D.  %\quad\text{for}\quad j=1,...,n,  
\end{equation}
Therefore, a combination of  Propositions \ref{prop4.1} and  \ref{JerisonKenig}  gives the following regularity result: 
\begin{proposition}[\bf Fractional Sobolev regularity for the Dirichlet problem]\label{prop-dirichlet-inhom}
%\label{JerisonKenig}
Let $D\subset \real^3$ be any bounded Lipschitz domain and assume that  $f\in L_2(D)$ and  $g\in H^{s-1/2}(\partial D)$ for some $1/2<s<3/2$. Moreover, let $\alpha_0>3/2$ be the constant from Proposition \ref{prop4.1}.  Then there exists a solution $u\in H^{\overline{\alpha}}(D)$ of \eqref{inhom-dirichlet} for   %,  which satisfies 
%belongs to $H^{\alpha}(D)$ for every $\alpha<s$ if $g\neq 0$ and $\alpha<\alpha_0$ for $g=0$. 
%
%\begin{align*}
%    u\in H^s(D) \qquad \text{and}\qquad u=g \quad\text{on}\quad \partial D.
%\end{align*}
%\[
%u\in H^{\alpha}(D),\quad \text{where}\quad \begin{cases}
%\alpha<s, & g\neq 0,\\
%\alpha<\alpha_0,& g=0. 
%\end{cases}
%\]
every $\overline{\alpha}<\alpha_0$ in case of homogeneous boundary conditions ($g=0$) and $\overline{\alpha}=s$ in case of inhomogeneous boundary conditions.
\end{proposition}

Moreover, in order to establish regularity in Besov spaces we rely on the   following important regularity result  of Maz'ya and Rossmann in weighted Sobolev spaces, whose proof can be found in \cite[p.~127, Thm.~3.5.10]{MR10}. Since we want to apply it in combination with the embeddings from Section \ref{embeddings3} we restrict ourselves to the case when $p=2$.

\begin{theorem}[\bf Weighted Sobolev regularity for the Dirichlet Problem]
\label{theoremMR}
 Let $K$ be a polyhedral cone, $f\in V_{\beta,\delta}^{l-2,2}(K),\: l\geq 1$, and  $g_{j}\in V_{\beta,\delta}^{l-k+1/2,2}(\Gamma_j) $ for $j=1,...,n$. We assume that the line $\text{Re}\: \lambda=l-\beta-3/2$ does not contain eigenvalues of the pencil $\mathfrak{A}(\lambda)$ and that the components of $\delta$ satisfy 
 \begin{align}
 \label{conditionsept15}
     -\delta_{+}^{(j)}<\delta_j-l+1 <\delta_{-}^{(j)}\quad \text{for}\quad j=1,...,n.   
 \end{align}
% where $\theta_j$ is the angle at the edge $M_j$. 
 Then there exists a unique solution $u\in V_{\beta,\delta}^{l,2}(K)$ of the problem \eqref{poisson} satisfying the estimate
 \begin{align}
     \label{estimatesept15}
     \|u|V_{\beta,\delta}^{l,2}(K)\|\leq c \left( \|f|V_{\beta,\delta}^{l-2,2}(K)\|+\displaystyle\sum_{j=1}^n\ \|g_{j}|V_{\beta,\delta}^{l-k+1/2,2}(\Gamma_j) \|  \right)
 \end{align}
 with a constant $c$ independent of $f$ and $g_{j}$.
 \end{theorem}

We can now combine Proposition \ref{prop-dirichlet-inhom} and  Theorem \ref{theorem4},  
%to $\varphi u$,  
which together with  Theorem \ref{theoremMR} %and \eqref{phimultipl} 
yield the following: 
%%Besov regularity of the solution of problem \eqref{poisson}.

\begin{theorem}[\bf Besov regularity for the Dirichlet problem] \label{theorem7}
Let $K$ be a  polyhedral cone %and let $\varphi$ be the cut-off function of \eqref{cutoff}.
and let the right-hand side of \eqref{poisson} satisfy $f\in V_{\beta,\delta}^{l-2,2}(K)\cap L_2(\tilde{K})$ with $l\in\mathbb{N}$, $\beta\in \mathbb{R}$, and $\delta=(\delta_1,\ldots, \delta_n)\in \mathbb{R}^n$. Concerning the boundary data let $g\in H^{s-1/2}(\tilde{\Gamma})$ for some $1/2<s<3/2$ with $g=g_j$ on $\Gamma_j$ and $g_{j}\in V_{\beta,\delta}^{l-1/2,2}(\Gamma_j)$ for $j=1,...,n$. %Let $\alpha_0=\alpha_0(K)$ be the number defined in Proposition \ref{prop4.1}. 
Moreover, we assume that the line $\text{Re}\: \lambda=l-\beta-3/2$ does not contain eigenvalues of the pencil $\mathfrak{A}(\lambda)$ and that the components of $\delta$ satisfy \eqref{conditionsept15} and $\delta_j>0 $ for all $j=1,...,n$ and $l>\beta$. Then for the solution $u$ of \eqref{poisson} we have
\begin{equation}\label{ass-B}
 u\in B_{\tau,\tau}^r(\tilde{K}) \qquad  \text{for} \quad r<\min\Big\{  l,3(l-|\delta|),3 \overline{\alpha} \Big\} , \qquad \frac{1}{\tau}=\frac{r}{3}+\frac{1}{2},
\end{equation}
where $\overline{\alpha}(K)=\overline{\alpha}$  is the number from Proposition \ref{prop-dirichlet-inhom}, i.e., $\overline{\alpha}<\alpha_0$ in case of homogeneous boundary conditions ($g=0$) and $\overline{\alpha}=s$ in case of inhomogeneous boundary conditions.
\end{theorem}

\begin{proof}
 Our assumptions on $f$ together with the regularity results from Theorem \ref{theoremMR} and Proposition \ref{prop-dirichlet-inhom} yield the existence of a solution $u$ such that 
\[  u\in V^{l,2}_{\beta, \delta}(K)\cap H^{s'}(\tilde{K}) \qquad \text{for every}\quad s'\leq \overline{\alpha}. \]
Using  the fact that  $B^{s'}_{2,2}=H^{s'}$ in the sense of equivalent norms, 
 Theorem \ref{theorem4} implies that $u\in B_{\tau,\tau}^r(\tilde{K})$ with the restrictions from \eqref{ass-B}.

%We consider the same wavelet characterization of the Besov norm as in Chapter \ref{embeddings3} and the different sets of wavelets from the proof of Theorem \ref{theorem4}.\\
%In Step 2 of that proof we obtained that 
%\begin{align}
%    \label{egy2}
% \displaystyle \sum_{(I,\psi)\in \Lambda^0} |I|^{(\frac{1}{p}-\frac{1}{2})\tau}  
% |\langle \tilde{u} ,\psi_I \rangle|^{\tau}\lesssim\|u|V_{\beta,\delta}^{l,p}(K)\|^{\tau},
%\end{align}
%for \textcolor{blue}{$l>r,\:l>\beta   $} and \textcolor{blue}{$3(l-|\delta|)>r$}. By Theorem \ref{theoremMR} $u\in %V_{\beta,\delta}^{l,p}(K)$, and hence the left-hand side of \eqref{egy2} is finite.\\
%In Step 3 of Theorem \ref{theorem4} we obtained \eqref{dec9} for arbitrary $s>0$ and $r>0$. It is known that $B^s_{2,2}=H^s$ in the sense of equivalent norms. Moreover, since $f\in L_2(K)$, we can use Proposition \ref{prop4.1}, from which it follows, that $u\in H^s(K)$ for all \textcolor{blue}{$s<\alpha_0$}. Hence, the left-hand side of \eqref{dec9} is finite as well.\\
%In Step 4 of Theorem \ref{theorem4} we had the same calculation as in Step 3 with the parameter restriction \textcolor{blue}{$r<3s$} for all $s<\alpha_0$.\\
%Summarizing the three cases and using the characterization of the $B_{\tau,\tau}^r$-norm of $u$ we obtain that $u\in B_{\tau,\tau}^r(K)$ with the parameter conditions of the theorem.
\end{proof}

\begin{remark}
\label{remark7}
%\begin{itemize}
% \item[(i)] 
   Theorem \ref{theorem7} also holds  for $\delta$ with negative components. In this case we  use  Theorem \ref{theorem5} in the proof of Theorem \ref{theorem7} above: Then instead of the parameter restriction 
   $r<\min\{  l,3(l-|\delta|),3\overline{\alpha} \}$ we obtain  that $r<3\overline{\alpha}$ and one of the conditions for  $r$ in \eqref{sept16} holds.
\end{remark}

\subsection{Regularity for the Neumann  and for mixed boundary value problems}

In order to obtain similar regularity results  for the Neumann and mixed boundary value problems we first need to adapt our weighted Sobolev spaces appropriately. Following \cite[Subsection~2.2, p.~438]{MR02} we now consider the following generalization of the spaces $V_{\beta,\delta}^{l,p}$. 
%, which play a crucial role for investigating  the regularity of these problems as we will see below. 

\begin{definition}[\bf Weighted Sobolev spaces $\mathcal{W}_{\beta,\delta}^{l,p}(K;\tilde{J})$]\label{def-W-space}
Let $K$ be a (bounded or unbounded) polyhedral cone in $\mathbb{R}^3$ and  $S=\{0\}\cup M_1 \cup...\cup M_n$. Moreover, 
let $\tilde{J}$ be a subset of $J=\{1,2,...,n   \}$, $l\in \mathbb{N}_0$, $1\leq p<\infty$, $\beta\in \mathbb{R}$, $\delta=(\delta_1,...,\delta_n)\in \mathbb{R}^n$ and $\delta_j>-2/p$ for $j\in J\setminus\tilde{J}$. 
Then the space $\mathcal{W}_{\beta,\delta}^{l,p}(K;\tilde{J})$ is defined as the closure of the set $$C_*^{\infty}(K,S):=\{ u|_K: \: u\in C_0^{\infty}(\mathbb{R}^3\setminus S) \}$$ 
with respect to the norm 
\begin{align}
\label{specWspace}
    \|u|\mathcal{W}_{\beta,\delta}^{l,p}(K;\tilde{J}) \|:=\displaystyle \left(\int_K \sum_{|\alpha|\leq l} \rho^{p(\beta-l+|\alpha|)}  \prod_{j\in \tilde{J}} \left(\frac{r_j}{\rho}\right)^{p(\delta_j-l+|\alpha|)}  \prod_{j\in J \setminus \tilde{J}} 
    \left(\frac{r_j}{\rho}\right)^{p\delta_j}  |\partial^{\alpha}u|^p\: dx
    \right)^{1/p}. 
\end{align}
Furthermore, we put $V_{\beta,\delta}^{l,p}(K):=\mathcal{W}_{\beta,\delta}^{l,p}(K;J)$ and $ W_{\beta,\delta}^{l,p}(K):=\mathcal{W}_{\beta,\delta}^{l,p}(K;\emptyset)$.
\end{definition}

\begin{remark}\label{rem-traces-cond}
\begin{itemize}
    \item[(i)] According to the definition of the norms one immediately obtains  that
\begin{align}
    \label{sept23}
    V_{\beta,\delta}^{l,p}(K) \subset \mathcal{W}_{\beta,\delta}^{l,p}(K;\tilde{J}) \subset  W_{\beta,\delta}^{l,p}(K) .
\end{align}
\item[(ii)] 
The trace spaces for $\mathcal{W}_{\beta,\delta}^{l,p}(K;\tilde{J})$ and $W_{\beta,\delta}^{l,p}(K)$ when $l\geq 1$ on $\Gamma_j$  can be defined in the same way as in Definition \ref{tracespaceV} and are denoted by $\mathcal{W}_{\beta,\delta}^{l-\frac 1p,p}(\Gamma_j;\tilde{J})$ and ${W}_{\beta,\delta}^{l-\frac 1p,p}(\Gamma_j)$, respectively.
\item[(iii)] The set $\tilde{J}$ will be chosen to contain all $j\in J$ such that the Dirichlet condition of problem \eqref{generalproblem} is given on at least one face $\Gamma_{j_{\pm}}$ adjacent to the edge $M_j$. In particular, for the Neumann problem this yields $\tilde{J}=\emptyset$. 
\item[(iv)] The restriction $\delta_j>-2/p$ for $j\in J\setminus \tilde{J}$ (i.e., Neumann boundary conditions on both sides) comes from the fact that for  Neumann boundary conditions we need to be a little careful when dealing with the traces on the faces $\Gamma_{j_{\pm}}$ adjacent to an edge $M_j$ (the functions $g_{j_{\pm}}$ need to satisfy certain 'compatibility' conditions). In particular, if $p=2$, it can be shown that the trace of a function $u\in \mathcal{W}_{\beta,\delta}^{l,p}(K;\tilde{J})$ with support near an edge $M_j$ with  $j\in J\setminus \tilde{J}$  belongs to the Sobolev-Slobodeckij space 
$$H_p^{l-\delta_j-2/p}(M_j)\qquad  \text{if}\qquad  -2/p<\delta_j<l-2/p$$ and $\delta_j+2/p$ is not integer, cf. \cite[Lem.~6.2.2]{MR10}. Note that for $p=2$ the Sobolev-Slobodeckij spaces coincide with our fractional Sobolev spaces $H^s$.  For $\delta_j\geq l-2/p$  the traces of functions $u\in \mathcal{W}_{\beta,\delta}^{l,p}(K;\tilde{J})$ in general do not exist on $M_j$, cf. \cite[p.~225]{MR10}. 
\end{itemize}
\end{remark}

In particular,  the extension operator $\mathfrak{E}$ from  Theorem \ref{extop} generalizes to the setting of   the spaces  $\mathcal{W}_{\beta,\delta}^{l,p}(K;\tilde{J}) $.

\begin{theorem}[\bf Extension operator for the spaces $\mathcal{W}_{\beta,\delta}^{l,p}(K;\tilde{J})$]
\label{extopneu}
Let $K\subset \mathbb{R}^3 $ be a  polyhedral cone and let $\tilde{J}$ be a subset of $J=\{1,2,...,n   \}$, $l\in \mathbb{N}_0$, $1\leq p<\infty$, $\beta\in \mathbb{R}$, $\delta=(\delta_1,...,\delta_n)\in \mathbb{R}^n$, and $\delta_j>-2/p$ for $j\in J\setminus\tilde{J}$. Then there exists a bounded linear extension operator $$\mathfrak{E}:\mathcal{W}_{\beta,\delta}^{l,p}(K;\tilde{J})\rightarrow
\mathcal{W}_{\beta,\delta}^{l,p}(\mathbb{R}^3;\tilde{J}).
$$ 
\end{theorem}

\begin{remark}
The norm of the space $ \mathcal{W}_{\beta,\delta}^{l,p}(\mathbb{R}^3;\tilde{J})$ is defined  similarly as in  \eqref{specWspace}  replacing the integral domain by $\mathbb{R}^3$.
\end{remark}

\begin{proof}
According to \cite[Sect.~2, p.~439]{MR02} the norm localization in \cite[Lem.~2.4]{SS20} holds also for the spaces $\mathcal{W}_{\beta,\delta}^{l,p}(K;\tilde{J})$. Therefore, \cite[Lem.~3.8]{SS20} and Theorem \ref{extop} can be transferred to these spaces mutatis mutandis, which proofs Theorem \ref{extopneu}.
\end{proof}

With this extension operator at  hand we can further refine the embedding results of Theorems \ref{theorem4} and \ref{theorem5}. 

\begin{theorem}[\bf Embeddings for the spaces $\mathcal{W}_{\beta,\delta}^{l,p}(K;\tilde{J})$]\label{embeddingBesovWspaces}
Let $K$ be a bounded polyhedral cone in $\mathbb{R}^3$. Moreover, let $\tilde{J}$ be a subset of $J=\{1,2,...,n   \}$, $l\in \mathbb{N}_0$, $1< p<\infty$, $\beta\in \mathbb{R}$ with $l>\beta$, and  $\delta=(\delta_1,...,\delta_n)\in \mathbb{R}^n$.  %with $\delta_i>0$ for all $i=1,...,d$. 
Then  we have a continuous embedding
$$\mathcal{W}_{\beta,\delta}^{l,p}(K;\tilde{J})\cap B_{p,p}^s(K)\hookrightarrow B_{\tau,\tau}^r(K),\quad \frac{1}{\tau}=\frac{r}{3}+\frac{1}{p},$$
where either 
$$ r<\min\Big\{l,3(l-|\delta|), 3s\Big\} $$ 
if $\delta_j>0$ for all $j=1,...,n$ or  
\[r<3s \quad \text{and one of the conditions for  $r$ in \eqref{sept16} holds},  \]
if $\delta$ has  at least one negative component and $\delta_j>-\frac{2}{p}$ for all $j\in J\setminus \tilde{J}$.
\end{theorem}

\begin{proof}
The proof follows along the same lines as in Theorems \ref{theorem4} and \ref{theorem5} replacing the space $ V_{\beta,\delta}^{l,p}(K)$ by $\mathcal{W}_{\beta,\delta}^{l,p}(K;\tilde{J})$. In particular, we remark that in the estimates there we only considered the term with the highest derivatives $|\alpha|=l$ appearing in the norm of the spaces $V_{\beta,\delta}^{l,p}(K)$, which coincides with the corresponding term in the norm of the spaces $\mathcal{W}_{\beta,\delta}^{l,p}(K;\tilde{J})$. 
\end{proof}

\subsubsection{The Neumann problem}

In this subsection we study the Besov regularity of the solution of the Poisson equation with Neumann boundary conditions only, i.e., $J_0=\emptyset$ in \eqref{generalproblem}, which reads as:
\begin{equation}\label{neumannp}
    -\Delta u  =f \quad \text{in}\quad  K, \qquad 
    \frac{\partial}{\partial \nu}u =g_{j} \quad \text{on}\quad  \Gamma_j \quad\text{for}\quad j=1,...,n. 
\end{equation}

%In this subsection we consider the Besov regularity of the solution to the Neumann problem,
%\begin{align}
%\label{neumannp}
%    -\Delta u & =f \quad &\text{in}\quad & K\notag\\
%    \frac{\partial}{\partial n}u &=g_{j} \quad &\text{on}\quad & \Gamma_j \quad\text{for}\quad j=1,...,d,
%\end{align}
%as a special case of \eqref{generalproblem} with $\tilde{J}=\emptyset$.

For the regularity of the  Neumann problem with homogeneous boundary conditions in fractional Sobolev spaces we can again rely on Proposition \ref{prop4.1}.  For inhomogeneous boundary values we use  \cite[Thm.~2.1]{Dah2001}, whose original proof can be found in \cite{Zan00}.

\begin{proposition}
\label{Zanger-Dahlke-0}
Let $D\subset \real^3$ be any bounded Lipschitz domain whose complement is connected. Moreover, let $g\in H^{s-3/2}_{1^\perp}(\partial D):=\{g\in H^{s-3/2}(\partial D): \ g(1)=0\}$ for some $1/2<s<3/2$ (here $H^r(\partial D)$ with $r<0$ stands for the dual space of $H^{-r}(\partial D)$). Then there exists a unique harmonic function $u$, i.e. $\Delta u=0$ in $D$, such that
\begin{align*}
    u\in H^s(D) \qquad \text{and}\qquad \frac{\partial u}{\partial \nu} =g \quad\text{on}\quad \partial D.
\end{align*}
\end{proposition}

\begin{remark} 
Note that since we are interested in solutions of the Neumann problem, we  consider function spaces on the boundary with functions having mean-value $0$. This is reflected in choosing the spaces 
\[
H^{s-3/2}_{1^\perp}(\partial D)=\{g\in H^{s-3/2}(\partial D): \ g(1)=0\}, 
\]
where any $g\in H^{s-3/2}(\partial D)$  with $1/2<s<3/2$ can be identified as a linear
functional on the  function space $H^{3/2-s}(\partial D)$. %: In particular,  $H^r(\partial D)$ with $r<0$ stands for the dual space of $H^{-r}(\partial D)$. \\ 
Moreover, it can be shown that if  $\mathcal{H}\subset L_1(\partial D)$ is a dense subset of $H^{s-3/2}(\partial D)$  then 
\[
\mathcal{H}_{1^\perp}:=\left\{
f-\frac{1}{|\partial D|}\int_{\partial D}f \mathrm{d}\sigma : \ 
f\in \mathcal{H} 
\right\}
\]
is dense in $H^{s-3/2}_{1^\perp}(\partial D)$, cf. \cite[p.~1781]{Zan00}. 
\end{remark}

Proposition \ref{Zanger-Dahlke-0} now allows us to reduce the inhomogeneous Neumann problem 
\begin{equation}\label{neumannp-inhom}
    -\Delta u  =f \quad \text{in}\quad  D, \qquad 
    \frac{\partial}{\partial \nu }u =g \quad \text{on}\quad  \partial D, 
\end{equation}
to the homogeneous one similarly as explained for the Dirichlet problem after Proposition \ref{prop-dirichlet-inhom-0}. Therefore, in view of  Propositions \ref{prop4.1} and \ref{Zanger-Dahlke-0} we obtain   the following:

\begin{proposition}
[\bf Fractional Sobolev regularity for  the Neumann problem]
\label{Zanger-Dahlke}
Let $D\subset \real^3$ be any bounded Lipschitz domain whose complement is connected. Moreover, let $g\in H^{s-3/2}_{1^{\perp}}(\partial D)$ for some $1/2<s<3/2$ and  $\alpha_0>3/2$ be the constant from Proposition \ref{prop4.1}. Then there exists a unique solution $u\in H^{\overline{\alpha}}(D)$ of \eqref{neumannp-inhom} for every  $\overline{\alpha}<\alpha_0$ in case of homogeneous boundary conditions ($g=0$) and $\overline{\alpha}=s$ in case of inhomogeneous boundary conditions.
\end{proposition}

Concerning the regularity in Sobolev spaces for the Neumann problem, we use 
 the following result of Maz'ya and Rossmann \cite[Section~6.1, p.~459]{MR02}.

\begin{theorem}[\bf Weighted Sobolev regularity for the Neumann Problem]
\label{dec10de}
Let $K$ be a  polyhedral cone in $\mathbb{R}^3$.
Moreover, let $l\in \mathbb{N}_0,\: l\geq 2$, $\beta\in \mathbb{R}$, and  $\delta=(\delta_1,...,\delta_n)\in \mathbb{R}^n$. % and $\delta_j>-2/p$ for $j=1,...,d$. 
Then the problem \eqref{neumannp} is uniquely solvable in $W_{\beta,\delta}^{l,2}(K)$ for arbitrary $f\in W_{\beta,\delta}^{l-2,2}(K)$ and $g_j\in W_{\beta,\delta}^{l-3/2,2}(\Gamma_j)$, $j=1,...,n$, if $l-\beta-3/2$ is not an eigenvalue of the pencil $\mathfrak{A}(\lambda)$ and the components of $\delta$ satisfy the inequalities
\begin{align}
\label{dec10}
    \max(l-\delta_{+}^{(j)},0)< \delta_j+1<l,\qquad j=1,...,n.  
\end{align}
%where $\theta_j$ denotes the angle at the edge $M_j$.
\end{theorem}

\begin{remark} As for condition \eqref{dec10} we refer to Remark \ref{rem-traces-cond} (iv). 
\end{remark}

With Propostion \ref{Zanger-Dahlke} and Theorem \ref{dec10de} we obtain the following counterpart of Theorem \ref{theorem7} for the Neumann problem.

\begin{theorem}[\bf Besov regularity for the Neumann problem]
\label{pdeforWspace}
Let $K$ be a polyhedral cone and %let $\varphi$ be the cut-off function of \eqref{cutoff}.
let $l\in\mathbb{N},\: l\geq 2$, $\beta\in \mathbb{R}$, $l>\beta$, and $\delta=(\delta_1,\ldots, \delta_n)\in \mathbb{R}^n$.  
%with $\delta_j>0$ for all $j=1,...,d$. 
Further we assume that  the right-hand side of \eqref{neumannp} satisfies $f\in W_{\beta,\delta}^{l-2,2}(K)\cap L_2(\tilde{K})$, and for the boundary data we have $g\in H^{s-3/2}_{1^{\perp}}(\tilde{\Gamma})$ for some $1/2<s<3/2$ with  $g=g_j$ on $\Gamma_j$ and    $g_{j}\in W_{\beta,\delta}^{l-3/2,2}(\Gamma_j)$ for $j=1,...,n$. 
%Let $\alpha_0=\alpha_0(K)$ be the number defined in Proposition \ref{prop4.1}. 
Assume that $l-\beta-3/2$ is not an eigenvalue of the pencil $\mathfrak{A}(\lambda)$ and the components of $\delta$ satisfy 
\begin{align}
\label{dec11}
    \max(l-\delta_{+}^{(j)},0)< \delta_j+1<l, \qquad j=1,...,n. 
\end{align}
Then for the solution $u$ of \eqref{neumannp} we have
\begin{align*}
    u\in B_{\tau,\tau}^r(\tilde{K}), \quad \frac{1}{\tau}=\frac{r}{3}+\frac{1}{2},
    %\quad \text{for} \quad \textcolor{blue}{r<\min\{  l,3(l-|\delta|),3\alpha_0 \} }.
\end{align*}
where either 
$$ r<\min\Big\{{l},{3(l-|\delta|)}, {3\overline{\alpha}} \Big\} $$ 
if $\delta_j>0$ for all $j=1,...,d$ or  
\[r<3\overline{\alpha}  \quad \text{and one of the conditions for  $r$ in \eqref{sept16} holds},  \]
if $\delta$ has  at least one negative component, % and $\delta_j>-\frac{2}{p}$ for all $j\in J\setminus \tilde{J}$.
where $\overline{\alpha}(K)=\overline{\alpha}$ is the number from Proposition \ref{Zanger-Dahlke}. 
\end{theorem}

\begin{proof}
The proof is a counterpart of Theorem \ref{theorem7} replacing Theorems \ref{theorem4} and  \ref{theorem5} by Theorem \ref{embeddingBesovWspaces}, Theorem \ref{theoremMR} by Theorem \ref{dec10de}, and Proposition \ref{prop-dirichlet-inhom} by Proposition \ref{Zanger-Dahlke}.
\end{proof}

\paragraph{Comparision with previous results} We compare Theorem \ref{pdeforWspace} with the result of Dahlke and Sickel in \cite[Thm.~3.1, p.~11]{HS1} which reads as follows: 

\begin{theorem}[\bf Besov regularity result from Dahlke/Sickel]
\label{theorem8}
Suppose that the right-hand side of \eqref{neumannp} satisfies $f\in W_{\beta,\delta}^{l-2,2}(K)\cap L_2(\tilde{K})$, where $l\geq 2$ is a natural number. Further assume that $g_j\in W_{\beta,\delta}^{l-3/2,2}(\Gamma_j),\: j=1,...,n$. Let $\alpha_0=\alpha_0(K)$ be the number defined in Proposition \ref{prop4.1}. Then there exists a countable set $E$ of complex numbers such that the following holds: If the real number $\beta$ and the vector $\delta$ are chosen such that $\beta<l$,
\begin{align}
\label{jan7}
    \lambda\neq l-\beta-3/2 \quad \text{for all} \: \: \lambda\in E,
\end{align}
and 
\begin{align*}
    \max(l-\frac{\pi}{\theta_j},0)<\delta_j +1<l, \quad j=1,...,n,
\end{align*}
where $\theta_j$ denotes the angle at the edge $M_j$, then the solution of \eqref{neumannp} satisfies
\begin{align*}
    u\in B_{\tau,\tau}^r(\tilde{K}),\qquad  r<\min\left\{l,\frac{3}{2}\alpha_0,3(l-|\delta|)\right\}, \quad  \frac{1}{\tau}=\frac{r}{3}+\frac{1}{2}.
\end{align*}
\end{theorem}

%\todo[inline]{vereinheitlichen schreibweise in Theoremen, operator pencils vermeiden?!}
 Most obvious, we see in comparison that the restriction in the upper bound for $r$ in terms of  $r<\frac{3}{2}\alpha_0$ in Theorem \ref{theorem8} was  improved to $r<3\alpha_0$ in Theorem \ref{pdeforWspace}. The improvement comes from the fact that we use the  extension operator from Theorem \ref{extopneu} for the proof of the embedding in Theorem \ref{embeddingBesovWspaces}. This way (in the  proof of Theorem \ref{embeddingBesovWspaces}) we can consider the wavelets with support intersecting the surface of the cone in the same way as the interior wavelets. Therefore, only the wavelets  which are close to the singular set, i.e.,  the vertex and the edges of the cone, have to be treated in a different manner. In the corresponding proof for the embedding of Dahlke and Sickel there is no such extension operator and therefore, they have to treat these wavelets as the other boundary wavelets, which leads to the additional restriction on $r$.\\
 Moreover, we remark that Theorem \ref{theorem8} does not hold in this form for $\delta$ with negative components. In this case one requires the additional assumptions from \eqref{sept16}, since the estimates when the components of $\delta$ are negative are more delicate, cf. the proof of Theorem \ref{theorem5} compared to the proof of Theorem \ref{theorem4}. \\  
 There is one more apparent difference between the two theorems: In Theorem \ref{pdeforWspace} we require that $l-\beta-3/2$ is not an eigenvalue of the pencil $\mathfrak{A}(\lambda)$, but  according to \cite[p.~97, 292/293]{MR10} this is equivalent to the condition  \eqref{jan7} in Theorem \ref{theorem8}.\\
 Finally, let us remark that Theorem \ref{theorem8} above as stated in \cite{HS1}  is not quite correct since the assumptions on the boundary data $g_j$ are not sufficient in order guarantee a solution $u\in H^s(\tilde{K})$. One has to add our assumptions from  Theorem \ref{pdeforWspace} for  inhomogeneous boundary data.

 \begin{remark}
 Let us mention that there is another regularity result of Dahlke  in \cite[Thm.~3.3]{Dah2001}, which gives  for the  homogeneous Neumann problem \eqref{neumannp-inhom} on a bounded Lipschitz domain $D$ with $f=0$ and 
 %within the Besov scale $B^{r}_{\tau,\tau}$, 
   boundary values $g\in H^{s-\frac 32}(\partial D)$,  $1/2<s<3/2$, that   the solution $u$  belongs to $B^r_{\tau,\tau}(D)$ for $0<r<\frac{3}{2}s+\frac{3}{4}$. It is obtained using a bootstrapping strategy and also improves the bound  $r<\frac{3}{2}s$.  However, we see that our restriction  $r<3s$  is even better (but restricted to polyhedral cones $K$).  
 \end{remark}

\subsubsection{Mixed boundary value problems}

We consider now mixed boundary value problems \eqref{generalproblem}, where sets $J_0$ and $J_1$ are both supposed to be non-empty, i.e.,  some of the faces ($\Gamma_j$ with $j\in J_0$) of the cone satisfy  Dirichlet  and the other ones ($\Gamma_j$ with $j\in J_1$) Neumann boundary conditions. 
%In this subsection we denote by $J_0$ and $J_1$ the non-empty sets of all indices $j\in J=\{1,\ldots, n\}$ such that the face $\Gamma_j$ satisfies the Dirichlet or the Neumann boundary condition, respectively. 

% Our problem  reads as follows:
%%\textcolor{blue}{das könnten wir weglassen und verweisen auf $\eqref{generalproblem}$}
%\begin{equation}\label{mixedproblem_neu}
% \begin{cases} \ 
%  -\Delta u \ =\ f &\text{in}\quad  K, \\
%  \quad \quad u \ =\ g_{j} &\text{on}\quad  \Gamma_j \quad\text{for}\quad j\in{J_0}, \\
%   \quad  \frac{\partial}{\partial \nu}u \ =\ g_{j} \quad &\text{on}\quad  \Gamma_j \quad\text{for}\quad j\in J_1,
% \end{cases}
%\end{equation}

Moreover, we denote by $\tilde{J}$ the set of all $j\in J$ such that the Dirichlet condition is given on at least one face $\Gamma_k$ adjacent to the edge $M_j$, i.e., $M_j\subset \overline{\Gamma}_k$ for at least one $k\in J_0$.

\paragraph{Fractional Sobolev regularity for mixed problems} 
 Concerning the fractional Sobolev regularity of mixed boundary value problems, we need a counterpart of Proposition \ref{prop4.1}, 
 since it holds only for pure Dirichlet or pure Neumann boundary conditions. For this we  use the following result of Grisvard, whose proof can be found in \cite[Thm.~2.6.3, p.~75]{Gris}.

%\todo[inline]{check ob Grisvards Ergebnisse nur fuer homogene Randbedingungen stimmen!}
%\textcolor{green}{Also Corollary 2.6.7/Grisvard/p.79 is based on Thm.2.6.3 and he obtained the Corollary only for homogeneous problem. Dann auch das Ergebnis von Dahlke/Sickel geht nur für homogenes %Problem...?}

\begin{proposition}
\label{propGris}
Let $D\subset \mathbb{R}^3$ be a bounded polyhedral domain. The open neighbourhood near a vertex of $D$ is a polyhedral cone $K$ which in spherical coordinates can be described as  
$$K=\{(\rho,\omega): \ 0<\rho<c, \ \omega\in \Omega:=K\cap S^2\}.$$  
We denote by $\Delta_{\omega}$ the Laplace-Beltrami operator on $\Omega$ with eigenvalues $\widetilde{\lambda}_l $  and  corresponding orthonormalized eigenfunctions $\psi_l$, i.e., 
\[
-\Delta_{\omega}\psi_l=\widetilde{\lambda}_l\psi_l \quad \text{on}\quad \Omega. 
\]
For $f\in L^2(D)$ let $u$ be the solution of \eqref{generalproblem} with homogeneous boundary conditions. Then there exists unique numbers $c_l$ such that
\begin{align}
    \label{uregGrisvard}
    u_{\text{reg}}=u-\displaystyle\sum_l c_l \rho^{-\frac{1}{2}+ \sqrt{ \widetilde{\lambda}_l+\frac{1}{4}}} \psi_l(\omega) \in H^s (K), 
\end{align}
where  the sum is over all $l$ such that $0\leq \widetilde{\lambda}_l\leq s^2-2s+\frac{3}{4}$,   $s\leq 2$ with $s<\Lambda+1$ and 
\begin{align}
    \label{LambdaGris}
    \Lambda:=\min \{ \lambda_{j,m} \in (0,1): \ j=1,\ldots, n;  \  m\geq 1 \},
\end{align}
where $\lambda_{j,m} $ are the eigenvalues at the edge $M_{j}$.

\end{proposition}

\begin{remark}
The function $u_{\text{reg}}$ is the regular part of $u$ and the terms $ \rho^{-\frac{1}{2}+ \sqrt{ \widetilde{\lambda}_l+\frac{1}{4}}} \psi_l$ constitute the singular part of $u$  denoted  by $u_{\text{sing}}$. 
\end{remark}

Using this result of Grisvard we obtain the following substitute of  Proposition \ref{prop4.1} for mixed boundary value problems: 

\begin{proposition}[\bf Fractional Sobolev regularity for  mixed problems with homogeneous boundary values]
\label{H5/4theorie}
Let $D$ be a bounded polyhedral domain in $\mathbb{R}^3$. Then there exists a number $\alpha_0\geq 5/4$ such that for every $f\in L_2(D)$ the solution $u$ of the mixed boundary value problem   \eqref{generalproblem} with homogeneous boundary conditions belongs to $H^s(D)$ for every $s<\alpha_0$.
\end{proposition}

\begin{proof}
We show first that the number $\Lambda\geq \frac 14$ in  Proposition \ref{propGris}. According to \cite[p.~72]{Gris} the eigenvalues $\lambda_{j,m}>0$, $m\in \mathbb{N}$,  at an edge $\ M_j=\overline{\Gamma}_j \cap \overline{\Gamma}_k\ $ are given by 
\begin{align}
    \label{lambdajkm_neu}
    \lambda_{j,m} &\ =\ m \frac{\pi}{\theta_j} &\quad &\text{for}\quad  j,\ k\in J_0,   \notag\\
    \lambda_{j,m}&\ =\ (m-1) \frac{\pi}{\theta_j} &\quad &\text{for}\quad j,\ k\in  J_1, \notag\\
    \lambda_{j,m}&= \left(m-\frac{1}{2}\right) \frac{\pi}{\theta_j} &\quad &\text{for}\quad j\in J_0 \:\text{ and }\: k \in J_1 \quad \text{or} \quad j\in J_1 \:\text{ and }\: k\in  J_0,
\end{align}
where $\theta_j$ denotes the angle at the edge $M_j$. According to the definition of $\Lambda$ in \eqref{LambdaGris} we obtain the smallest eigenvalues in \eqref{lambdajkm} if we choose $\theta_j=2\pi$, which  in the first line for $m=1$ gives $\lambda_{j,1}=\frac{1}{2}$, in the second line for $m=2$ we get $ \lambda_{j,2}=\frac{1}{2}$ (here $m=1$ leads to the eigenvalue $\lambda_{j,1}=0$ which is not admissible in the set $\Lambda$), and in the last line $m=1$ yields  $\lambda_{j,1}=\frac{1}{4}$. Since we have mixed boundary condition, there exist at least one edge $M_j$ of the cone with two adjacent faces  such that one of them has Dirichlet and the other one  Neumann boundary conditions. Therefore, the eigenvalues in the third line of \eqref{lambdajkm} appear and we deduce  $\Lambda\geq \frac{1}{4}$. Thus, according to Proposition \ref{propGris} we have
\begin{align}
    \label{ureg}
  u_\text{reg} \in H^{s}(K) \quad \text{for}\quad s< \frac{5}{4}.
\end{align}
We now consider the singular part of $u$. According to \cite[Lem.~2.6.2, p.~75]{Gris} the eigenfunctions satisfy $\psi_l\in H^{s}(\Omega)$ for $s< \frac{5}{4}$, where $\Omega=K\cap S^2$   is the intersection of the cone with the unit sphere $S^2$. Since these eigenfunctions only depend on $\omega$ but not on  $\rho$,  we immediately get $\psi_l\in H^{s}(K) $ for $s< \frac{5}{4}$.  Moreover, we show that 
\begin{align}
    \label{u_sing}
    \rho^{\lambda}:=\rho^{   -1/2+ \sqrt{ \widetilde{\lambda}_l+1/4}} \in H^{3/2+\varepsilon}(K)
\end{align}
for small $\varepsilon>0$.  We have 
\begin{align}
\label{febr2}
    \| \rho^{\lambda}|H^s(K) \|^2 \sim \displaystyle\int_{0}^c \rho^{(\lambda-s)2} \rho^2 \: d\rho \sim  \Big[ \rho^{(\lambda-s)2+3}\Big]_0^c < \infty \quad \iff \quad s <\frac{3}{2}+\lambda.
\end{align}
Since $\widetilde{\lambda}_l\geq 0$,   we obtain  $\lambda=0$ when $\widetilde{\lambda}_l=0$ in which case  the function $\rho^{\lambda}=1$ is constant and trivially belongs to $H^s(K)$ for any $s$. Otherwise,  we have  $\lambda>0$  and according to \eqref{febr2} the statement \eqref{u_sing} holds for arbitrary small $\varepsilon>0$. Now  we apply the multiplier results from \cite[Thm.~1(i), Sect.~4.6.1, p.~177]{RuSi}, which state that 
\[
%H^{s_1}(\real^3)\cdot H^{s_2}(\real^3)\hookrightarrow 
\|uv|H^{s_1}(\real^3)\|\lesssim \|v|H^{s_1}(\real^3)\|\cdot \|u|H^{s_2}(\real^3)\|\quad \text{for} \quad s_2>s_1, \ s_1+s_2>0,\  s_2>\frac 32
\]
together with the extension operator for the $H^s$-spaces from  \cite[Thm.~1.2.10]{Gris} and obtain for
$u\in H^{s_2}(K)$ and $v\in H^{s_1}(K)$ that
\begin{align}
    \label{runstsickel}
     \| uv|H^{s_1}(K)\| & \lesssim
     \|\text{Ex}\: u \cdot \text{Ex}\: v  |H^{s_1}(\mathbb{R}^3)\| \notag\\
    & \lesssim  
     \|\text{Ex}\: u |H^{s_2}(\mathbb{R}^3)\| \cdot \| \text{Ex}\: v|H^{s_1}(\mathbb{R}^3)\| \notag\\
    & \lesssim \| u|H^{s_2}(K)\| \cdot \| v|H^{s_1}(K)\|,
\end{align}
where $\text{Ex}\: u,\: \text{Ex}\: v$ are the corresponding extensions from $K$ to $\mathbb{R}^3$ of $u$ and $v$.
We choose in  \eqref{runstsickel} the functions $u=\rho^{\lambda} $, $v=\psi_l $, and the regularity parameters $0<s_1<\frac{5}{4}$, $s_2=\frac{3}{2}+\varepsilon$ and obtain that
\begin{align}
    \label{using2}
    u_{\text{sing}}= \rho^{-\frac{1}{2}+ \sqrt{ \widetilde{\lambda}_l+\frac{1}{4}}} \psi_l \in H^{s}(K) \quad \text{for} \quad s<\frac{5}{4}.
\end{align}
In view of Proposition \ref{propGris}, by \eqref{ureg} and \eqref{using2} we obtain the desired regularity result, which  completes the proof.
\end{proof}

In order to allow inhomogeneous mixed boundary values  we invoke  the following result of Agranovich \cite[Thm.~4.2]{Ag11}.
\begin{proposition}
\label{agranovich}
Let $D\subset\mathbb{R}^3$ be a bounded Lipschitz domain and assume that the boundary $\partial D=\Gamma$ consists of  two domains $ \Gamma^0$ and $\Gamma^1$  and their common boundary $\partial \Gamma^j$, which is a  $1$-dimensional closed Lipschitz curve  without self intersections. Then there exists $\varepsilon>0$ such that for boundary data 
% Moreover, let 
$g_0\in H^{s-1/2}(\Gamma^0)$ and $g_1\in H^{s-3/2}(\Gamma^1)$ with $1-\varepsilon<s<1+\varepsilon$,  there exists  a unique  harmonic function $u$, i.e., $\Delta u =0$ in $D$, such that  
\[
u\in H^s(D) \qquad with \qquad u=g_0\quad \text{on}\quad \Gamma^0, \qquad \frac{\partial u}{\partial \nu}=g_1\quad \text{on}\quad \Gamma^1. 
\]
%Moreover, let the functions of the homogeneous mixed boundary value problem satisfy $g^0\in H^{s-1/2}(\Gamma^0)$ and $g^1\in H^{s-3/2}(\Gamma^1)$ for $1-\varepsilon<s<1+\varepsilon$ with sufficiently small $\varepsilon$ such that $g^0=g_j$ on $\Gamma_j$ for $j\in J_0$ and $g^1=g_j$ on $\Gamma_j$ for $j\in J_1$, where $\Gamma^0=\bigcup_{j\in J_0}\Gamma_j$ and $\Gamma^1=\bigcup_{j\in J_1}\Gamma_j$. Then the solution $u$ of \eqref{generalproblem} belongs to $u\in H^s(D)$.
\end{proposition}

\begin{remark}
Similar results were obtained in \cite[Thm., p.~4145]{mitrea} for so-called {'creased'} domains (here the authors additionally require  that the angle between the normals to $\Gamma^0$ and $\Gamma^1$ at the points of $\partial\Gamma_j$ are less than $\pi$). Unfortunately, the precise range of   $\varepsilon=\varepsilon(D,\Gamma^0,\Gamma^1)>0$ is not specified any further in \cite{Ag11,mitrea}. 
\end{remark}

Together Propositions \ref{agranovich}  and  \ref{H5/4theorie} allow us to reduce the inhomogeneous mixed problem 
\begin{equation}\label{inhom_mixed}
   -\Delta u=f \quad \text{on}\quad D, \qquad  u=g_0\quad \text{on}\quad \Gamma^0, \qquad \frac{\partial u}{\partial \nu}=g_1\quad \text{on}\quad \Gamma^1 
\end{equation}

to the homogeneous one as explained before for pure Dirichlet and Neumann problems. In total we obtain the following result: 

\begin{proposition}[\bf Fractional Sobolev regularity for mixed problems]
\label{prop-frac-mixed}
Let $D\subset\mathbb{R}^3$ be a bounded Lipschitz domain and assume that the boundary $\partial D=\Gamma$ consists of  two domains $ \Gamma^0$ and $\Gamma^1$  and their common boundary $\partial \Gamma^j$, which is a  $1$-dimensional closed Lipschitz curve  without self intersections. %Then there exists $\varepsilon>0$ such that for 
 Moreover, let $\varepsilon>0$ be small and consider boundary data 
$g_0\in H^{s-1/2}(\Gamma^0)$ and $g_1\in H^{s-3/2}(\Gamma^1)$ with   $1-\varepsilon<s<1+\varepsilon$. Let  $\alpha_0>5/4$ be the constant from Proposition \ref{propGris}. Then   there exists  a unique  solution $u\in H^{\overline{\alpha}}(D)$ of \eqref{inhom_mixed} for every  $\overline{\alpha}<\alpha_0$ in case of homogeneous boundary conditions ($g_0=g_1=0$) and $\overline{\alpha}=s$ in case of inhomogeneous boundary conditions.
\end{proposition}

%\textcolor{red}{ $\Gamma^0$ und $\Gamma^1$ dürfen nicht nicht-zusammenhängend sein. introduction Agran: the boundary $\Gamma$ is divided into two domains $\Gamma_1,\: \Gamma_2$ by a closed $n-1$ dimensional Lipschitz surface $\partial \Gamma_j$ without self-intersections; the Dirichlet condition is posed on $\Gamma_1$, and the Neumann condition is posed on $\Gamma_2$.
%Wir bräuchten die Annahme: o.B.d.A.: $\Gamma^0=\Gamma_1\cup...\cup \Gamma_k,\: \Gamma^1=\Gamma_{k+1}\cup...\cup \Gamma_n$ für ein $k\in \{1,...,n   \}$.
%}

%\textcolor{red}{Mitrea, Sect. 7 : D, N are admissible patches. Def. von "admissible patch" ist auf Seite 5, aber ist vielleicht zu kompliziert für unseren Paper;
%keine Info über $\varepsilon$, nur dass es von $\Omega$ hängt... Aber auf dem Diagram mit Hexagonal domain ist 0<s<1 mit beliebigem $\varepsilon\in (0,1/2]$ für $p=1/2$. }

\paragraph{Besov regularity of the mixed problem}

Before we come to the regularity of the solution of  \eqref{generalproblem} in Besov spaces,
we first deduce the following regularity assertion in weighted Sobolev spaces  as a consequence of  Maz'ya and Rossmann \cite[Lem.~4.4, Thm.~4.4]{MR02}.

\begin{theorem}[\bf Weighted Sobolev regularity]
\label{mixedregularity_neu}
 Let $K$ be a polyhedral cone and let $J_0,\: J_1$ be non-empty subsets of $J=\{1,2,...,n   \}$, $l\in\mathbb{N}$, $l\geq 2$, $\beta\in \mathbb{R}$, $\delta=(\delta_1,...,\delta_n)\in \mathbb{R}^n$. Moreover, let $f\in  \mathcal{W}_{\beta,\delta}^{l-2,2}(K,\tilde{J})$, $g_j\in \mathcal{W}_{\beta,\delta}^{l-3/2,2}(\Gamma_j;\tilde{J}) $  for $j\in J_1$ and $g_j\in \mathcal{W}_{\beta,\delta}^{l-1/2,2}(\Gamma_j;\tilde{J}) $ for $j\in J_0$. Suppose that the line $\text{Re}\: \lambda=l-\beta-3/2$ does not contain eigenvalues of the pencil $\mathfrak{A}$ and that the components of $\delta$ satisfy the inequalities \begin{align}
\label{dec11b_alt}
    l-\delta_{+}^{(j)} <\delta_j+1<l &\quad\text{for}\quad j\in \tilde{J}, \notag\\
    \max\left(l-\delta_{+}^{(j)},l-2\right)<\delta_j+1<l &\quad\text{for}\quad j\in J\setminus\tilde{J},
\end{align}
where $\theta_j$ denotes the angle at the edge $M_j$. 
Then the problem \eqref{generalproblem} has a unique solution $u\in\mathcal{W}_{\beta,\delta}^{l,2}(K,\tilde{J})$.
\end{theorem}

\begin{proof} We apply 
%Applying the embeddings \eqref{sept23} and 
the following elementary embedding from \cite[p.~439]{MR02}, i.e., 
\begin{equation}\label{emb-w-spaces}
     \mathcal{W}_{\beta+1,\delta}^{l+1,2}(K,\tilde{J})\hookrightarrow \mathcal{W}_{\beta,\delta'}^{l,2}(K,\tilde{J})
\end{equation}
with $\delta_j-\delta_j'\leq 1$ for $j=1,...,n$ and $\delta_j,\:\delta_j'>-1$ for $j\in J\setminus \tilde{J}$, which together with the embedding \eqref{sept23} yields %we obtain 
\begin{align}
    \label{embeddingfor_f}
    f\in \mathcal{W}_{\beta,\delta}^{l-2,2}(K,\tilde{J}) \hookrightarrow \mathcal{W}_{\beta-(l-2),\delta-(l-2)}^{0,2}(K,\tilde{J})\hookrightarrow W_{\beta-(l-2),\delta'}^{0,2}(K),
\end{align}
where $\delta_j':=\delta_j-(l-2)$ and 
\begin{align}
    \label{assumptionfordelta'_neu}
        2-\frac{\pi}{\theta_j}<\delta_j'+1<2 &\quad\text{for}\quad j\in \tilde{J}, \notag\\
    \max\left(2-\frac{\pi}{\theta_j},0\right)<\delta_j'+1 <2 &\quad\text{for}\quad j\in J\setminus\tilde{J}. 
\end{align}
Another application of \eqref{embeddingfor_f} to the boundary values gives 
\begin{align}
    \label{functiongj_neu}
   g_j\in \mathcal{W}_{\beta,\delta}^{l-1/2,2}(\Gamma_j;\tilde{J}) &\hookrightarrow \mathcal{W}_{\beta-(l-2),\delta'}^{2-1/2,2}(\Gamma_j;\tilde{J})\qquad \text{for}   \quad  j\in J_0,\notag\\
   g_j\in \mathcal{W}_{\beta,\delta}^{l-1-1/2,2}(\Gamma_j;\tilde{J})&\hookrightarrow \mathcal{W}_{\beta-(l-2),\delta'}^{1-1/2,2}(\Gamma_j;\tilde{J}) \qquad \text{for} \quad  j\in J_1.
\end{align}
Moreover, from the initial assumptions we see that the line 
$$\text{Re} \: \lambda=2-(\beta-(l-2))-3/2=l-\beta-3/2$$ does not contain eigenvalues of the pencil $\mathfrak{A}$. Hence, according to \eqref{embeddingfor_f}, \eqref{assumptionfordelta'_neu} and \eqref{functiongj_neu} we can apply \cite[Thm.~4.4]{MR02} and obtain that there is a unique solution $u\in \mathcal{W}_{\beta-(l-2),\delta'}^{2,2}(K,\tilde{J}) $ of \eqref{generalproblem}. Furthermore, with this also the assumptions of \cite[Lem.~4.4]{MR02} are satisfied for $k=2$, and we deduce that the solution $u$ in fact belongs to $\mathcal{W}_{\beta,\delta}^{l,2}(K,\tilde{J})$.
\end{proof}

\begin{remark} Note that the lower bound for $\delta_j$ when $j\in J\setminus \tilde{J}$ in \eqref{dec11b_alt} can probably be replaced by 
\[
\max\left(l-\delta_{+}^{(j)},0\right)<\delta_j+1 
\]
if we combine \cite[Thms.~5.4,~5.5]{MR02} and work with right-hand sides $f$ in \eqref{generalproblem} being functionals from the dual space  $V^{-1,2}_{\beta,\delta}(K)$ of $V^{1,2}_{-\beta,-\delta}(K)$. However, in order to keep our presentation as simple as possible, we decided to stick with the additional assumption $l-2<\delta_j+1$ above for the time being. 
\end{remark}

Finally, combining the regularity results from Proposition \ref{prop-frac-mixed} and  Theorem \ref{mixedregularity_neu}   we obtain the following regularity result in Besov spaces for the mixed boundary value problem.

\begin{theorem}[\bf Besov regularity for mixed boundary value problems]
\label{pdeforWspacemixed}
Let $K$ be a polyhedral cone and $\tilde{K}$ its truncation.      % and let $\varphi$ be the cut-off function of \eqref{cutoff}. 
For non-empty subsets $J_0$ and $J_1$  of $J=\{1,2,...,n   \}$ with $J_0\cup J_1=J$ and $J_0\cap J_1=\emptyset$  put  
$\Gamma^i:=\bigcup_{j\in J_i}\Gamma_j$ \text{and} $\tilde{\Gamma}^i:=\Gamma^i\cap \tilde{K}$ for $i=0,1$.  
%$$\Gamma^i:=\bigcup_{j\in J_i}\Gamma_j\qquad \text{and}\qquad  \tilde{\Gamma}^i:=\Gamma^i\cap \tilde{K}, \quad i=0,1,  
%\Gamma^1:=\bigcup_{j\in J_1}\Gamma_j
%$$ 
Additionally,   assume that the intersection $\overline{\Gamma^0}\cap \overline{\Gamma^1}$ consists of two edges of $K$ only. \\  %\textcolor{red}{Moreover,  $\tilde{\Gamma}^j:=\Gamma^j\cap \tilde{K}$, $j=0,1$. } 
Let  $l\in\mathbb{N}$, $l\geq 2$, $\beta\in \mathbb{R}$, $l>\beta$,  $\delta=(\delta_1,...,\delta_n)\in \mathbb{R}^n$,  
%with
% \textcolor{red}{ $\delta_j>0$ for all $j=1,...,d$ /OR/
%$\delta_j>-1$ for $j\in J\setminus\tilde{J}$ (depending on Theorem 3.1)}.
and assume  the right-hand side of \eqref{generalproblem} satisfies $f\in \mathcal{W}_{\beta,\delta}^{l-2,2}(K)\cap L_2(\tilde{K})$. For $\varepsilon>0$ small and $1-\varepsilon<s<1+\varepsilon$, let the boundary data satisfy 
$g^1\in H^{s-3/2}(\tilde{\Gamma}^1)$ with $g^1=g_j$ on $\Gamma_j$ and $g_j\in \mathcal{W}_{\beta,\delta}^{l-3/2,2}(\Gamma_j;\tilde{J}) $ for $j\in J_1$ and $g^0\in H^{s-1/2}(\tilde{\Gamma}^0)$ with
$g^0=g_j$ on $\Gamma_j$ and
$g_j\in \mathcal{W}_{\beta,\delta}^{l-1/2,2}(\Gamma_j;\tilde{J}) $ for $j\in J_0$. 
% Let $\alpha_0=\alpha_0(K)$ be the number defined in Proposition \ref{H5/4theorie}.
Further assume that the line $\text{Re}\: \lambda=l-\beta-3/2$ does not contain eigenvalues of the pencil $\mathfrak{A}$
and the components of $\delta$ satisfy \eqref{dec11b_alt}.
Then for the solution $u$ of \eqref{generalproblem} we have
\begin{align*}
    u\in B_{\tau,\tau}^r(\tilde{K}), \qquad r<\min\Big\{{l},{3(l-|\delta|)}, 3\overline{\alpha}  \Big\}, \quad  \frac{1}{\tau}=\frac{r}{3}+\frac{1}{2}, 
    %\quad \text{for} \quad \textcolor{blue}{r<\min\{  l,3(l-|\delta|),3\alpha_0 \} }.
\end{align*}
where $\overline{\alpha}(K)=\overline{\alpha}$ is the number from Proposition \ref{prop-frac-mixed}.
%where 
%$$ r<\min\{{l},{3(l-|\delta|)}, {3\alpha_0}\} .$$ 
\end{theorem}

\begin{proof}
The proof follows immediately from  Proposition  \ref{prop-frac-mixed}  and  Theorem \ref{mixedregularity_neu} combined with Theorem \ref{embeddingBesovWspaces}.
\end{proof}

\paragraph{Acknowledgement.} The authors thank Stephan Dahlke and Winfried Sickel for helpful discussions  and  pointing out some references on the subject. 

%\newpage

\end{document}

%% file: infiniteconefebr2021.pdf_tex
%% Creator: Inkscape inkscape 0.92.1, www.inkscape.org
%% Accompanies image file 'infiniteconefebr2021.pdf' (pdf, eps, ps)
\begingroup%
  \makeatletter%
  \providecommand\color[2][]{%
    \errmessage{(Inkscape) Color is used for the text in Inkscape, but the package 'color.sty' is not loaded}%
    \renewcommand\color[2][]{}%
  }%
  \providecommand\transparent[1]{%
    \errmessage{(Inkscape) Transparency is used (non-zero) for the text in Inkscape, but the package 'transparent.sty' is not loaded}%
    \renewcommand\transparent[1]{}%
  }%
  \providecommand\rotatebox[2]{#2}%
  \ifx\svgwidth\undefined%
    \setlength{\unitlength}{235.91504168bp}%
    \ifx\svgscale\undefined%
      \relax%
    \else%
      \setlength{\unitlength}{\unitlength * \real{\svgscale}}%
    \fi%
  \else%
    \setlength{\unitlength}{\svgwidth}%
  \fi%
  \global\let\svgwidth\undefined%
  \global\let\svgscale\undefined%
  \makeatother%
  \begin{picture}(1,0.74137579)%
    \put(0,0){\includegraphics[width=\unitlength,page=1]{infiniteconefebr2021.pdf}}%
    \put(0.75714161,0.3364906){\color[rgb]{0,0,0}\makebox(0,0)[lb]{\smash{$M_j$}}}%
    \put(0.46193832,0.47500909){\color[rgb]{0,0,0}\makebox(0,0)[lb]{\smash{$\Omega$}}}%
  \end{picture}%
\endgroup%

%% file: boundedconefebr2021.pdf_tex
%% Creator: Inkscape inkscape 0.92.1, www.inkscape.org
%% Accompanies image file 'boundedconefebr2021.pdf' (pdf, eps, ps)
\begingroup%
  \makeatletter%
  \providecommand\color[2][]{%
    \errmessage{(Inkscape) Color is used for the text in Inkscape, but the package 'color.sty' is not loaded}%
    \renewcommand\color[2][]{}%
  }%
  \providecommand\transparent[1]{%
    \errmessage{(Inkscape) Transparency is used (non-zero) for the text in Inkscape, but the package 'transparent.sty' is not loaded}%
    \renewcommand\transparent[1]{}%
  }%
  \providecommand\rotatebox[2]{#2}%
  \ifx\svgwidth\undefined%
    \setlength{\unitlength}{262.15738311bp}%
    \ifx\svgscale\undefined%
      \relax%
    \else%
      \setlength{\unitlength}{\unitlength * \real{\svgscale}}%
    \fi%
  \else%
    \setlength{\unitlength}{\svgwidth}%
  \fi%
  \global\let\svgwidth\undefined%
  \global\let\svgscale\undefined%
  \makeatother%
  \begin{picture}(1,0.92951688)%
    \put(0,0){\includegraphics[width=\unitlength,page=1]{boundedconefebr2021.pdf}}%
    \put(0.87297409,0.39648558){\color[rgb]{0,0,0}\makebox(0,0)[lb]{\smash{$M_j$}}}%
    \put(0.69927816,0.21461532){\color[rgb]{0,0,0}\makebox(0,0)[lb]{\smash{$\theta_j$}}}%
    \put(0.53579946,0.62126866){\color[rgb]{0,0,0}\makebox(0,0)[lb]{\smash{$x$}}}%
    \put(0.40910356,0.37809402){\color[rgb]{0,0,0}\makebox(0,0)[lb]{\smash{$\rho(x)$}}}%
    \put(0.66862588,0.49457277){\color[rgb]{0,0,0}\makebox(0,0)[lb]{\smash{$r_j(x)$}}}%
    \put(0.41932088,0.78474719){\color[rgb]{0,0,0}\makebox(0,0)[lb]{\smash{$\Omega$}}}%
    \put(0,0){\includegraphics[width=\unitlength,page=2]{boundedconefebr2021.pdf}}%
  \end{picture}%
\endgroup%

%% file: coneK.pdf_tex
%% Creator: Inkscape inkscape 0.92.1, www.inkscape.org
%% Accompanies image file 'coneK.pdf' (pdf, eps, ps)
\begingroup%
  \makeatletter%
  \providecommand\color[2][]{%
    \errmessage{(Inkscape) Color is used for the text in Inkscape, but the package 'color.sty' is not loaded}%
    \renewcommand\color[2][]{}%
  }%
  \providecommand\transparent[1]{%
    \errmessage{(Inkscape) Transparency is used (non-zero) for the text in Inkscape, but the package 'transparent.sty' is not loaded}%
    \renewcommand\transparent[1]{}%
  }%
  \providecommand\rotatebox[2]{#2}%
  \ifx\svgwidth\undefined%
    \setlength{\unitlength}{236.17363065bp}%
    \ifx\svgscale\undefined%
      \relax%
    \else%
      \setlength{\unitlength}{\unitlength * \real{\svgscale}}%
    \fi%
  \else%
    \setlength{\unitlength}{\svgwidth}%
  \fi%
  \global\let\svgwidth\undefined%
  \global\let\svgscale\undefined%
  \makeatother%
  \begin{picture}(1,1.80487426)%
    \put(0,0){\includegraphics[width=\unitlength,page=1]{coneK.pdf}}%
    \put(0.57866757,0.96250483){\color[rgb]{0,0,0}\makebox(0,0)[lb]{\smash{L}}}%
    \put(0.61385496,0.75476674){\color[rgb]{0,0,0}\makebox(0,0)[lb]{\smash{$E_1$}}}%
    \put(0.56408582,1.11005063){\color[rgb]{0,0,0}\makebox(0,0)[lb]{\smash{$E_2$}}}%
    \put(0.24680128,1.07919472){\color[rgb]{0,0,0}\makebox(0,0)[lb]{\smash{$E_3$}}}%
    \put(0.0400745,0.80255594){\color[rgb]{0,0,0}\makebox(0,0)[lb]{\smash{$E_4$}}}%
    \put(0.33756131,0.81542641){\color[rgb]{0,0,0}\makebox(0,0)[lb]{\smash{$E_5$}}}%
    \put(0.744855,1.65463326){\color[rgb]{0,0,0}\makebox(0,0)[lb]{\smash{$M_2$}}}%
    \put(0.78113163,1.14850107){\color[rgb]{0,0,0}\makebox(0,0)[lb]{\smash{$M_1$}}}%
    \put(0.42956703,1.53852128){\color[rgb]{0,0,0}\makebox(0,0)[lb]{\smash{$M_3$}}}%
    \put(-0.00384549,1.226568){\color[rgb]{0,0,0}\makebox(0,0)[lb]{\smash{$M_4$}}}%
    \put(0.40737895,1.22467743){\color[rgb]{0,0,0}\makebox(0,0)[lb]{\smash{$M_5$}}}%
    \put(0.43741736,0.35423236){\color[rgb]{0,0,0}\makebox(0,0)[lb]{\smash{$\rho_I\sim k2^{-j}$}}}%
    \put(0.19500608,0.00089728){\color[rgb]{0,0,0}\makebox(0,0)[lb]{\smash{0}}}%
    \put(0,0){\includegraphics[width=\unitlength,page=2]{coneK.pdf}}%
  \end{picture}%
\endgroup%

%% file: layerL.pdf_tex
%% Creator: Inkscape inkscape 0.92.1, www.inkscape.org
%% Accompanies image file 'layerL.pdf' (pdf, eps, ps)
\begingroup%
  \makeatletter%
  \providecommand\color[2][]{%
    \errmessage{(Inkscape) Color is used for the text in Inkscape, but the package 'color.sty' is not loaded}%
    \renewcommand\color[2][]{}%
  }%
  \providecommand\transparent[1]{%
    \errmessage{(Inkscape) Transparency is used (non-zero) for the text in Inkscape, but the package 'transparent.sty' is not loaded}%
    \renewcommand\transparent[1]{}%
  }%
  \providecommand\rotatebox[2]{#2}%
  \ifx\svgwidth\undefined%
    \setlength{\unitlength}{388.91064402bp}%
    \ifx\svgscale\undefined%
      \relax%
    \else%
      \setlength{\unitlength}{\unitlength * \real{\svgscale}}%
    \fi%
  \else%
    \setlength{\unitlength}{\svgwidth}%
  \fi%
  \global\let\svgwidth\undefined%
  \global\let\svgscale\undefined%
  \makeatother%
  \begin{picture}(1,0.52884968)%
    \put(0,0){\includegraphics[width=\unitlength,page=1]{layerL.pdf}}%
    \put(0.61675302,0.00634535){\color[rgb]{0,0,0}\makebox(0,0)[lb]{\smash{$E_1$}}}%
    \put(0.81740134,0.5011481){\color[rgb]{0,0,0}\makebox(0,0)[lb]{\smash{$E_2$}}}%
    \put(1.05895868,0.48166754){\color[rgb]{0,0,0}\makebox(0,0)[lb]{\smash{}}}%
    \put(0.2290926,0.49725186){\color[rgb]{0,0,0}\makebox(0,0)[lb]{\smash{$E_3$}}}%
    \put(-0.00272446,0.03361785){\color[rgb]{0,0,0}\makebox(0,0)[lb]{\smash{$E_4$}}}%
    \put(0.37519573,0.10959147){\color[rgb]{0,0,0}\makebox(0,0)[lb]{\smash{$E_5$}}}%
    \put(0.74142771,0.14076021){\color[rgb]{0,0,0}\makebox(0,0)[lb]{\smash{$m2^{-j}$}}}%
    \put(0.75116794,0.23426622){\color[rgb]{0,0,0}\makebox(0,0)[lb]{\smash{$(m+1)2^{-j}$}}}%
    \put(1.08623118,-0.04040782){\color[rgb]{0,0,0}\makebox(0,0)[lb]{\smash{}}}%
    \put(0,0){\includegraphics[width=\unitlength,page=2]{layerL.pdf}}%
    \put(0.61870109,0.17777305){\color[rgb]{0,0,0}\makebox(0,0)[lb]{\smash{S}}}%
  \end{picture}%
\endgroup%